\documentclass[12pt]{amsart}
\usepackage{amssymb, a4wide, url, hyperref, color}
\usepackage{amsxtra}

\newcommand{\C}{\mathbb{C}}
\newcommand{\Z}{\mathbb{Z}}
\newcommand{\Q}{\mathbb{Q}}
\newcommand{\R}{\mathbb{R}}
\newcommand{\N}{\mathbb{N}}

\newcommand{\A}{\mathbb{A}}

\newcommand{\W}{\mathcal{W}}
\newcommand{\F}{\mathcal{F}}
\newcommand{\E}{\mathcal{E}}

\newcommand{\Aut}{\mathcal{A}}

\newcommand{\reg}{{\operatorname{reg}}}

\renewcommand{\Re}{\operatorname{Re}}

\newcommand{\Ind}{\operatorname{Ind}}

\newcommand{\iii}{\mathrm{i}}

\newcommand{\Hom}{\operatorname{Hom}}
\renewcommand{\subset}{\subseteq}

\newcommand{\bs}{\backslash}
\newcommand{\Comp}{\mathcal{C}}

\newcommand{\Data}{\mathcal{D}}

\newcommand{\diag}{\operatorname{diag}}

\newcommand{\GL}{{\operatorname{GL}}}
\newcommand{\Sp}{{\operatorname{Sp}}}
\newcommand{\SL}{\operatorname{SL}}

\newcommand{\abs}[1]{\left|{#1}\right|}
\newcommand{\aaa}{\mathfrak{a}}

\newcommand{\Res}{\operatorname{Res}}

\newcommand{\Ad}{\operatorname{Ad}}

\newcommand{\inj}{\iota}

\newcommand{\sprod}[2]{\left\langle#1,#2\right\rangle}

\newcommand{\set}{\mathfrak{c}}

\newcommand{\triv}{{\mathbf{1}}}
\renewcommand{\P}{{\mathcal{P}}}
\newcommand{\std}{{\operatorname{std}}}
\newcommand{\sstd}{{\operatorname{ss}}}

\newcommand{\Tau}{{\mathcal{T}}}

\newcommand{\HC}{\mathbf{H}}
\newcommand{\jj}{\mathrm{j}}
\newcommand{\LL}{\mathcal{L}}
\newcommand{\tbl}{\mathsf{t}}
\newcommand{\tbls}{\mathsf{s}}
\newcommand{\wperm}{\kappa}

\newcommand{\itlim}[1]{\mathop{\operatorname{it-lim}}\limits_{{#1}}}

\newcommand{\sm}[4]{{\bigl(\begin{smallmatrix}{#1}&{#2}\\{#3}&{#4}
\end{smallmatrix}\bigr)}}

\newtheorem{theorem}{Theorem}[section]
\newtheorem{lemma}[theorem]{Lemma}
\newtheorem{proposition}[theorem]{Proposition}
\newtheorem{remark}[theorem]{Remark}

\newtheorem{corollary}[theorem]{Corollary}

\title[On residual automorphic representations]{On residual automorphic representations and period integrals for symplectic groups}

\author{Solomon Friedberg}
\address{Department of Mathematics, Boston College, Chestnut Hill, MA 02467}
\email{solomon.friedberg@bc.edu}

\author{David Ginzburg}
\address{Department of Mathematics, Tel Aviv University, Tel-Aviv, Israel}
\email{ginzburg@tauex.tau.ac.il}

\author{Omer Offen}
\address{Department of Mathematics, Brandeis University, Waltham, MA 02453}
\email{offen@brandeis.edu}

\date{\today}

\thanks{This work was supported by the BSF, grant number 2020011 (all authors), by the NSF, grant numbers DMS-2401309 (Friedberg) and DMS-2401308 (Offen), and by a Simons Collaboration Grant, award number 709678 (Offen).}

\begin{document}
\begin{abstract} 
We construct new irreducible components in the discrete automorphic spectrum of symplectic groups. The construction lifts a cuspidal automorphic representation of $\GL_{2n}$ with a linear period to an irreducible component of the residual spectrum of the rank $k$ symplectic group $\Sp_k$ for any $k\ge 2n$. We show that this residual representation admits a non-zero $\Sp_n\times \Sp_{k-n}$-invariant linear form. This generalizes a construction of Ginzburg, Rallis and Soudry, the case $k=2n$, that arises in the descent method.
\end{abstract} 

\maketitle

\section{Introduction}

Let $G$ be a reductive group defined over a number field $F$. The construction of the automorphic spectrum of $G(\A_F)$ is a fundamental problem.  If $H<G$ is a semisimple subgroup, 
Burger, Li and Sarnak \cite{MR1118700}, and Burger and Sarnak \cite{MR1123369}, explain that (in a precise sense) $H$ should `contribute' to the automorphic spectrum of $G(\A_F)$.  
This contribution is also of central interest in the relative Langlands program.

The automorphic spectrum is also
governed by functoriality, and it is an important problem to characterize the image of a (conjectured or known) functorial transfer.  If $G'$ is a reductive algebraic group over $F$
such that there is a functorial transfer of automorphic representations from $G'(\A_F)$ to $G(\A_F)$, then one may seek a subgroup $H$ such that the period integral
\[
\P_H(\phi)=\int_{H(F)\bs \left(H(\A_F)\cap G(\A_F)^1\right)}\phi(h)\, dh
\]
detects the transfer:  $\pi$ on $G(\A_F)$ is in the image of the transfer if and only if there exists an automorphic form $\phi$ in the space
of $\pi$ such that the period integral $\P_H(\phi)$ converges and is nonzero.
Here $G(\A_F)^1$ is the kernel of the Harish-Chandra height function on $G(\A_F)$.  If this condition on $\P_H$ is satisfied we say that $\pi$ is {\sl $H$-distinguished}.

One of the main goals of the relative Langlands program is to understand the relation between $H$-distinguished representations, special values of $L$-functions and the image of a functorial transfer from automorphic representations of a group $G'(\A_F)$ to  automorphic representations of $G(\A_F)$ where the group $G'$ is prescribed by the pair $(G,H)$.
This is explained in  \cite{MR3764130} for certain spherical pairs $(G,H)$ and more recently in \cite{BZ-S-V} in a more general framework.

Though it is natural to study $H$-distinguished representations, a broader question is also natural.
Indeed, when $\pi$ is $H$-distinguished the period map $\P_H$ gives a functional on the space of $\pi$ that is $H(\A_F)$-invariant:
for $h\in H(\A_F)$, let $\phi_h(g)=\phi(gh)$; then $\P_H(\phi)=\P_H(\phi_h)$ for all $h\in H(\A_F)$.   So a natural generalization is to study automorphic
representations of $G(\A_F)$ that admit a nontrivial $H(\A_F)$-invariant functional. We give new examples of such representations in this paper.

This extra breadth is important in the following context.
If $H$ is reductive then the period $\P_H$ converges on all cusp forms on $G(\A_F)$ by a result of Ash, Ginzburg and Rallis \cite{MR1233493}.
However, focusing on the cuspidal spectrum is incomplete.  For example, Ash, Ginzburg and Rallis also provide many families of \emph{vanishing pairs}, 
that is, $(G,H)$ where $\P_H$ vanishes identically on all cuspidal automorphic forms on $G(\A_F)$.   
As the next candidates, one could ask if there are any residual automorphic representations,
that is, non-cuspidal discrete irreducible components of $L^2(G(F)\bs G(\A_F))$, that are $H$-distinguished. But for such representations, $\P_H$ may not converge.
So we ask, more broadly, for any residual automorphic representations admitting
a nontrivial $H(\A_F)$-invariant linear functional.
 This leads to the framework of this work.  

In this work we construct certain new residual automorphic representations on symplectic groups as liftings from even general linear groups. We also give a subgroup $H$ that detects this lifting,
in the sense of possessing such a nontrivial $H(\A_F)$-invariant functional. The functionals arise from regularizations of the period integral, a process with many subtleties.
Our construction generalizes one of Ginzburg, Rallis and Soudry \cite{MR1740991}, who handle a smaller case for which the appropriate period integral converges.
We next explain their results. 

We denote by $\Sp_k$ the rank $k$ symplectic group defined over $F$. It is the isometry group of a $2k$-dimensional symplectic form.  For a composition $k=a+b$ with $a,b\in \Z_{\ge 0}$ we consider $\Sp_a\times \Sp_b$ as a subgroup of $\Sp_k$ by decomposing the symplectic form as a direct sum of two symplectic forms of corresponding dimensions $2a$ and $2b$. It is proved in \cite{MR1233493} that $(\Sp_k,\Sp_a\times \Sp_b)$ is a vanishing pair.

Let $n\in \N$ and $G=\Sp_{2n}$.
Let $\pi$ be an irreducible cuspidal automorphic representation of $\GL_{2n}(\A_F)$ such that the exterior square $L$-function $L(\pi,\wedge^2,s)$ has a pole at $s=1$ and the standard $L$-function satisfies $L(\pi,\frac12)\ne 0$ or equivalently, such that $\pi$ is $(\GL_n\times \GL_n)$-distinguished. Consider the Eisenstein series $E^0(\varphi,s)$ defined by the meromorphic continuation to $s\in \C$ of the series
\[
E^0(g,\varphi,s)=\sum_{\gamma\in P(F)\bs G(F)}\varphi_s(\gamma g),\ \ \ g\in G(\A_F)
\]
for a holomorphic flat section $\varphi_s\in \Ind_{P(\A_F)}^{G(\A_F)}(\pi\abs{\det}^s)$ where $P$ is the Siegel parabolic subgroup of $G$ and $\pi$ is considered as a representation of $P(\A_F)/U_P(\A_F)\simeq \GL_{2n}(\A_F)$ where $U_P$ is the unipotent radical of $P$.

By  \cite[Proposition 1]{MR1740991} the Eisenstein series $E^0(\varphi,s)$ has a simple pole at $s=\frac12$ and by \cite[Theorem 2.1]{MR2848523} the space $\E^0(\pi)$ spanned by the residues 
\[
\E^0(\varphi)=\lim_{s\to \frac12}(s-\frac12)E^0(\varphi,s)
\]
is an irreducible residual automorphic representation of $G(\A_F)$.
Furthermore $\E^0(\pi)$ is $\Sp_n\times \Sp_n$-distinguished, that is, for $H=\Sp_n\times \Sp_n$ the period integral $\P_H$ converges on $\E^0(\pi)$ and is not identically zero \cite[Theorem D]{MR1740991}. 

Applying an integrability criterion of Zydor \cite[Theorem 4.5]{MR4411860}, it is easy to see that the period integral $\P_{\Sp_a\times \Sp_b}$ converges on $\E^0(\pi)$ for any composition $2n=a+b$ with $a,b\in \Z_{\ge 0}$. If $a\ne b$ then $\P_{\Sp_a\times \Sp_b}$ is identically zero on $\E^0(\pi)$. This is based on the fact that $(\GL_{2n},\GL_a \times \GL_b)$ is a vanishing pair for $a\ne b$ \cite{MR1233493}, and follows from Lemma \ref{lem van per sp}. 
In this sense, amongst the period integrals over $\Sp_a\times \Sp_b$, the case $a=b$ detects the transfer $\pi \mapsto \E^0(\pi)$ from $\GL_{2n}$ to $\Sp_{2n}$.

In this work we generalize this construction as follows. Let $m\in \N$. Our first result is the construction of a transfer $\pi\to \E^m(\pi)$ from $\GL_{2n}$ to $\Sp_{m+2n}$ from representations $\pi$ as in the Ginzburg--Rallis--Soudry construction to irreducible, residual automorphic representations $\E^m(\pi)$ of $\Sp_{m+2n}(\A_F)$.

Let $G=\Sp_{m+2n}$ and $P$ a maximal parabolic subgroup of $G$ with Levi component isomorphic to $\GL_m \times \Sp_{2n}$. Let $\abs{\det}^s$ be the adelic absolute value on $\GL_m(\A_F)$ and consider $\abs{\det}^s \otimes \E^0(\pi)$ as a representation of $P(\A_F)/U_P(\A_F)\simeq \GL_m(\A_F)\times \Sp_{2n}(\A_F)$. Denote by $E^m(\varphi,s)$ the Eisenstein series induced from a holomorphic flat section in 
\[
I(\pi,s)=\Ind_{P(\A_F)}^{G(\A_F)}(\abs{\det}^s \otimes \E^0(\pi)).
\]
The following is a consequence of the results of \S\ref{sec sq int} after realizing the trivial representation of $\GL_m(\A_F)$ as a multi-residue of the unramified Eisenstein series induced from the principal series.
\begin{theorem}\label{thm res intro}
\begin{enumerate}
\item 
The Eisenstein series $E^m(\varphi,s)$ has a simple pole at $s=\frac{m+1}2$. 

\item Denote by $\E^m(\pi)$ the space spanned by the residues
\[
\E^m(\varphi)=\lim_{s\to \frac{m+1}2}(s-\frac{m+1}2)E^m(\varphi,s).
\]
The space $\E^m(\pi)$ is an irreducible, residual automorphic representation of $G(\A_F)$.
\end{enumerate}
\end{theorem}

The main task of this work is to make precise the statement that the subgroup 
\[
H=\Sp_{m+n}\times \Sp_n
\] 
detects the transfer $\pi\mapsto \E^m(\pi)$.
A first obstacle, observed in Lemma \ref{lem not reg},  is that the period integral $\P_H$ does not converge on $\E^m(\pi)$. Before reformulating the problem we go back to the case $m=0$ and explain the tools used by Ginzburg, Rallis and Soudry.

The proof of non-vanishing of $\P_{\Sp_n\times \Sp_n}|_{\E^0(\pi)}$ is based on the residue method introduced in \cite{MR1142486}. Jacquet and Rallis proved that the two-block Speh representation of $\GL_{2n}(\A_F)$ associated to an irreducible, cuspidal automorphic representation of $\GL_n(\A_F)$ is $\Sp_n$-distinguished. 
The two main steps in the method, carried out in both constructions \cite{MR1142486} and \cite{MR1740991}, are as follows.
\begin{itemize}
\item They realized the residual representation as a residue of an Eisenstein series and obtained a formula for the period integral of the truncated Eisenstein series using Arthur's truncation operator. 
\item They deduced a formula for the convergent period integral on the residual representation by taking the residue through the period integral and through the truncation operator.
\end{itemize}
Carrying out the first step requires an unfolding process along double cosets $P\bs G/H$ and the proof of convergence of the contribution of each double coset. Convergence of each term becomes quite technical and for many terms it is only used in order to claim that the corresponding contribution is zero. 

In the context of Galois symmetric pairs, period integrals of truncated automorphic forms have been studied more systematically in \cite{MR1625060} and \cite{MR2010737}.
Jacquet, Lapid and Rogawski modified Arthur's truncation operator and defined a mixed truncation operator that carries automorphic forms on $G(\A_F)$ to automorphic functions of rapid decay on $H(\A_F)$. The mixed truncation operator is better suited for the study of period integrals. 

In \cite{MR4411860}, Zydor successfully generalized these ideas, defined the mixed truncation relative to any reductive subgroup $H$ of $G$ and 
established its basic properties. The upshot is that for an automorphic form $\phi$ of $G(\A_F)$, applying Zydor's mixed truncation $\Lambda^T\phi$ relative to $H$ yields an $H(F)\bs H(\A_F)\cap G(\A_F)^1$-integrable function and the period integral $\P_H(\Lambda^T\phi)$ is an exponential polynomial function of the truncation parameter $T$. For a subspace of $H$-regular automorphic forms (see \S\ref{sec mixed trunc} for the definition), the polynomial coefficient of the zero exponent is a constant independent of $T$.  In \cite[Theorem 4.5]{MR4411860} it is proved that if $\phi$ is $H$-integrable then $\P_H(\phi)$ equals this constant. More generally, for any $H$-regular $\phi$ we denote this constant by $\P_H(\phi)$. By \cite[Theorem 4.1]{MR4411860} this defines an $H(\A_F)$-invariant linear form on the space of $H$-regular automorphic forms and regularizes the period integral. Applying a combinatorial inversion formula, Zydor also obtains a different formula for the regularized period integral as a sum over parabolic subgroups of \emph{closed orbit} integrals over partial truncations of constant terms of $\phi$. 
Applying this formula, the proofs of the above mentioned results of Jacquet and Rallis as well as of Ginzburg, Rallis and Soudry simplify significantly. In fact, it makes the use of the residue method much simpler, especially when the residual representation is realized as the residue of an Eisenstein series induced from a cuspidal representation of a maximal parabolic. This is carried out in many cases in \cite{MR4255059}.

The Eisenstein series that we consider in this work is induced from a residual automorphic representation and this raises another technical difficulty. In fact, the residual representation $\E^m(\pi)$ is not contained in the $H$-regular space (see  Lemma \ref{lem not reg}) and we cannot automatically consider Zydor's regularization $\P_H$.

To explain our results we recall the projection $\E^m:I(\pi,\frac{m+1}2)\rightarrow \E^m(\pi)$. 
Let $\jj:\Sp_{2n}\rightarrow M$ be the imbedding of $\Sp_{2n}$ into the Levi subgroup $M\simeq \GL_m\times \Sp_{2n}$ of $P$ in the second component. We warn the reader that the symbols $P$, $M$, $I(\pi,s)$ and $E^m$ are used in this introduction for the maximal parabolic subgroup with Levi $\GL_m\times \Sp_{2n}$, whereas from \S\ref{ss other res} onwards $P=M\ltimes U$ denotes instead the parabolic subgroup with Levi component isomorphic to $\GL_1^m\times \Sp_{2n}$, and $I^m(\pi)$, $E^m(\xi,\mu)$ the corresponding induced space and Eisenstein series; the passage between the two is the realization of the trivial representation of $\GL_m(\A_F)$ as a multi-residue described below. Similarly, $\jj$ denotes here the imbedding of $\Sp_{2n}$ into $M$, while from \S\ref{vanishing-22} onwards $\jj=\jj_{a,b}$ denotes the imbedding of $\Sp_a\times \Sp_b$ into $\Sp_N$.
It is a simple verification that the linear form
\[
\LL(\varphi)=\int_{K_H}\int_{(\Sp_n(F)\times \Sp_n(F))\bs (\Sp_n(\A_F)\times \Sp_n(\A_F))}\varphi(\jj(h)k)\ dh\ dk
\]
is $H(\A_F)$-invariant on $I(\pi,\frac{m+1}2)$. Here $K_H=K\cap H(\A_F)$ where $K$ is a maximal compact subgroup of $G(\A_F)$. The results of Ginzburg, Rallis and Soudry on $\E^0(\pi)$ imply that the inner period integral is not identically zero and the argument of \cite[Proposition 2]{MR1142486} then implies that $\LL$ is not identically zero.

We now formulate a version of the main result of the paper. It follows from Theorem \ref{thm main}.
In the following theorem we apply Zydor's mixed truncation operator $\Lambda^T$ for $G=\Sp_{m+2n}$ relative to $H=\Sp_{m+n}\times \Sp_n$. Denote by $\P_0(\phi;T)$ the polynomial coefficient of the zero exponent in $T$ of $\int_{H(F)\bs \left(H(\A_F)\cap G(\A_F)^1\right)}\Lambda^T\phi(h)\ dh$ for an automorphic form $\phi$ on $G(\A_F)$.
\begin{theorem}\label{thm main intro}
The polynomial $\P_0(\phi;T)$ is a constant independent of $T$ for any $\phi\in \E^m(\pi)$. Denoting this constant by $\P_H(\phi)$ we have that $\P_H$ is a non-zero $H(\A_F)$-invariant linear form on $\E^m(\pi)$. 
More precisely, up to normalization of measures we have
\[
\P_H(\E^m(\varphi))=\LL(\varphi).
\]
In particular, the linear form $\LL$ factors through the projection $\E^m$.
\end{theorem}
We remark that our only proof that $\P_H|_{\E^m(\pi)}$ is $H(\A_F)$-invariant is via its relation with the linear form $\LL$. Due to the failure of $H$-regularity, this does not automatically follow from Zydor's work. Furthermore, it follows from Proposition \ref{prop Sp vanishing} that if $\min(a,b)<n$ then the $\Sp_a\times \Sp_b$-period integral of the mixed truncation is  identically zero on $\E^m(\pi)$.

We end this introduction by pointing out another crucial obstacle that we encounter. In fact, for general $m$ the Eisenstein series $E^m(\varphi,s)$ for $s\in \C$ in general position still fails the $H$-regularity test. For this reason, we cannot apply the residue method directly to the realization of $\E^m(\pi)$ as residues of $E^m(\varphi,s)$. Instead, by realizing the trivial representation of $\GL_m(\A_F)$ as the multi-residue of an Eisenstein series induced from the unramified principal series we realize $\E^m(\pi)$ as a multi-residue of an Eisenstein series induced from a parabolic subgroup with Levi component isomorphic to $\GL_1^m\times \Sp_{2n}$.

This allows us to carry out the residue method in its updated form. We obtain a recursive formula for the period integral of the Eisenstein series at hand as a sum of closed orbit integrals. The multi-residue of the sum can be expressed as a limit in an $m$-dimensional complex variable. This limit cannot be taken through the sum; however, we show that it can be taken through the sum term by term as an iterated limit in a certain direction. This leads to the proof of Theorem \ref{thm main intro}.
We further remark that if $X$ is the spherical variety $X=H\backslash G$, then the results of this paper are compatible with the discussion of $X$-distinguished Arthur parameters in \cite[Section 16.2]{MR3764130}.

Section~\ref{notation-11} contains notation and Lie-theoretic preliminaries, while Section~\ref{auto-form-11} collects notation and information about automorphic forms that will be used in the sequel. 
Section~\ref{singularities-11} analyzes the singularities of the intertwining operators which appear in the computation of the constant terms of the Eisenstein series studied here.  This is used in Section~\ref{sec sq int}
to construct new residual representations on $\Sp_N$ as multi-residues of Eisenstein series; doing so requires an analysis of the exponents of the Eisenstein series in combination with information about
the intertwining operators.  In Section~\ref{sec mixed trunc} we introduce the mixed truncation operator,  recall Zydor's formula for the 
regularized period integral, and explain how this may be used to regularize the periods of multi-residues of Eisenstein series.  The next two sections establish the vanishing of certain periods; in Section~\ref{vanishing-11} we treat linear periods 
and in Section~\ref{vanishing-22} we treat symplectic periods.  We combine all these pieces in the key Section~\ref{periods of res reps-11} to analyze the regularized periods of the residual
representations and to show the existence of a nonvanishing $H$-invariant form (Theorem~\ref{thm main}).

\section{Notation and preliminaries}\label{notation-11}

Let $F$ be either a local or global field. In the local setting let $\abs{\cdot}=\abs{\cdot}_F$ be the standard absolute value on $F$ and in the global setting let $\A=\A_F$ be the adele ring of $F$ and $\abs{\cdot}=\abs{\cdot}_\A$ the absolute value on $\A^*$.

Let $G$ be a connected reductive group defined over $F$. We denote by $X^*(G)$ the group of $F$-rational characters from $G$ to the multiplicative group $\mathbb{G}_m$ and let $\aaa_G^*=X^*(G)\otimes_\Z \R$. Let $\aaa_G=\Hom(\aaa_G^*,\R)$ be the dual vector space and let $\HC_G$ be the Harish-Chandra map from $G(F)$ in the local case (resp. $G(\A)$ in the global case) to $\aaa_G$ defined by
\[
e^{\sprod{\chi}{\HC_G(g)}}=\abs{\chi(g)},\ \ \ \chi\in X^*(G).
\]
For a real vector space $\aaa$ denote by $\aaa_\C=\aaa\otimes_\R\C$ its complexification.

Let $T_G$ be the split component of $G$, that is, the maximal split torus in the center of $G$. For groups $A\subset B$ denote by $Z_B(A)$ the centralizer of $A$ in $B$ and by $N_B(A)$ the normalizer of $A$ in $B$. 
For a subgroup $Q$ of $G$ denote by $\delta_Q$ the modulus function of $Q(F)$ in the local case (resp. $Q(\A)$ in the global case).

For integers $a\le b$ in this work $[a,b]=\{i\in \Z: a\le i\le b\}$ stands for the integer interval. For $n\in \N$ we write $a^{(n)}$ for the $n$-tuple $(a,\dots,a)$. 
We use the symbol $\sqcup$ for disjoint union of sets.

\subsection{Parabolic subgroups and Iwasawa decomposition}

Let $P_0$ be a minimal parabolic subgroup of $G$ and $T_0$ a maximal $F$-split torus of $G$ contained in $P_0$. 
The group $P_0$ admits a Levi decomposition $P_0=M_0\ltimes U_0$ where $M_0=Z_G(T_0)$ and $U_0$ is the unipotent radical of $P_0$. Let $\aaa_0=\aaa_{T_0}=\aaa_{M_0}$ and $\aaa_0^*=\aaa_{T_0}^*=\aaa_{M_0}^*$.

A parabolic subgroup $P$ of $G$ is called semi-standard if it contains $T_0$, and standard if it contains $P_0$. A semi-standard parabolic subgroup $P$ of $G$ with  unipotent radical $U$ admits a unique Levi decomposition $P=M\ltimes U$ with the Levi subgroup $M$ containing $T_0$. Such a Levi subgroup is called semi-standard. If in addition $P$ is standard, $M$ is called a standard Levi subgroup. Whenever we write $P=M\ltimes U$ is a (semi-)standard parabolic subgroup we implicitly assume the above (semi-)standard Levi decomposition. 

Denote by $\P_\std$ the set of standard parabolic subgroups and by $\P_{\sstd}$ the set of semi-standard parabolic subgroups of $G$.

Fix a maximal compact subgroup $K$ of $G(F)$ in the local case (resp. $G(\A)$ in the global case) in good position with respect to $P_0$ (see \cite[I.1.4]{MR1361168}). For $P=M\ltimes U\in \P_\std$ this allows one to extend 
 $\HC_M$ to a function on $G(F)=U(F)M(F)K$ in the local case (resp. $G(\A)=U(\A)M(\A)K$ in the global case) by 
$
 \HC_M(umk)=\HC_M(m)$. 
  Set $\HC_0=\HC_{M_0}$.

\subsection{Roots and Weyl groups}

\subsubsection{Roots}\label{s roots}

Denote by $R(T_0,G)$ the root system of $G$ with respect to $T_0$, by $R(T_0,P_0)$ the set of positive roots of $R(T_0,G)$ with respect to $P_0$, and by $\Delta_0=\Delta_0^G$ the corresponding set of simple roots. Note that $R(T_0,G)$ lies in $\aaa_0^*$. For every $\alpha\in R(T_0,G)$ we denote by $\alpha^\vee\in \aaa_0$ the corresponding coroot. 

There is a unique element $\rho_0\in \aaa_0^*$ (half the sum of positive roots with multiplicities) such that 
\[
\delta_{P_0}(p)=e^{\sprod{2\rho_0}{\HC_{M_0}(p)}},\ \ \ p\in P_0(F)\ \text{(local) resp. }p\in P_0(\A) \text{ (global).}
\]

For $P=M\ltimes U\in\P_\std$ let $R(T_M,G)$ be the set of non-trivial restrictions to $T_M$ of elements of $R(T_0,G)$, $R(T_M,P)$ the subset of non-trivial restrictions to $T_M$ of elements of $R(T_0,P_0)$ and $\Delta_P$ the subset of non-trivial restrictions to $T_M$ of elements in $\Delta_0$. 
For $\alpha\in R(T_M,G)$ we write $\alpha>0$ if $\alpha\in R(T_M,P)$ and $\alpha<0$ otherwise.

Let $P\subseteq Q=L\ltimes V\in \P_\std$. 
The restriction map from $T_M$ to $T_L$ defines a projection $\aaa_M^*\rightarrow \aaa_L^*$ that gives rise to a direct sum decomposition $\aaa_M^*=\aaa_L^*  \oplus (\aaa_M^L)^*$ (the second component is the kernel of this projection) and a compatible decomposition $\aaa_M=\aaa_L  \oplus \aaa_M^L$ for the dual spaces.  For $\lambda\in \aaa_M^*$ (respectively $\nu\in\aaa_M$) we write $\lambda=\lambda_L+\lambda_M^L$ (respectively, $\nu=\nu_L+\nu_M^L$) for the corresponding decomposition.
Set
\[
\Delta_P^Q=\{\alpha\in \Delta_P:\ \alpha_L=0 \}.
\] 
(Note that $\Delta_P^Q$ identifies with $\Delta_{P\cap L}$ defined with respect to $(L,L\cap P_0)$ replacing $(G,P_0)$.)

The coroot $\alpha^\vee\in\aaa_M$ associated to $\alpha\in \Delta_P$ is defined as follows (cf.\ \cite{MR1361168}): $\alpha$ is the restriction to $T_M$ of a unique $\alpha_0\in\Delta_0$, and we set $\alpha^\vee=(\alpha_0^\vee)_M$. (For roots of $T_M$ that are not simple, coroots are used only in \S\ref{ss sing}, where their normalization is specified.)
Set $(\Delta_P^Q)^\vee=\{\alpha^\vee: \alpha\in \Delta_P^Q\}$ and let $\hat \Delta_P^Q$ be the basis of $(\aaa_M^L)^*$ dual to $(\Delta_P^Q)^\vee$ (the weights) and $({\hat\Delta}_P^Q)^\vee$ the basis of $\aaa_M^L$ dual to $\Delta_P^Q$ (the coweights).

We will use the notation attached to a standard parabolic subgroup and to its Levi subgroup interchangeably: for $P=M\ltimes U\in \P_\std$ we write $\aaa_P=\aaa_M$ and $\Delta_M=\Delta_P$, and for $P\subseteq Q=L\ltimes V$ we write $\lambda_L=\lambda_Q$ for the projection of $\lambda\in\aaa_M^*$ to $\aaa_L^*$. Once the group $A_M$ has been introduced in \S\ref{auto-form-11} we will likewise write $R(A_M,\cdot)$ for $R(T_M,\cdot)$.

Let $\rho_P=(\rho_0)_M$. We have the relation
\[
\delta_P(p)=e^{\sprod{2\rho_P}{\HC_M(p)}},\ \ \ p\in P(F)\ \text{(local) (resp. }P(\A)\ \text{(global))}.
\]

\subsubsection{Weyl groups and Bruhat decomposition}

 We denote by  $W=W_G=N_G(T_0)/M_0$ the Weyl group of $G$.  
 For $P=M\ltimes U\in\P_\sstd$ the semi-standard Levi part $M$ contains $M_0$. Therefore expressions such as $wP$ and $Pw$ are well defined for $w\in W$. Furthermore, $W_M$ is a subgroup of $W$. We write
 \[
 w(P)=wPw^{-1}\ \ \ \text{and}\ \ \ w(M)=wMw^{-1}.
 \]
  
 Let $P=M\ltimes U,\ Q=L\ltimes V\in\P_{\sstd}$.
 The Bruhat decomposition is the bijection $W_MwW_L\mapsto PwQ$ from $W_M\backslash W_G/ W_L$ to $P\backslash G / Q$. 

A Weyl element $w\in W$ is right $M$-reduced if it is of minimal length in $wW_M$ or equivalently if $w\alpha>0$ for all $\alpha\in \Delta_0^M$. The set $[W/W_M]$ of right $M$-reduced elements in $W$ is a complete set of representatives for $W/W_M$. Similarly, let 
\[
[W_M\bs W]=\{w^{-1}: w\in [W/W_M]\}
\] 
be the set of left $M$-reduced elements in $W$.

Let ${}_LW_M$ be the set of reduced representatives in $W_L\bs W/W_M$, that is, the set of $w\in W$ such that $w$ is of minimal length in $W_L w W_M$. Equivalently,
\begin{equation}\label{eq basic red weyl}
{}_LW_M=[W/W_M]\cap [W_L\bs W].
\end{equation}
Then 
\[
w\mapsto W_L w W_M:{}_LW_M \rightarrow W_L\bs W/W_M
\] 
is a bijection. Furthermore, 
$M_w=M\cap w^{-1}Lw$ (respectively $L_w=L\cap wMw^{-1}$)  is a standard Levi subgroup of $M$ (resp. $L$)  for $w\in {}_LW_M$. Let $P_w$ (resp. $Q_w$) be the standard parabolic subgroup of $G$ with Levi $M_w$ (resp. $L_w$). 

For $P=M\ltimes U\in\P_\sstd$ there is a unique $w\in [W/W_M]$ such that  $w(P)\in\P_\std$.
If $P\subseteq Q$, we define each of the sets $\Delta_P^Q$, $(\Delta_P^Q)^\vee$, $\hat\Delta_P^Q$, $(\hat\Delta_P^Q)^\vee$ as the translation via $w$ from the standard parabolic case. For example, $(\hat\Delta_P^Q)^\vee=w^{-1}((\hat\Delta_{w(P)}^{w(Q)})^\vee)$.

\subsubsection{Elementary symmetries}\label{s elementary sym}

For standard Levi subgroups $M\subseteq L$ of $G$ let $W_L(M)$ be the set of elements $w\in [W_L/W_M]$ such that $wMw^{-1}$ is a standard Levi subgroup of $L$. Note that for $w\in W_L(M)$ and $w'\in W_L(wMw^{-1})$ we have $w'w\in W_L(M)$. By definition $W_L(M)$ is the disjoint union over standard Levi subgroups $M'$ of $G$ of the sets 
\[
W_L(M,M')=\{w\in W_L(M): wMw^{-1}=M'\}.
\]

Set $W(M)=W_G(M)$. In \cite[I.1.7, I.1.8]{MR1361168} the elementary symmetries $s_\alpha\in W(M)$ attached to each $\alpha\in \Delta_P$ are introduced and used in order to define a length function $\ell_M$ on $W(M)$. There is a unique element of $W_L(M)$ of maximal length for $\ell_M$ and we denote it by $w_M^L$.

\subsection{Special notation for general linear groups}
Let $N\in\N$ and set $G=\GL_N$. Let $P_0=B_N^\GL$ be the Borel subgroup of upper-triangular matrices
in $G$. We identify the Weyl group $W$ of $G$ with the group $S_N$ of permutations on $[1,N]$. 

Let $\Comp_N$ be the set of all compositions of $N$, that is, all tuples of positive integers $(n_1,\dots,n_k)$ such that $N=n_1+\cdots+n_k$. For convenience we also allow tuples with zero entries and view them as compositions by ignoring those entries. For example, $(1,2,0,3)=(1,2,3)$ in $\Comp_6$. 

The set $\P_\std$ is in bijection with the set $\Comp_N$ of compositions of $N$. For $\alpha=(n_1,\dots,n_k)\in\Comp_N$ let $P_\alpha=M_\alpha\ltimes U_\alpha\in \P_\std$ be the group of block upper triangular matrices with unipotent radical $U_\alpha$ and Levi subgroup
\[
M_\alpha=\{\diag(g_1,\dots,g_k):g_i\in\GL_{n_i},\ i\in [1,k]\}.
\] 

\subsection{Special notation for symplectic groups}

We set up the following notation for the rank $N$ symplectic group defined over $F$.
Let 
\[
G=\Sp_N=\{g\in\GL_{2N}: {}^tgJ_N g=J_N\}
\] 
where $J_N=\sm{}{w_N}{-w_N}{}$ and $w_N=(\delta_{i,N+1-j})\in \GL_N$ and let $P_0=B_N$ be the Borel subgroup of upper triangular matrices and $T_0=M_0$ the torus of diagonal matrices in $G$.
\subsubsection{The Weyl group} 
We identify both $\aaa_0$ and its dual with $\R^N$ in the obvious way so that the duality between them is realized as the standard inner product on $\R^N$. 
With respect to the standard basis $\{e_i\}_{i=1}^N$ of $\R^N$ the positive roots are the set
\[
R(T_0,P_0)=\{e_i-e_j:1\le i<j\le N\}\sqcup \{e_i+e_j:1\le i\le j\le N\}
\]
and the basis of 
simple roots is
\[
\Delta_0=\{e_i-e_{i+1}:i\in [1,N-1]\}\sqcup\{2e_N\}.
\]

We identify the Weyl group $W$ of $G$ with the signed permutation group $\W_N$ on $\R^N$. We write 
\[
\W_N=S_N \ltimes \Xi_N
\] 
where $S_N$ is the permutation group acting on $\R^N$ via 
$\tau e_i=e_{\tau(i)}$, $i\in [1,N]$ and $\Xi_N\simeq (\Z/2\Z)^N$ is realized as the group of subsets of $[1,N]$ with respect to symmetric difference. An element $\set\in \Xi_N$ acts on $\R^N$ by
\[
\set e_i=\begin{cases} -e_i & i\in \set \\ e_i & i\not\in \set \end{cases} \ \ \ i\in [1,N].
\]
Note that $\tau \set \tau^{-1}=\tau(\set)$ is the image of the set $\set\in\Xi_N$ under the permutation $\tau\in S_N$.
Whenever we write $w=\tau\set\in  W$ we always implicitly mean that $\tau\in S_N$ and $\set\in \Xi_N$.

\subsubsection{Standard parabolic subgroups of $G$}
Let $\Comp_N^\Sp$ be the set of tuples $\alpha=(n_1,\dots,n_k;r)$ (mind the semicolon to the left of $r$) where $k,r\in \Z_{\ge 0}$ and $(n_1,\dots,n_k,r)\in\Comp_N$. In accordance with our convention on compositions, we also consider such tuples where we allow some $n_i$'s to be zero and consider them as elements of $\Comp_N^\Sp$ by ignoring those entries. For example, $(1,0,2;3)=(1,2;3)$, however, $(1,2,3;0)\ne (1,2;3)$ in $\Comp_6^\Sp$.

The standard parabolic subgroups of $G$ are in bijection with $\Comp_N^{\Sp}$. 
For $g\in \GL_m$ let $g^*=w_m{}^t g^{-1} w_m$. For $\alpha=(n_1,\dots,n_k;r)\in \Comp_N^\Sp$,
the associated $P_\alpha=M_\alpha\ltimes U_\alpha\in\P_\std$ consists of block upper-triangular matrices with Levi subgroup 
\[
M_\alpha=\{\inj(g_1,\dots,g_k;h):g_i\in \GL_{n_i},\ i\in [1,k],\ h\in \Sp_r\}
\]
and unipotent radical $U_\alpha$. Here and henceforth we set
\[
\inj(g_1,\dots,g_k;h)=\diag(g_1,\dots,g_k,h,g_k^*,\dots,g_1^*)
\]
and if $r=0$ we simply write $\inj(g_1,\dots,g_k)=\diag(g_1,\dots,g_k,g_k^*,\dots,g_1^*)$.

\subsubsection{Reduced Weyl elements}\label{ss reduced}

Let $\alpha=(n_1,\dots,n_k;r)\in\Comp^\Sp_N$ and $P=M\ltimes U=P_\alpha$.
Set $\nu_0=0$ and $\nu_i=\sum_{j=1}^i n_j$, $i\in [1,k]$. 
Explicitly, an element $w=\tau\set\in W$ satisfies $w\in[W/W_M]$ if and only if 
there exist $d_i\in [0,n_i]$ for $i\in [1,k]$ such that 
\begin{itemize}
\item $\set=\sqcup_{i=1}^k [\nu_{i-1}+d_i+1,\nu_i]$ and
\item $\tau$ is order reversing on $[\nu_{i-1}+d_i+1,\nu_i]$ and order preserving on $[\nu_k+1,N]$ and on $[\nu_{i-1}+1,\nu_{i-1}+d_i]$ for every $i\in[1,k]$.
\end{itemize}

\section{Automorphic forms}\label{auto-form-11}
Let $F$ be a number field. Denote by $A_G$ the connected component of the identity in the group of real points of the maximal $\Q$-split torus in the restriction of scalars $\Res_{F/\Q}(T_G)$. It is naturally a subgroup of $G(\A)$ and we have $G(\A)=A_G \times G(\A)^1$ where  $G(\A)^1=\ker \HC_G$.
Denote by $[G]=G(F)\bs G(\A)$ the automorphic quotient of $G$. For a subgroup $Q$ of $G$ let $Q(\A)^{1,G}=Q(\A)\cap G(\A)^1$ and set $[Q]^{1,G}=Q(F)\bs Q(\A)^{1,G}$. 

For a parabolic subgroup $P=M\ltimes U$ of $G$ let $\Aut_P=\Aut_P(G)$ be the space of automorphic forms on $U(\A)M(F)\bs G(\A)$ as defined in \cite[I.2.17]{MR1361168}. For $\varphi\in \Aut_P$ and $\lambda\in \aaa_{P,\C}^*$ let
\[
\varphi_\lambda(g)=e^{\sprod{\lambda}{\HC_0(g)}} \varphi(g),\ \ \ g\in G(\A).
\]
Then $\varphi_\lambda\in \Aut_P$. 

We consider various operations between the spaces $\Aut_P$ for various $P$ (see \cite{MR1361168} or \cite{arXiv:1911.02342v3} for the aforementioned results).
\subsection{Eisenstein series and constant terms}\label{ss eis ct}
Let $P=M\ltimes U\subseteq Q=L\ltimes V$ in $\P_{\sstd}$. 
The Eisenstein series $E_P^Q(g,\varphi,\lambda)$ for $g\in G(\A)$ and $\varphi\in \Aut_P$ is the meromorphic continuation to $\lambda\in \aaa_{P,\C}^*$ of the function defined for $\lambda$ such that $\sprod{\lambda}{\alpha^\vee}\gg 1$, $\alpha\in \Delta_P^Q$ by the convergent series
\[
E_P^Q(g,\varphi,\lambda)=\sum_{\gamma\in P(F)\bs Q(F)} \varphi_\lambda(\gamma g).
\]
Whenever holomorphic in $\lambda$ we have $E_P^Q(\varphi,\lambda)\in \Aut_Q$.
Note that $L\cap P$ is a parabolic subgroup of $L$ and the inclusion of $L(F)$ in $Q(F)$ defines a bijection 
\[
(L\cap P)(F)\bs L(F)\simeq P(F)\bs Q(F). 
\]
Consequently, for any $g\in G(\A)$ and $\lambda$ regular we have 
\begin{equation}\label{eq Eis on Levi}
E_P^Q(\cdot g,\varphi,\lambda)=E_{L\cap P}^L(\cdot,\xi,\lambda)\in \Aut_{L}(L)
\end{equation}
is an Eisenstein series on $L$ with $\xi=\varphi(\cdot g)|_{L(\A)}\in \Aut_{L\cap P}(L)$.

The constant term map $C_{Q,P}:\Aut_Q \rightarrow \Aut_P$ is defined by the integral
\[
C_{Q,P}\phi(g)=\int_{[U\cap L]}\phi(ug)\ du=\int_{[U]}\phi(ug)\ du.
\]
It is straightforward that for $w\in W$ and $\phi\in \Aut_G$ we have
\begin{equation}\label{eq nonstd const}
C_{G,w(P)}\phi(wg)=C_{G,P}\phi(g),\ \ \ g\in G(\A).
\end{equation}

\subsection{Intertwining operators}
Let $P=M\ltimes U\in \P_\std$ and $w\in W(M)$. Let $L=wMw^{-1}$ and let $Q=L\ltimes V$ be the unique element of $\P_\std$ with Levi subgroup $L$. The intertwining operator $M(w,\lambda)$ is the meromorphic continuation to $\lambda\in \aaa_{P,\C}^*$ of the function defined for 
$\varphi\in \Aut_P$ and $\lambda$ such that $\sprod{\lambda}{\alpha^\vee}\gg 1$ for $\alpha\in \Delta_P$ such that $w\alpha<0$ by the convergent integral
\[
(M(w,\lambda)\varphi)_{w\lambda}(g)=\int_{(V\cap wUw^{-1})(\A)\bs V(\A)} \varphi_\lambda(\tilde w^{-1} ug)\ du.
\]
Here $\tilde w\in G(F)$ is a choice of representative for $w$. The operator $M(w,\lambda)$ is independent of this choice.
Whenever holomorphic at $\lambda$ we have $M(w,\lambda):\Aut_P\rightarrow \Aut_Q$.

For $w'\in W(wMw^{-1})$ we have (see \cite[Theorem 2.3 (5)]{arXiv:1911.02342v3})
\begin{equation}\label{eq mult int}
M(w',w\lambda)\circ M(w,\lambda)=M(w'w,\lambda).
\end{equation}

\subsection{The global geometric lemma} Let $P=M\ltimes U$ and $Q=L\ltimes V$ in $\P_{\sstd}$ and $\varphi\in \Aut_P$.

As a meromorphic function on $\aaa_{M,\C}^*$ we have (see \cite[(6.17)]{arXiv:1911.02342v3}) 
\begin{equation}\label{eq consterm}
C_{G,Q}E_P^G(\varphi,\lambda)=\sum_{w\in {}_LW_M} E_{Q_w}^Q(M(w,\lambda)(C_{P,P_w}\varphi),w\lambda).
\end{equation}
Here, for $w\in {}_LW_M$, the operator $M(w,\lambda)$ is the one attached to the standard parabolic $P_w$ with Levi $M_w=M\cap w^{-1}Lw$, on which $C_{P,P_w}\varphi$ lives, so that $w\in W(M_w)$. 

Let $\Aut_P^\circ$ be the space of cusp forms in $\Aut_P$, that is, $\varphi\in \Aut_P$ such that $C_{P,Q}\varphi=0$ for every proper parabolic subgroup $Q$ of $P$. For $\varphi\in \Aut_P^\circ$ the formula simplifies to 
\begin{equation}\label{eq consterm cusp}
C_{G,Q}E_P^G(\varphi,\lambda)=\sum_{w\in {}_LW^\circ_M} E_{Q_w}^Q(M(w,\lambda)\varphi,w\lambda)
\end{equation}
where
\[
{}_LW^\circ_M=\{w\in  {}_LW_M: w(M)\subseteq L\}.
\]

\subsection{Generalized eigenspaces}
Let $P=M\ltimes U\in\P_\std$. We consider the twisted action of $A_M$ on $\Aut_P$ by 
\[
a\cdot \varphi (g)=\delta_P^{-1/2}(a)\varphi(ag),\ \ \ a\in A_M,\ g\in G(\A).
\]
It gives rise to a decomposition (see \cite[I.3.2]{MR1361168})
\[
\Aut_P=\oplus_{\lambda\in \aaa_{P,\C}^*} \Aut_{P,\lambda}
\]
where $\Aut_{P,\lambda}$ is the generalized eigenspace of the twisted $A_M$-action for the character $a\mapsto e^{\sprod{\lambda}{\HC_0(a)}}$.
We will simply say that $\lambda$ is the (generalized) eigenvalue.

The following will be used freely throughout the paper.
For $\mu, \lambda\in \aaa_{P,\C}^*$ and $\varphi\in \Aut_{P,\mu}$ we have:
\begin{equation}\label{eq hol sec ev}
\varphi_\lambda\in \Aut_{P,\mu+\lambda},
\end{equation}
and
\begin{equation}\label{eq eis ev}
E_P^Q(\varphi,\lambda)\in \Aut_{Q, (\mu+\lambda)_Q}
\end{equation}
for every parabolic $Q\supseteq P$ whenever regular in $\lambda$. For $w\in W(M)$ we have
\begin{equation}\label{eq intop ev}
M(w,\lambda)\varphi\in \Aut_{Q,w\mu}
\end{equation}
whenever regular in $\lambda$ where $Q$ is the standard parabolic subgroup of $G$ with Levi subgroup $wMw^{-1}$.

\subsection{Exponents}
As explained in \cite[I.2.17]{MR1361168}, the group $G(\A)$ does not act on $\Aut_P$ by right translations. Instead, $\Aut_P$ is a $(\mathfrak{g}_\infty\times K_\infty\times G(\A_f))$-module where $\A_f=\prod'_{v\nmid \infty} F_v$ is the ring of finite adeles, $F_\infty=\prod_{v\mid \infty} F_v$, $\mathfrak{g}_\infty$ is the complexified Lie algebra of $G(F_\infty)$ and $K_\infty=\prod_{v\mid \infty} K_v$ is the maximal compact subgroup of $G(F_\infty)$.

Accordingly, when we say that $(\pi,V)$ is an automorphic representation of $G(\A)$ we mean that $V \subseteq \Aut_G$ and that $\pi$ represents the $(\mathfrak{g}_\infty\times K_\infty\times G(\A_f))$-module action on $V$.

The exponents of $\pi$ along $P=M\ltimes U\in \P_\std$ consist of the finite set $\E_P(\pi)$ of $\lambda\in \aaa_{M,\C}^*$ such that the projection of $C_{G,P}(V)$ to $\Aut_{P,\lambda}$ is non-zero.
\subsection{Induced representations}
Let $(\pi,V)$ be an automorphic representation of $M(\A)$. Denote by $I_P^G(\pi)$ the space of $\varphi\in \Aut_P$ such that
\[
(m\mapsto \delta_P^{-1/2}(m)\varphi(mg))\in V
\]
and let $I_P^G(\pi,\lambda)$ be the $(\mathfrak{g}_\infty\times K_\infty\times G(\A_f))$-module induced by the $G(\A)$-action $(R(g)(\varphi_\lambda))_{-\lambda}$ where for $g\in G(\A)$, $R(g)$ is the action by right translation. 

For $w\in W(M)$ let $w\pi=\pi\circ\Ad(\tilde w^{-1})$ be the representation of $w(M)$ on the space of $\pi$. Its isomorphism class is independent of a choice of representative $\tilde w$ for $w$ in $G(F)$. The associated intertwining operator restricts to $M(w,\lambda):I_P^G(\pi,\lambda)\rightarrow I_Q^G(w\pi,w\lambda)$ where $Q\in\P_{\std}$ has Levi subgroup $w(M)$.

\subsection{A convention on equivariant linear forms} Let $H$ be a subgroup of $G$. By a character of $[H]$ we mean a character of $H(\A)$ that is trivial on $H(F)$. Let $\chi$ be a character of $[H]$ and $V$ a $(\mathfrak{g}_\infty\times K_\infty\times G(\A_f))$-submodule of $\Aut_P$.
By abuse of notation, we will say that a linear form on $V$ is $(H(\A),\chi)$-equivariant (or $H(\A)$-invariant if $\chi$ is trivial) if it is $(H(\A_f) \times K_\infty\cap H(F_\infty)\times \mathfrak{h}_\infty\,,\chi)$-equivariant (we continue to denote by $\chi$ the character it induces on $\mathfrak{h}_\infty$).

\subsection{Measures}
Fix a $W$-invariant inner product on $\aaa_0$. For $P=M\ltimes U\in \P_\sstd$ the decomposition $\aaa_0=\aaa_M\oplus \aaa_0^M$ is orthogonal. This defines a Haar measure on each of the vector spaces $\aaa_M$. Via the isomorphism $a\mapsto \HC_M(a):A_M\rightarrow \aaa_M$ this gives a Haar measure on $A_M$. 

We endow discrete groups with the counting measure. For any unipotent subgroup $U$ of $G$ let $du$ be the Haar measure on $U(\A)$ such that $[U]$ has volume one. 
 Let $dk$ be the Haar measure on $K$ of total volume one.

Fix a Haar measure $dg$ on $G(\A)$. For $P=M\ltimes U\in \P_{\sstd}$ a Haar measure $dm$ on $M(\A)$ is then determined via the Iwasawa decomposition by the requirement
\[
\int_{G(\A)} f(g)\ dg=\int_K\int_{M(\A)}\int_{U(\A)} f(umk)\delta_P^{-1}(m)\ du\ dm\ dk.
\]
Via the isomorphism $M(\A)\simeq A_M \times M(\A)^1$ this also determines a Haar measure on $M(\A)^1$.

\section{Singularities of certain intertwining operators}\label{singularities-11}
Our first task in this work is to construct certain new residual representations of $\Sp_N(\A)$ as multi-residues of certain Eisenstein series. This will require a careful analysis of the exponents of these Eisenstein series as well as the study of singularities of certain intertwining operators arising from  \eqref{eq consterm}. In this section we start with the necessary preparation.

\subsection{A strategy}\label{ss strategy sing}
In order to analyze the singularities of intertwining operators in this paper we apply the following strategy. Consider a parabolic subgroup $P=M\ltimes U\in \P_\std$, $w\in W(M)$ and an irreducible automorphic representation $\pi$ of $M$. Let $Q\in \P_\std$ be the parabolic with Levi subgroup $wMw^{-1}$.  The intertwining operator 
\[
M(w,\lambda):I_P^G(\pi)\rightarrow I_Q^G(w\pi)
\] 
decomposes as a tensor product over all places of $F$. More precisely, for $\phi\in I_P^G(\pi)$ decomposable, there is a finite set of places $S$ containing all archimedean places and all finite places where $\pi_v$ ramifies such that $\phi=\otimes_v \phi_v$ where $\phi_v$ is the unramified section such that $\phi_v(e)=1$ for all $v\not\in S$. Applying the Gindikin-Karpelevich formula, in all cases we consider, we can prescribe an Eulerian meromorphic function (a product of quotients of completed global $L$-functions) $f=f_w=\prod_v f_v$ on $\aaa_{M,\C}^*$ such that 
\[
M(w,\lambda)\phi=f(\lambda) \prod_{v\in S} \frac{M_v(w,\lambda)\phi_v}{f_v(\lambda)}  \otimes_{v\not\in S}\tilde\phi_v
\]
where $\tilde \phi_v$ is the unramified section induced from $w\pi_v$ for $v\not\in S$ normalized to be one at the identity.

In order to show that $M(w,\lambda)$ is holomorphic at $\lambda_0$ it suffices to show the following two statements.
\begin{itemize}
\item {\bf The global statement:} $f(\lambda)$ is holomorphic at $\lambda=\lambda_0$.
\item {\bf The local statement:} $\frac{M_v(w,\lambda)\phi_v}{f_v(\lambda)}$ is holomorphic at $\lambda=\lambda_0$ for every place $v$.
\end{itemize}
Similarly, in order to show that $M(w,\lambda)$ has singularities without multiplicity
(in the sense of \cite[\S IV.1.6]{MR1361168}) at $\lambda_0$ it suffices to show that
\begin{itemize}
\item {\bf The global multiplicity free statement:} $f(\lambda)$ has singularities without multiplicity at $\lambda=\lambda_0$.
\item {\bf The local statement:} $\frac{M_v(w,\lambda)\phi_v}{f_v(\lambda)}$ is holomorphic at $\lambda=\lambda_0$ for every place $v$.
\end{itemize}

\subsection{Singularity lemmas}\label{ss sing}
\begin{lemma}\label{lem singunrm}
Let $G$ be either $\GL_2$ or $\SL_2$, $\alpha\in R(T_0,P_0)$ the unique positive root and $w$ the non-trivial Weyl element of $G$. 
Then 
\[
(\sprod{\lambda}{\alpha^\vee}-1)M(w,\lambda):I_{P_0}^G(\triv_{M_0}) \rightarrow I_{P_0}^G(\triv_{M_0})
\] 
is holomorphic at all $\lambda\in \aaa_{0,\C}^*$ such that $\Re(\sprod{\lambda}{\alpha^\vee})>0$.
\end{lemma}
\begin{proof}
We apply the strategy in \S\ref{ss strategy sing} to $M(w,\lambda)$ with $f(\lambda)=f_w(\lambda)=\frac{\zeta_F(\sprod{\lambda}{\alpha^\vee})}{\zeta_F(\sprod{\lambda}{\alpha^\vee}+1)}$. Since $(s-1)\zeta_F(s)$ is holomorphic for $\Re(s)>0$, while $\zeta_F(s+1)$ is holomorphic and non-zero for $\Re(s)>0$, the global (if $\sprod{\lambda}{\alpha^\vee}\ne 1$) respectively, global multiplicity free (if $\sprod{\lambda}{\alpha^\vee}=1$) statement follows. The local statement is well known, in fact, $f_v(\sprod{\lambda}{\alpha^\vee})^{-1}M_v(w,\lambda)$ is holomorphic for $\Re(\sprod{\lambda}{\alpha^\vee})\ge 0$ (see e.g. \cite[I.2]{MR1026752}). 
\end{proof}

\begin{lemma}\label{lem intglhol}
Let $p\in \N$, $\pi$ an irreducible cuspidal automorphic representation of $\GL_p(\A)$ with unitary central character and $G=\GL_{p+1}$. Assume that either $p\ge 2$ or $\pi$ is not an unramified character so that $L(\pi,s)$ and $L(\pi^\vee,s)$ are entire. 
Let $P=M\ltimes U=P_{(p,1)}$, $Q=L\ltimes V=P_{(1,p)}$, $w\in W(M)$ have representing matrix
\[
\tilde w=\begin{pmatrix} & 1 \\ I_p & \end{pmatrix},
\]
$R(A_M,P)=\{\alpha\}$ and $R(A_L,Q)=\{\beta\}$. 
Then for $\lambda\in \aaa_{M,\C}^*$ we have
\begin{enumerate}
\item\label{part k1m} $M(w,\lambda):I_P^G(\pi\otimes \triv_{\GL_1}) \rightarrow I_Q^G(\triv_{\GL_1}\otimes \pi)$ is holomorphic whenever $\Re(\sprod{\lambda}{\alpha^\vee})\ge\frac12$;
\item\label{part 1mk} $M(w^{-1},w\lambda): I_Q^G(\triv_{\GL_1}\otimes \pi)\rightarrow I_P^G(\pi\otimes \triv_{\GL_1})$ is holomorphic whenever $-\Re(\sprod{\lambda}{\alpha^\vee})=\Re(\sprod{w\lambda}{\beta^\vee})\ge\frac12$.
\end{enumerate}
\end{lemma}
\begin{proof}
We apply the strategy in \S\ref{ss strategy sing} for both cases. The corresponding global Eulerian functions are respectively
\[
f_1(\lambda)=f_w(\lambda)=  \frac{L(\pi,\sprod{\lambda}{\alpha^\vee})}{L(\pi,\sprod{\lambda}{\alpha^\vee}+1)}\ \ \ \text{and} \ \ \ f_2(\lambda)=f_{w^{-1}}(w\lambda)= \frac{L(\pi^\vee,\sprod{w\lambda}{\beta^\vee})}{L(\pi^\vee,\sprod{w\lambda}{\beta^\vee}+1)}.
\]
The functions $L(\pi,s)$ and $L(\pi^\vee,s)$ are entire and do not vanish at $\Re(s)> 1$. The global statement follows in both cases. For the local statement we recall that $\pi_v$ is an irreducible, unitary and generic representation of $\GL_p(F_v)$ for every place $v$ of $F$. Representing $\pi_v$ as in \cite[I.11]{MR1026752}, in both cases the local statement follows from \cite[I.2]{MR1026752}. The lemma follows.
\end{proof}

Let $p,m\in \N$ and $G=\Sp_{m+p}$. For $i\in [0,m]$ let $Q_i=L_i\ltimes V_i=P_{(1^{(i)},p,1^{(m-i)};0)}$ and
\[
I^m_i(\pi)=I_{Q_i}^G(\triv_{\GL_1^{i}}\otimes \pi\otimes \triv_{\GL_1^{m-i}}). 
\]
For $w\in W(L_i)$ let 
\[
R_{i,w}=\{\alpha\in R(A_{L_i},G):\alpha>0,\ w\alpha<0\}.
\]

Let $\{e_1^i,\dots,e^i_{m+1}\}$ be a basis of $\aaa_{L_i}^*$ so that the set of positive roots is
\[
R(A_{L_i},Q_i)=\{e_l^i-e_k^i: 1\le l<k\le m+1\}\sqcup \{e_l^i+e_k^i: 1\le l\le k\le m+1\}.
\]
For a root $\alpha$ of $A_{L_i}$ we normalize the coroot $\alpha^\vee\in\aaa_{L_i}$ as follows: $\alpha^\vee=(\alpha_0^\vee)_{L_i}$ where $\alpha_0$ is a root of $T_0$ restricting to $\alpha$ that is of the form $e_x\pm e_y$ with $x\ne y$ if $\alpha=e^i_l\pm e^i_k$ with $l<k$, and of the form $2e_x$ if $\alpha=2e^i_l$. With this normalization $\alpha^\vee$ is independent of the choice of $\alpha_0$ (the only ambiguity arises for $\alpha=2e^i_{i+1}$, which is also the restriction of the roots $e_x+e_y$ with $x\ne y$ in the block of $\GL_p$), it agrees with \S\ref{s roots} for $\alpha\in\Delta_{Q_i}$, and for $\lambda\in \aaa_{L_i,\C}^*\subseteq \aaa_{0,\C}^*$ we simply have $\sprod{\lambda}{\alpha^\vee}=\sprod{\lambda}{\alpha_0^\vee}$, that is, $\sprod{\lambda}{\alpha^\vee}$ equals $\lambda_x\pm\lambda_y$ if $\alpha_0=e_x\pm e_y$ and $\lambda_x$ if $\alpha_0=2e_x$, where $\lambda=(\lambda_1,\dots,\lambda_{m+p})$ in the coordinates of $\aaa_{0,\C}^*$.
Set $E_i=\{\pm e_l^i: l\in [1,m+1]\}$.
For $w\in W(L_i)$ there exists $j\in [0,m]$ such that $w(L_i)=L_j$. Note that $w$ defines a bijection from $E_i$ to $E_j$ and from $R(A_{L_i},G)$ to $R(A_{L_j},G)$. Also
\[
{}_{L_j}W_{L_i}^\circ=\{w\in W(L_i): w(L_i)=L_j\}.
\]

Note that for $w\in {}_{L_j}W_{L_i}^\circ$ we have
\begin{equation}\label{eq wlr cond}
we_{i+1}^i=\begin{cases} e_{j+1}^j & 2e_{i+1}^i\not\in R_{i,w}\\ -e_{j+1}^j & 2e_{i+1}^i\in R_{i,w}. \end{cases} 
\end{equation}

Let 
\[
R_i=\{e^i_l\pm e^i_k\in R(A_{L_i},Q_i): i+1\in \{l,k\}\}.
\]
Let $\aaa_{i,w}^+$ be the set of $\lambda\in \aaa_{L_i}^*$ such that 
\begin{itemize}
\item $\sprod{\lambda}{\alpha^\vee}>0$ if $\alpha\in R_{i,w}\setminus R_i$;
\item $\sprod{\lambda}{\alpha^\vee}\ge \frac12$ if $\alpha\in R_{i,w}\cap R_i$.
\end{itemize}
\begin{proposition}\label{prop sing}
With the above notation let $\pi$ be an irreducible cuspidal automorphic representation of $\GL_p(\A)$ with unitary central character such that either $p\ge 2$ or $\pi$ is not an unramified character. Let $w\in {}_{L_j}W_{L_i}^\circ$ be such that $2e_{i+1}^i\not\in R_{i,w}$. Then the intertwining operator 
\[
[\prod_{\alpha\in R_{i,w}\setminus R_i}(\sprod{\lambda}{\alpha^\vee}-1)]M(w,\lambda):I^m_i(\pi)\rightarrow I^m_j(\pi)
\]
is holomorphic at all $\lambda\in \aaa_{L_i,\C}^*$ such that $\Re(\lambda)\in \aaa_{i,w}^+$.
\end{proposition}
\begin{proof}
Based on the decomposition of $w$ as a product of elementary symmetries, \cite[I.1.8]{MR1361168}, we prove the proposition by induction on the length of $w$.
If $w=e$ there is nothing to prove. Otherwise, by \cite[I.1.8]{MR1361168} there exists $\alpha_0\in \Delta_{Q_i}$ such that $w\alpha_0<0$, i.e. $\alpha_0\in \Delta_{Q_i}\cap R_{i,w}$.  Fix such $\alpha_0$ and write $w=w's$ where $s=s_{\alpha_0}$ is the elementary symmetry associated with $\alpha_0$ (so that $\ell(w')=\ell(w)-1$). Let $h\in [0,m]$ be such that $s(L_i)=L_h$. 
It follows from Lemmas \ref{lem singunrm} and \ref{lem intglhol} that
\[
(\sprod{\lambda}{\alpha_0^\vee}-1)^\epsilon M(s,\lambda)\ \ \ \text{where}\ \ \ \epsilon=\begin{cases} 0& \alpha_0\in R_i\\ 1 & \text{otherwise}\end{cases}
\] 
is holomorphic whenever  $\Re(\lambda)\in \aaa_{i,s}^+$. Note further that $\aaa_{i,w}^+ \subseteq \aaa_{i,s}^+$. 
The fact that
\[
R_{i,w}=\{\alpha_0\}\sqcup s^{-1}(R_{h,w'})
\]
together with \eqref{eq wlr cond} imply that 
\[
s^{-1}(R_{h,w'}\setminus R_h)= R_{i,w}\setminus (R_i\sqcup\{\alpha_0\})\ \ \ \text{and}\ \ \ s^{-1}(R_{h,w'}\cap R_h)= (R_{i,w}\cap R_i)\setminus \{\alpha_0\}
\]
and consequently $s(\aaa_{i,w}^+)\subseteq \aaa_{h,w'}^+$. By the induction hypothesis it follows that 
\[
[\prod_{\beta\in R_{h,w'}\setminus R_h}(\sprod{s\lambda}{\beta^\vee}-1)]M(w',s\lambda)
\]
is holomorphic when $\Re(\lambda)\in \aaa_{i,w}^+$. By \eqref{eq mult int} we have $M(w,\lambda)=M(w',s\lambda)\circ M(s,\lambda)$. Since
we also have
\[
[\prod_{\alpha\in R_{i,w}\setminus R_i}(\sprod{\lambda}{\alpha^\vee}-1)]=(\sprod{\lambda}{\alpha_0^\vee}-1)^\epsilon\cdot [\prod_{\beta\in R_{h,w'}\setminus R_h} (\sprod{s\lambda}{\beta^\vee}-1)]
\]
the proposition follows.
\end{proof}
Assume now that $p=2n$ is even and that $\pi$ satisfies
\begin{itemize}
\item $\Res_{s=1}L(\pi,\wedge^2,s)\ne 0$ and
\item $L(\pi,\frac12)\ne 0$.
\end{itemize}
\begin{corollary}\label{cor sing}
Let $j\in [0,m]$, $w\in {}_{L_j}W^\circ_{L_m}$ and $\alpha_0=2e_{m+1}^m\in \Delta_{L_m}$. The intertwining operator 
\[
[(\sprod{\lambda}{\alpha_0^\vee}-\frac12)^\epsilon\prod_{\alpha\in R_{m,w}\setminus R_m}(\sprod{\lambda}{\alpha^\vee}-1)]M(w,\lambda):I^m_m(\pi)\rightarrow I^m_j(\pi) \ \ \ \text{where} \ \ \ \epsilon=\begin{cases} 0 & \alpha_0\not\in R_{m,w} \\ 1 & \alpha_0\in R_{m,w} \end{cases}
\]
is holomorphic whenever $\lambda\in \aaa_{L_m,\C}^*$ is such that $\Re(\lambda)\in \aaa_{m,w}^+$ and $\sprod{\lambda}{\alpha_0^\vee}= \frac12$ if $\epsilon=1$.
\end{corollary}
\begin{proof}
If $\epsilon=0$ this is Proposition \ref{prop sing} with $i=m$. Otherwise, as in the proof of the proposition $w=w's$ where $s=s_{\alpha_0}$
is the elementary symmetry associated to $\alpha_0$. Note that $s(L_m)=L_m$ and $R_{m,w}=\{\alpha_0\}\sqcup s^{-1}R_{m,w'}$. It follows from \cite{MR1740991} that $(\sprod{\lambda}{\alpha_0^\vee}-\frac12)M(s,\lambda)$ is holomorphic whenever $\sprod{\lambda}{\alpha_0^\vee}=\frac12$. It further follows that $\Re(s\lambda)\in \aaa_{m,w'}^+$ and 
\[
\prod_{\alpha\in R_{m,w'}\setminus R_m}(\sprod{s\lambda}{\alpha^\vee}-1)=\prod_{\alpha\in R_{m,w}\setminus R_m}(\sprod{\lambda}{\alpha^\vee}-1).
\]
By Proposition \ref{prop sing} 
\[
[\prod_{\alpha\in R_{m,w'}\setminus R_m}(\sprod{s\lambda}{\alpha^\vee}-1)]M(w',s\lambda)
\] 
is holomorphic whenever $\Re(\lambda)\in \aaa_{m,w}^+$.  
The corollary follows applying the decomposition $M(w,\lambda)=M(w',s\lambda)\circ M(s,\lambda)$ (see \eqref{eq mult int}).
\end{proof}

\section{Construction of residual representations}\label{sec sq int}
Fix $n\in \N$ and let $\pi$ be an irreducible, cuspidal automorphic representation of $\GL_{2n}(\A)$ such that
\begin{itemize}
\item $\Res_{s=1}L(\pi,\wedge^2,s)\ne 0$ and
\item $L(\pi,\frac12)\ne 0$.
\end{itemize}
The first condition implies that $\pi$ is self-dual and in particular that the space of $\pi$ lies in $\Aut_{\GL_{2n},0}(\GL_{2n})$.
Fix $m\in \Z_{\ge 0}$ and set $N=m+2n$ and $G=\Sp_N$.

 \subsection{The construction of residues}\label{ss construct res 1m2n}
 Let $Q_m=L_m\ltimes V_m=P_{(1^{(m)},2n;0)}$. Set 
 \[
 I_m(\pi)=I_{Q_m}^G(\triv_{\GL_1^m}\otimes \pi) \ \ \ \text{and}\ \ \ E_m(\varphi,\lambda)=E_{Q_m}^G(\varphi,\lambda),\ \ \ \varphi\in I_m(\pi),\ \lambda\in \aaa_{L_m,\C}^*.
 \]
For $\lambda\in \aaa_{L_m,\C}^*$ we further write $I_m(\pi,\lambda)=I_{Q_m}^G(\triv_{\GL_1^m}\otimes \pi,\lambda)$.
Set
\[
\lambda_0=(m,\dots,2,1,(\frac12)^{(2n)})\in \aaa_{L_m}^*.
\]
It follows from  \cite[Proposition IV.1.11 (c)]{MR1361168} that the Eisenstein series $E_m(\varphi,\lambda)$ has singularities without multiplicity
in the sense of \cite[\S IV.1.6]{MR1361168} at $\lambda=\lambda_0$. We denote by $\E_m(\varphi)$ the multi-residue at $\lambda=\lambda_0$ of $E_m(\varphi,\lambda)$ and set
\[
\E^m(\pi)=\{\E_m(\varphi):\varphi\in I_m(\pi)\}.
\]
Thus, $\E^m(\pi)$ is an automorphic representation of $G(\A)$ and 
\[
\E_m: I_m(\pi,\lambda_0)\rightarrow \E^m(\pi)
\]
is a projection of representations.

It follows from \cite{MR1740991} and \cite{MR2848523} that for the case $m=0$, $\E^0(\pi)$ is an irreducible, discrete automorphic representation of $G(\A)$. Our first goal is to prove the analogue for $\E^m(\pi)$ for every $m\in \N$.

\subsection{Singularities of intertwining operators} Write $\lambda=(\lambda_1,\dots,\lambda_m,t^{(2n)})\in\aaa_{L_m,\C}^*$.
Note that the exponent $\lambda_0$ defined above lies in the intersection of the $m+1$ hyperplanes of $\aaa_{L_m}^*$ defined by the equations
\[
t=\frac12,\ \lambda_m=1,\ \lambda_i-\lambda_{i+1}=1,\ i\in [1,m-1].
\]
(Here and below, when $m=0$ the coordinates $\lambda_1,\dots,\lambda_m$ are absent, the only equation is $t=\frac12$ and, accordingly, $q(\lambda)=t-\frac12$ in Lemma \ref{lem hol cusp}.)

For $i\in [0,m]$ let $Q_i=L_i\ltimes V_i=P_{(1^{(i)},2n,1^{(m-i)};0)}$. 
For $Q=L\ltimes V\in \P_\std$ we have that ${}_LW^\circ_{L_m}$ is empty unless $L$ contains $L_i$ for some $i\in [0,m]$. As a consequence of \eqref{eq consterm cusp} we also have $C_{G,Q}E_m(\varphi,\lambda)=0$ for $\varphi\in I_m(\pi)$
unless $L$ contains $L_i$ for some $i\in [0,m]$.

\begin{lemma}\label{lem hol cusp}
Let 
\[
\Gamma=\{\beta_1,\dots,\beta_{m+1}\}\subseteq R(A_{L_m},Q_m) 
\]
where for $\lambda=(\lambda_1,\dots,\lambda_m,t^{(2n)})\in \aaa_{L_m,\C}^*$ we have
\[
\sprod{\lambda}{\beta_j^\vee}=\begin{cases} \lambda_j-\lambda_{j+1} & j\in [1,m-1] \\ \lambda_m & j=m\\ t & j=m+1. \end{cases}
\]
\begin{enumerate}
\item For $i\in [0,m]$ and $w\in {}_{L_i}W^\circ_{L_m}$ the intertwining operator
\[
\left[\prod_{\beta\in \Gamma,\ w\beta<0}\sprod{\lambda-\lambda_0}{\beta^\vee} \right] M(w,\lambda):  I_m(\pi)\rightarrow I_{Q_i}^G(\triv_{\GL_1^{i}}\otimes \pi\otimes \triv_{\GL_1^{m-i}})
\]
 is holomorphic at $\lambda=\lambda_0$. 
 \item In particular, for
 \[
q(\lambda)=\prod_{\beta\in \Gamma}\sprod{\lambda-\lambda_0}{\beta^\vee}=(t-\frac12)(\lambda_m-1)\prod_{i=1}^{m-1} (\lambda_i-\lambda_{i+1}-1)
\]
we have 
\[
\lim_{\lambda \to \lambda_0}q(\lambda)M(w,\lambda)=0
\]
unless $w=\wperm w_{L_m}$ where $\wperm\in S_{m+2n}$ is represented in $\GL_{m+2n}$ by the permutation matrix
\[
\tilde \wperm=\begin{pmatrix} I_i & & \\ & & I_{2n}\\ & I_{m-i} &  \end{pmatrix}.
\]
\end{enumerate}
\end{lemma}
\begin{proof}
The first part follows from Corollary \ref{cor sing} and the simple observations that in the notation of \S\ref{ss sing} we have $\lambda_0\in \aaa_{m,w}^+$ for any $w\in {}_{L_i}W^\circ_{L_m}$ and restriction to $A_{L_m}$ defines an imbedding of $\Gamma \setminus \{\beta_{m+1}\}$ into $R(A_{L_m},Q_m)$ so that
\[
\Gamma\setminus R_m=\Gamma\setminus\{\beta_{m+1}\}=\{\alpha\in R(A_{L_m},Q_m):\sprod{\lambda_0}{\alpha^\vee}=1\}.
\]
In order to deduce the second part from Corollary \ref{cor sing} we need to show that if $w\in {}_{L_i}W^\circ_{L_m}$ and $\Gamma \subseteq R_{m,w}$ (i.e. $w\beta<0$ for all $\beta\in \Gamma$) then $w=\wperm w_{L_m}$.

Note that ${}_{L_i}W^\circ_{L_m}=\wperm \ {}_{L_m}W^\circ_{L_m}$ and that ${}_{L_m}W^\circ_{L_m}$ is a group isomorphic to $\W_m\times \Z_2$. In fact, setting $\alpha_0=2e_{m+1}^m\in \Delta_{L_m}$ in the notation of \S\ref{ss sing} we have
\[
{}_{L_m}W^\circ_{L_m}=\{\tau'\set s_{\alpha_0}^\epsilon: \tau\in S_m,\ \set\in \Xi_m, \ \epsilon\in \Z_2\}
\] 
where $\tau'\in S_ {m+2n}$ extends $\tau$ to $[1,m+2n]$ by acting as the identity on $[m+1,m+2n]$ and $s_{\alpha_0}\in W(L_m)$ is the elementary symmetry associated with $\alpha_0$. 

Note that  for $w=\wperm\tau'\set s_{\alpha_0}^\epsilon\in {}_{L_i}W^\circ_{L_m}$ we have $\alpha_0\in R_{m,w}$ if and only if $\epsilon=1$ and in this case $w\alpha_0=-2e_{i+1}^i$ and that for $l\in [1,m]$ we have
\[
we_{l}^m=\begin{cases} e_{\wperm'\tau(l)}^i & l\not\in\set\\ -e_{\wperm'\tau(l)}^i & l\in\set \end{cases} \ \ \ \text{where} \ \ \ \wperm'\in S_{m+1},\ \wperm'(j)=\begin{cases} j & j\in [1,i]\\ j+1 & j\in [i+1,m] \\i+1 & j=m+1.\end{cases}
\]
It follows that $2e_m^m\in R_{m,w}$ if and only if $m\in \set$ and in this case $w(e_m^m)=-e_{\wperm'\tau(m)}^i$ and by reverse induction for every $j\in [1,m-1]$
if $j+1\in \set$ then $e_j^m-e_{j+1}^m\in R_{m,w}$ if and only if $j\in \set$ and $\wperm'\tau(j)<\wperm'\tau(j+1)$. Note further that $\wperm'$ is increasing on $[1,m]$. It follows that $\Gamma \subseteq R_{m,w}$ if and only if $\epsilon=1$, $\set=[1,m]$ and $\tau(1)<\cdots <\tau(m)$. The latter inequalities are equivalent to $\tau=e$. 
Since $w_{L_m}=\set_0s_{\alpha_0}$ where $\set_0=[1,m]\in \Xi_m$ the lemma follows.
\end{proof}

\begin{corollary}\label{cor simple pole}
With $q(\lambda)$, $\lambda\in\aaa_{L_m,\C}^*$ as in Lemma \ref{lem hol cusp} we have that $q(\lambda)E_m(\varphi,\lambda)$
is holomorphic at $\lambda=\lambda_0$ for $\varphi\in I_m(\pi)$. 
In particular,
\begin{equation}\label{eq res as lim}
\E_m(\varphi)=\lim_{\lambda\to\lambda_0}q(\lambda)E_m(\varphi,\lambda),\ \ \ \varphi\in I_m(\pi).
\end{equation}

 \end{corollary}
\begin{proof}
We argue as in \cite[Corollary 3.4]{MR946349} in order to claim that singularities of meromorphic families of automorphic forms are determined by singularities of their cuspidal components in the sense of \cite[\S I.3.4]{MR1361168}. The argument of Kudla and Rallis is based on a result of Langlands (see \cite[Proposition I.3.4]{MR1361168}) that implies that an automorphic form with zero cuspidal component along any standard parabolic identically equals zero. 

Let $E(\lambda)=q(\lambda)E_m(\varphi,\lambda)$. For $Q=L\ltimes V\in\P_\std$ it follows from \eqref{eq consterm cusp} that the cuspidal component of $E(\lambda)$ along $Q$ is zero unless $Q=Q_i$ for some $i\in [0,m]$. 
Furthermore,
\[
C_{G,Q_i}E(\lambda)=\sum_{w\in {}_{L_i} W^\circ_{L_m}}q(\lambda)(M(w,\lambda)\varphi)_{w\lambda}.
\]
It follows from Lemma \ref{lem hol cusp} that each summand on the right hand side and hence $C_{G,Q_i}E(\lambda)$ is holomorphic at $\lambda=\lambda_0$. 
The corollary now follows from the argument of Kudla and Rallis.
\end{proof}

 \subsection{Square integrability}\label{ss sqrint} A cuspidal exponent of an automorphic form $\phi\in \Aut_G$ along $Q$ is an exponent of the projection of $C_{G,Q}\phi$ to the cuspidal component $\Aut_Q^\circ$ of $\Aut_Q$. We recall the square integrability criterion \cite[\S I.4.11]{MR1361168}: an automorphic representation of $G(\A)$ is square integrable if each of its cuspidal exponents is negative in the following sense. A cuspidal exponent $\lambda$ along $P\in \P_\std$ is negative if 
 \[
 \sprod{\Re\lambda}{\varpi}<0\ \ \ \text{for all} \ \ \varpi\in \hat\Delta_P^\vee.
 \]
 To make this more explicit we recall that for $P=M\ltimes U=P_{(n_1,\dots,n_k;r)}$ we have 
 \[
 \hat\Delta_P^\vee=\{\varpi_{\nu_j}:j\in[1,k]\}
 \]
 where $\nu_j=\sum_{h=1}^j n_h$ and for $\lambda=(\lambda_1^{(n_1)},\dots,\lambda_k^{(n_k)},0^{(r)})\in \aaa_M^*$
 we have
 \[
\sprod{\lambda}{\varpi_{\nu_j}}= c_j\sum_{h=1}^jn_h\lambda_h,\ \ \ j\in[1,k],\ \ \ \text{where}\ \ \ c_j=c_j(P)=\begin{cases} \frac12 & j=k \text{ and } r=0\\ 1 & \text{otherwise}.\end{cases}
\] 
Indeed, if $j<k$ or $r>0$ then the element of $\Delta_P$ corresponding to $j$ is the restriction of the simple root $e_{\nu_j}-e_{\nu_j+1}$ and $\varpi_{\nu_j}=e_1+\cdots+e_{\nu_j}$, whereas if $r=0$ the element of $\Delta_P$ corresponding to $j=k$ is the restriction of the long simple root $2e_N$ and $\varpi_N=\frac12(e_1+\cdots+e_N)$.

In particular, if $\Re\lambda_1<0$ and $\Re\lambda_i\le 0$ for all $i\in [2,k]$ then $\lambda$ is negative.

\begin{proposition}\label{prop sqrint}
The space
 $\E^m(\pi)$ is an irreducible discrete automorphic representation of $G(\A)$. Furthermore, $-\lambda_0$ is an exponent of $\E^m(\pi)$ along $Q_m$.
\end{proposition}
\begin{proof}
Our first step is to show that $\varphi\mapsto C_{G,Q_m}\E_m(\varphi)$ is a non-zero linear map from $I_m(\pi)$ to the eigenspace $\Aut_{Q_m,-\lambda_0}$ of the twisted action of $A_{L_m}$. In particular,  $\E^m(\pi)\ne 0$ and $-\lambda_0$ is an exponent of $\E^m(\pi)$ along $Q_m$. By \eqref{eq consterm cusp}  and Lemma \ref{lem hol cusp}, in its notation, we have
\[
C_{G,Q_m}\E_m(\varphi)=\lim_{\lambda\to \lambda_0}q(\lambda)C_{G,Q_m}E_m(\varphi,\lambda)=(M_{-1}(w_{L_m})\varphi)_{-\lambda_0}
\]
where 
\[
M_{-1}(w_{L_m})\varphi=\lim_{\lambda\to \lambda_0}q(\lambda)M(w_{L_m},\lambda)\varphi. 
\]
It therefore suffices to show that
$\lim_{\lambda\to \lambda_0}q(\lambda)M(w_{L_m},\lambda)\varphi(e)\ne 0$ for some $\varphi\in I_m(\pi)$. We argue as in \cite[\S2.5]{MR2848523}. Write $\lambda=(\lambda_1,\dots,\lambda_m,t^{(2n)})\in\aaa_{L_m,\C}^*$. For every place $v$ of $F$ let
\begin{multline*}
f_v(\lambda)=\frac{L(\pi_v,t) L(\pi_v,\wedge^2,2t)}{L(\pi_v,t+1)L(\pi_v,\wedge^2,2t+1)} \prod_{i=1}^m \frac{\zeta_{F_v}(\lambda_i) L(\pi_v,\lambda_i-t)L(\pi_v,\lambda_i+t)}{\zeta_{F_v}(\lambda_i+1) L(\pi_v,\lambda_i-t+1)L(\pi_v,\lambda_i+t+1)} \times \\
\prod_{1\le i<j\le m} \frac{\zeta_{F_v}(\lambda_i-\lambda_j)\zeta_{F_v}(\lambda_i+\lambda_j)}{\zeta_{F_v}(\lambda_i-\lambda_j+1)\zeta_{F_v}(\lambda_i+\lambda_j+1)}.
\end{multline*}
Let $S$ be a finite set of places containing all archimedean places and such that $\pi_v$ is unramified for all $v\not\in S$ and set $f^S(\lambda)=\prod_{v\not\in S} f_v(\lambda)$. By the Gindikin-Karpelevich formula for $v\not\in S$ and an unramified vector $\varphi_v\in I_m(\pi)_v$ the standard local intertwining operator satisfies $M_v(w_{L_m},\lambda)\varphi_v=f_v(\lambda)\varphi_v$; here the factors of $f_v(\lambda)$ correspond to the roots $2e_{m+k}$, $e_{m+k}+e_{m+l}$, $2e_i$, $e_i\pm e_{m+k}$, $e_i-e_j$ and $e_i+e_j$ ($i<j\le m$, $k<l\le 2n$) in the unipotent radical $V_m$, and we used that $\pi$ is self-dual.
By our assumptions on $\pi$ the function $
q(\lambda)f^S(\lambda)$ is holomorphic and non-zero at $\lambda=\lambda_0$.
Indeed, recall that $q(\lambda)=(t-\frac12)(\lambda_m-1)\prod_{i=1}^{m-1}(\lambda_i-\lambda_{i+1}-1)$ and that $\lambda_0=(m,\dots,1,(\frac12)^{(2n)})$, so that at $\lambda=\lambda_0$ we have $t=\frac12$ and $\lambda_i=m+1-i$. The factor $L^S(\pi,\wedge^2,2t)$ has a simple pole at $t=\frac12$, cancelled by $(t-\frac12)$; this is where the assumption $\Res_{s=1}L(\pi,\wedge^2,s)\ne 0$ is used. The factor $\zeta_F^S(\lambda_i)$ has a simple pole only for $i=m$, where $\lambda_m=1$, cancelled by $(\lambda_m-1)$. The factor $\zeta_F^S(\lambda_i-\lambda_j)$ has a simple pole only for $j=i+1$, cancelled by $(\lambda_i-\lambda_{i+1}-1)$. Finally, $L^S(\pi,\lambda_m-t)=L^S(\pi,t)=L^S(\pi,\frac12)\ne 0$ by our second assumption on $\pi$, while every remaining numerator and denominator is evaluated at a point of real part at least $\frac32$, where the corresponding $L$-function is holomorphic and non-zero.
For $v\in S$ there exists $\varphi_v\in I_m(\pi)_v$ such that $M_v(w_{L_m},\lambda)\varphi_v(e)$ is holomorphic and non-zero at $\lambda=\lambda_0$. Since $M(w_{L_m},\lambda)=\otimes _vM_v(w_{L_m},\lambda)$, for $\varphi=\otimes_v\varphi_v\in I_m(\pi)$ with $\varphi_v$ unramified for $v\not\in S$ we have that $\lim_{\lambda\to \lambda_0}q(\lambda)M(w_{L_m},\lambda)\varphi(e)=[\lim_{\lambda\to\lambda_0}q(\lambda)f^S(\lambda)]\prod_{v\in S}M_v(w_{L_m},\lambda_0)\varphi_v(e)\ne 0$. 
We deduce that
\[
M_{-1}(w_{L_m}):I_m(\pi,\lambda_0)\rightarrow I_m(\pi,-\lambda_0)
\]
is a well defined, non-zero intertwining operator. Furthermore, its image is irreducible. Indeed, at every place $v$ of $F$, the representation $I_m(\pi,\lambda_0)_v$ is a standard module. Here we use the fact that $\pi_v$ is irreducible unitary and generic and the classification of such representations of general linear groups \cite{MR870688, MR2537046}.   It therefore admits a unique irreducible quotient, in fact a unique semi-simple quotient that is isomorphic to the image of the local intertwining operator $M(w_{L_m},\lambda_0)$. Denote this image by $\Pi_v$. Clearly, the image of $M_{-1}(w_{L_m})$ is isomorphic to $\Pi=\otimes _v \Pi_v$ and is therefore irreducible.

Next we show that $\E^m(\pi)$ is contained in the discrete spectrum by applying the square-integrability criterion.
By taking residue, it follows from \eqref{eq consterm cusp} that the cuspidal component of $\E^m(\varphi)$ along $Q$ is zero unless $Q=P_{(1^{(i)},2n,1^{(m-i)};0)}$ for some $i\in [0,m]$. Fix $i\in [0,m]$, and let $Q=P_{(1^{(i)},2n,1^{(m-i)};0)}$.  
It follows from Lemma \ref{lem hol cusp}, in its notation, that $C_{G,Q}\E^m(\varphi)$ has at most one exponent, namely
\[
\wperm w_{L_m}\lambda_0=(-m,\dots,-(m+1-i),(-\frac12)^{(2n)},-(m-i),\dots,-1).
\]
This is a negative exponent and by the square integrability criterion, $\E^m(\pi)$ lies in the discrete spectrum.

It remains to show the irreducibility of $\E^m(\pi)$. We argue as in \cite[\S2.5]{MR2848523}. 
By square-integrability, $\E^m(\pi)$ is a unitary quotient of $I_m(\pi,\lambda_0)$ and in particular is semi-simple. Since $\Pi_v$ is the unique semi-simple quotient of $I_m(\pi,\lambda_0)_v$ at every place $v$ of $F$ we conclude that $\E^m(\pi)$ is isomorphic to $\Pi$. (Alternatively, it easily follows from Lemma \ref{lem hol cusp} that the map $\E^m:I_m(\pi,\lambda_0)\rightarrow \Aut_G$ factors through $M_{-1}(w_{L_m})$. Hence $\E^m(\pi)$ is a non-zero quotient of $\Pi$ hence isomorphic to $\Pi$.)
This concludes the proof of the proposition.

\end{proof}
 \begin{remark}
Theorem \ref{thm main} provides an independent proof that $\E^m(\pi)$ is non-zero. 
 \end{remark}

\subsection{On the Ginzburg--Rallis--Soudry construction}\label{ss gsr}
Consider here the case $m=0$ so that $G=\Sp_{2n}$ and $\lambda_0=((\frac12)^{(2n)})$.
It follows from \eqref{eq consterm cusp} that
\[
C_{G,Q}E^0(\varphi,\lambda)=0
\]
whenever $Q$ is a proper standard parabolic subgroup of $G$ different from $P_{(2n;0)}$ and
\[
C_{G,P_{(2n;0)}}E^0(\varphi,\lambda)=\varphi_\lambda+(M(w,\lambda)\varphi)_{w\lambda}
\]
where $w$ is represented in $G$ by $\tilde w=\sm{}{I_{2n}}{-I_{2n}}{}$. In particular, the results of \cite{MR1740991} imply that $M(w,\lambda)|_{I_0(\pi)}$ has a simple pole at $\lambda=\lambda_0$. Let $M_{-1}(w):I_0(\pi,\lambda_0)\rightarrow I_0(\pi,-\lambda_0)$ denote the residue defined by $M_{-1}(w)\varphi=\lim_{s\to \frac12} (s-\frac12)M(w,(s^{(2n)}))\varphi$ and recall that $\E_0(\varphi)=\lim_{s\to \frac12} (s-\frac12)E_0(\varphi,(s^{(2n)}))$.
It follows that 
\[
C_{G,Q}\E_0(\varphi)=0
\]
whenever $Q$ is a proper standard parabolic subgroup of $G$ different from $P_{(2n;0)}$ and 
\[
C_{G,P_{(2n;0)}}\E_0(\varphi)=(M_{-1}(w)\varphi)_{-\lambda_0},\ \ \ \varphi\in I_0(\pi)
\]
(the identity term $\varphi_\lambda$ is holomorphic at $\lambda=\lambda_0$ and does not contribute to the residue).
Furthermore, since $\delta_{P_{(2n;0)}}(\diag(m,m^*))=\abs{\det m}^{2n+1}$, the function
\[
m\mapsto \abs{\det m}^{-n} C_{G,P_{(2n;0)}}\E_0(\varphi)(\diag(m,m^*)g), \ \ \ m\in \GL_{2n}(\A) 
\] 
lies in the space of $\pi$. 
In particular, $C_{G,P_{(2n;0)}}:\E^0(\pi)\rightarrow \Aut_{P_{(2n;0)},((-\frac12)^{(2n)})}(\Sp_{2n})$.

\subsection{A different realization of $\E^m(\pi)$}\label{ss other res} 
 We return to the case $m\in \N$. Let $P=M\ltimes U=P_{(1^{(m)};2n)}$, 
$
 I^m(\pi)=I_P^G(\triv_{\GL_1^m}\otimes \E^0(\pi))$ and $E^m(\xi,\mu)=E_P^G(\xi,\mu)$ for $\xi\in I^m(\pi)$ and $\mu\in \aaa_{M,\C}^*$. 
 Also write $I^m(\pi,\mu)=I_P^G(\triv_{\GL_1^m}\otimes \E^0(\pi),\mu)$ and let 
 \[
 \mu_0=(\lambda_0)_P=(m,\dots,1,0^{(2n)})\in \aaa_M^*.
 \]
 Set $\theta=(0^{(m)},(\frac12)^{(2n)})\in \aaa_{L_m}^*$. For $\mu\in \aaa_{M,\C}^*$, the map 
 \[
 R_\mu:\varphi\mapsto \lim_{t\to \frac12}(t-\frac12)E_{Q_m}^P(\varphi,\mu+(0^{(m)},t^{(2n)})): I_m(\pi,\mu+\theta)\rightarrow I^m(\pi,\mu)
 \]
 is a surjective map of representations (note that $\mu+\theta$ is the value at $t=\frac12$ of the parameter of the Eisenstein series on the left hand side). Since $E_{Q_m}^P(\varphi,\lambda+\mu)=E_{Q_m}^P(\varphi,\lambda)_\mu$ for $\mu\in \aaa_{P,\C}^*$, we have $R_\mu\varphi=(R\varphi)_\mu$ where $R=R_0:I_m(\pi)\rightarrow I^m(\pi)$ is a surjective linear map that we refer to as the residue map. 
 By taking the limit in \eqref{eq res as lim} in stages, first $\lim_{t\to \frac12}$ and then $\lim_{(\lambda_1,\dots,\lambda_m)\to (m,\dots,1)}$, we observe that the map $\E_m:I_m(\pi,\lambda_0) \rightarrow \E^m(\pi)$ factors through  $R_{\mu_0}$ (note that $\lambda_0=\mu_0+\theta$). 
 We deduce that
 \[
 \E^m(\pi)=\{\E^m(\xi):\xi\in I^m(\pi)\}
 \]
 where
 \begin{equation}\label{eq new real}
\E^m(\xi)=\lim_{\mu\to \mu_0} [\prod_{\alpha\in \Delta_M} \sprod{\mu-\mu_0}{\alpha^\vee}]E^m(\xi,\mu). 
 \end{equation}
Indeed, note that $\Delta_M=\{\alpha_1,\dots,\alpha_m\}$ where for $\mu=(\mu_1,\dots,\mu_m,0^{(2n)}) \in \aaa_{M,\C}^*$ we have
\[
\sprod{\mu}{\alpha_m^\vee}=\mu_m\ \ \ \text{and}\ \ \ \sprod{\mu}{\alpha_j^\vee}=\mu_j-\mu_{j+1},\ \ \ j\in [1,m-1].
\]
That is, $\Delta_M$ is identified with $\Gamma\setminus\{\beta_{m+1}\}$ in the notation of Lemma \ref{lem hol cusp}.

For $Q=L\ltimes V\in\P_\std$ let
\[
{}_LW_M'=\{w\in {}_LW_M: w(L_m)\subseteq L\}.
\]
As a consequence of \S\ref{ss gsr} for $w\in {}_LW_M$ and $\varphi\in I^m(\pi)$ we have $C_{P,P_w}\varphi=0$ unless $w\in {}_LW_M'$. 
Note that $W_M(L_m)=\{e,w_{L_m}^M\}$ (see \S\ref{s elementary sym}). We apply freely the notation of \S\ref{ss sing}. We have $w_{L_m}^M=s_{\alpha_0}$ is the elementary symmetry associated to $\alpha_0=2e_{m+1}^m\in \Delta_{L_m}$. For $i\in [0,m]$ we observe that 
\[
{}_{L_i}W'_M=\{w\in  {}_{L_i}W^\circ_{L_m}: \alpha_0\not\in R_{m,w}\}.
\]
\begin{lemma}\label{lem sing indm}
For $i\in [0,m]$ and $w\in {}_{L_i}W'_M$ we have
\[
\left[\prod_{\alpha\in \Delta_M,\ w\alpha<0}  \sprod{\mu-\mu_0}{\alpha^\vee}\right]M(w,\mu)C_{P,P_w}|_{I^m(\pi)}
\]
is holomorphic at $\mu=\mu_0$.
\end{lemma}
\begin{proof}
Note that $P_w=Q_m$ for all $w\in {}_{L_i}W'_M$. Consequently, it follows from \S\ref{ss gsr} and \eqref{eq mult int} that $M(w,\mu) \circ C_{P,P_w}\circ R=\lim_{t\to \frac12} (t-\frac12)M(ws_{\alpha_0}, \mu+(0^{(m)},t^{(2n)}))$ as operators on $I_m(\pi)$, where $R:I_m(\pi)\rightarrow I^m(\pi)$ is the surjective residue map of \S\ref{ss other res}.
The lemma now follows from Lemma \ref{lem hol cusp}. Indeed, note that $s_{\alpha_0}\alpha=\alpha$ for all $\alpha\in \Delta_M$.
\end{proof}
\begin{corollary}\label{cor consterm hol}
 Let $Q=L\ltimes V\in \P_\std$ and $w\in {}_LW_M'$ and set 
 \[
 \Delta_M(Q,w)=\{\alpha\in \Delta_M: \text{ either }w\alpha<0 \text{ or }(w\alpha)_Q=0\}.
 \]
Then  for $\xi\in I^m(\pi)$ we have that
\[
\left[\prod_{\alpha\in \Delta_M(Q,w)} \sprod{\mu-\mu_0}{\alpha^\vee}\right]E_{Q_w}^Q(M(w,\mu)C_{P,P_w}\xi,w\mu)
\] 
is holomorphic at $\mu=\mu_0$. 
\end{corollary}
\begin{proof}
As in Corollary \ref{cor simple pole} we apply the argument in \cite[Corollary 3.4]{MR946349} based on \cite[Proposition I.3.4]{MR1361168}. 

Let $q_w(\mu)=\prod_{\alpha\in \Delta_M(Q,w)} \sprod{\mu-\mu_0}{\alpha^\vee}$ and
$E(\mu)=q_w(\mu)E_{Q_w}^Q(M(w,\mu)C_{P,P_w}\xi,w\mu)$. For $Q'=L'\ltimes V'\in\P_\std$ such that $Q'\subseteq Q$, based on \eqref{eq consterm cusp}, we observe that the cuspidal component of $E(\mu)$ along $Q'$ equals zero unless $Q'=Q_i$ for some $i\in [0,m]$ (such that $Q_i\subseteq Q$) and furthermore, that for such $i$ we have
\[
C_{Q,Q_i}E(\mu)=\sum_{\sigma}q_w(\mu)(M(\sigma w,\mu)C_{P,Q_m}\xi)_{\sigma w\mu}
\]
where the sum is over $\sigma\in {}_{L_i}(W_L)_{L_w}$ such that $\sigma w(L_m)=L_i$ and in particular, $\sigma w\in {}_{L_i}W_M'$ (so that $P_{\sigma w}=Q_m$).
If $\alpha\in \Delta_M$ is such that $w\alpha>0$ and $(w\alpha)_Q\ne 0$ then also $\sigma w \alpha>0$ (since $\sigma\in W_L$ we have $(\sigma w\alpha)_Q=(w\alpha)_Q$).
It follows from Lemma \ref{lem sing indm} that each summand on the right hand side and hence $C_{Q,Q_i}E(\mu)$ is holomorphic at $\mu=\mu_0$.  The corollary follows using the argument of Kudla and Rallis.
\end{proof}


\section{The mixed truncation operator}\label{sec mixed trunc}
Let $H$ be a reductive subgroup of $G$. 
In \cite{MR4411860}, based on previous ideas of \cite{MR1625060} and \cite{MR2010737} when $H$ is a Galois symmetric subgroup of $G$, Zydor introduced a mixed truncation operator that maps the space 
$\Aut_G$ of automorphic forms on $[G]$ to the space of rapidly decreasing functions on $[H]$.
\subsection{Definition of the mixed truncation operator}
Fix a minimal parabolic subgroup $P_0^H$ of $H$ contained in $P_0$ and let $\F=\F^G(P_0^H)$ be the set of semi-standard parabolic subgroups of $G$ that contain $P_0^H$.
We explicate the definition of the mixed truncation under certain assumptions that will automatically be satisfied in all cases considered in this work. Our assumptions are:
\begin{itemize}
\item $T_0\subseteq P_0^H=P_0\cap H$ and
\item $K_H:=K\cap H(\A)$ is a maximal compact subgroup of $H(\A)$ in good position with respect to $P_0^H$.
\end{itemize}
For $T\in \aaa_0$ and every $Q\in \F$ the mixed truncation operator $\Lambda^{T,Q}\phi$ is defined for $\phi\in \Aut_Q$ and $h\in (Q\cap H)(F) \bs H(\A)$ as follows:
\[
\Lambda^{T,Q}\phi(h)=\sum_{P\in \F,\ P\subseteq Q}(-1)^{\dim \aaa_P^Q}\sum_{\delta\in (P\cap H)(F)\bs (Q\cap H)(F)} \hat\tau_P^Q(\HC_0(\delta h)^Q-T^Q)C_{Q,P}\phi(\delta h).
\]
Here $\hat\tau_P^Q$ is the characteristic function of $X\in \aaa_0^Q$ such that $\sprod{X}{\varpi}>0$ for every $\varpi\in \hat\Delta_P^Q$.
Write $\Lambda^T=\Lambda^{T,G}$. It follows from \cite[Theorem 3.9]{MR4411860} that there exists $c>0$ such that if $T\in \aaa_0$ is such that $\sprod{T}{\alpha}>c$ for all $\alpha\in \Delta_0$ then for all $\phi\in\Aut_G$ the function $\Lambda^T\phi$ is rapidly decreasing on $[H]^{1,G}$ and in particular integrable on $[H]^{1,G}$. We say that such $T$ is positive enough.
\subsection{The regularization of period integrals}\label{ss reg}
For $\xi\in \aaa_{H,\C}^*\subseteq \aaa_{0,\C}^*$\footnote{the inclusion is by extending $\xi$ to be trivial on $\aaa_0^H$} let $\Aut_G^{\xi-\reg}$ be the subspace of $\phi\in \Aut_G$ such that $(\lambda+\xi+\rho_P-2\rho_{P\cap H})_P^G\ne 0$ whenever $P\in \F$ is a maximal parabolic subgroup of $G$ and $\lambda$ is an exponent of $\phi$ along $P$. 
For $\xi=0$ set 
\[
\Aut_G^{H-\reg}=\Aut_G^{0-\reg}.
\]
In order to verify the regularity condition in practice it is convenient to reformulate it in terms of standard parabolic subgroups. For $P=M\ltimes U\in \P_\std$ let 
\[
\Sigma_P=\{\sigma\in [W/W_M]:\sigma(P)\in \F\} 
\]
and for $\sigma\in \Sigma_P$ set
\[
\rho_{P,\sigma}=\rho_P-2\rho_{P\cap \sigma^{-1}H\sigma}.
\]

Then based on \eqref{eq nonstd const} we have
\begin{multline}\label{eq reg cond}
\phi\in \Aut_G^{H-\reg} \  \text{if and only if} \ (\lambda+\rho_{P,\sigma})_P^G\ne 0 \  \text{for all}  \ P\in \P_\std \\   \text{such that} \ \dim\aaa_P^G=1,\ \sigma\in \Sigma_P \ \text{and exponents}  \ \lambda  \ \text{of} \ \phi  \ \text{along} \ P.
\end{multline}

Let $\chi$ be a character of $[H]$ 
and let $\xi\in \aaa_{H,\C}^*$ be such that $\chi(a)=e^{\sprod{\xi}{\HC_0(a)}}$, $a\in A_H$. The regularized period $\P_{H,\chi}(\phi)$ is defined for $\phi\in \Aut_G^{\xi-\reg}$ as follows. 
For $T$ positive enough, the convergent integral
\[
\int_{[H]^{1,G}}\chi(h)\Lambda^T\phi(h) \ dh
\]
is a polynomial exponential function of $T$, that is, it is a finite sum of the form $\sum_\lambda p_\lambda(T)e^{\sprod{\lambda}{T}}$ for some $\lambda$'s in $\aaa_{0,\C}^*$ and polynomials $p_\lambda$ on $\aaa_0$ and furthermore, for $\phi\in \Aut_G^{\xi-\reg}$, $p_0$ is independent of $T$. Set $\P_{H,\chi}(\phi)=p_0$. It defines an $(H(\A)^{1,G},\chi^{-1})$-equivariant linear form on $\Aut_G^{\xi-\reg}$. Furthermore, if the integral 
\[
\int_{[H]^{1,G}}\chi(h)\phi(h) \ dh
\]
converges then $\phi\in \Aut_G^{\xi-\reg}$ and the period integral equals the regularization $\P_{H,\chi}(\phi)$ (\cite[Theorem 4.5]{MR4411860}).

\subsection{Zydor's formula for the regularized period integral}
We explicate the formula \cite[(4.3)]{MR4411860} under some special assumptions. Let $\chi$ and $\xi$ be as in \S\ref{ss reg}. Let $\phi\in \Aut_G^{\xi-\reg}$ and assume further that for every $Q=L\ltimes V\in \P_\std$ there is an expansion
\[
C_{G,Q}\phi=\sum_i \phi_{Q,i}
\]
and exponents $\lambda[Q,i]\in \aaa_{L,\C}^*$ such that $\phi_{Q,i}$ is an eigenfunction for the twisted action of $A_L$ with eigenvalue $\lambda[Q,i]$. 

In what follows, our evaluation of the Fourier transforms of cones (see \cite[\S1.4]{MR4411860}) is based on \cite[Chapter II]{MR1625060}. Denote by $v_Q$ the covolume of the lattice of coweights $\oplus_{\varpi\in \hat\Delta_Q^\vee}\Z \varpi$ in $\aaa_Q^G$ (with respect to the measure of \S\ref{auto-form-11}). Based on \eqref{eq nonstd const}, the explication of \cite[(4.3)]{MR4411860} gives
\begin{multline}\label{eq Zydor}
\P_{H,\chi}(\phi)=\sum_{Q\in \P_\std} (-1)^{\dim\aaa_Q^G} \ v_Q\sum_{i}\sum_{\sigma\in \Sigma_Q}\frac{e^{\sprod{\sigma^{-1}T}{(\lambda[Q,i]+\rho_{Q,\sigma}+\sigma^{-1}\xi)_Q^G}}}{\prod_{\varpi\in \hat\Delta_Q^\vee}\sprod{\lambda[Q,i]+\rho_{Q,\sigma}+\sigma^{-1}\xi}{\varpi} }\times \\
\int_{K_H} \int_{[L\cap \sigma^{-1} H\sigma]^{1,L}} \chi(\sigma\ell\sigma^{-1}k)\delta_{Q\cap \sigma^{-1} H\sigma}^{-1}(\ell) \Lambda^{\sigma^{-1}T,Q}\phi_{Q,i}(\ell\sigma^{-1} k)\ d\ell \ dk.
\end{multline}
Here $\chi(\sigma a\sigma^{-1})=e^{\sprod{\sigma^{-1}\xi}{\HC_0(a)}}$ for $a$ in the split component of $Q\cap \sigma^{-1}H\sigma$. 
In \S\ref{vanishing-22} and \S\ref{periods of res reps-11} the group $H$ is semisimple, so that $\xi=0$ and the terms involving $\xi$ are absent.

\subsection{Multi-residues and truncation}
We introduce the following notation for the period of a truncated Eisenstein series. For $P=M\ltimes U\in \P_{\sstd}$, $\pi$ an automorphic representation of $M(\A)$, $\varphi\in I_P^G(\pi)$, $\lambda\in\aaa_{M,\C}^*$ and a character $\chi$ of $[H]^{1,G}$ set
\begin{equation}\label{eq truncint}
\P^{G,T}_{H,\chi}(\varphi,\lambda)=\int_{[H]^{1,G}}\chi(h)\Lambda^TE_P^G(h,\varphi,\lambda)\ dh.
\end{equation} 
If $\chi$ is the trivial character it is often dropped from the notation. Suppose that $\lambda=\lambda_0$ is a singularity of the Eisenstein series $E_P^G(h,\varphi,\lambda)$ and let
\[
\E(\varphi)=\Res_{\lambda=\lambda_0}E_P^G(\varphi,\lambda)
\] 
be a multi-residue. Based on Zydor's bounds for the mixed truncation operator $\Lambda^T$ on $\Aut_G$, Arthur's argument in \cite[Lemma 3.1]{MR650368}, showing that a multi-residue commutes with a convergent integral and with the truncation operator, also gives
\begin{equation}\label{eq com res}
\int_{[H]^{1,G}}\chi(h)\Lambda^T\E(h,\varphi) \ dh=\Res_{\lambda=\lambda_0}\P^{G,T}_{H,\chi}(\varphi,\lambda).
\end{equation}
This principle will be used repeatedly in our work.

\section{Vanishing of certain linear periods}\label{vanishing-11}
Let $N\in\N$. In this section set $G=\GL_N$ and let $P_0=B_N^{\GL}$ be the Borel subgroup of upper-triangular matrices
in $G$. Fix a composition $N=a+b$ with $a,\,b\in \Z_{\ge 0}$ and let 
\[
H=M_{(a,b)}=\GL_a\times \GL_b.
\]
Our main result in this section is the following.
\begin{proposition}\label{prop GL vanishing}
Let $\alpha=(n_1,\dots,n_k)\in \Comp_N$, $P=M\ltimes U=P_\alpha\in \P_\std$ and $\pi$ be an irreducible, cuspidal automorphic representation of $M(\A)$. If 
\[
\max(1,2\min(a,b))<\max(n_1,\dots,n_k)
\] 
then $\P_{H,\chi}^{\GL_N,T}$ is identically zero on $I_P^G(\pi)\times \aaa_{M,\C}^*$ for any  character $\chi$ of $[H]$.
\end{proposition}

\subsection{A local generic vanishing result for linear periods}\label{ss localvan}

Assume here that $F$ is a $p$-adic field and $q$ the size of its residue field. Let $B=T\ltimes U_0$ be the standard Borel subgroup of $\GL_N(F)$ with unipotent radical $U_0$ consisting of the upper triangular unipotent matrices and diagonal torus $T$. For $\lambda=(\lambda_1,\dots,\lambda_N)\in  \C^N$  let $I(\lambda)$ be the unramified principal series representation of $\GL_N(F)$ parabolically induced from $B$ and the character $\chi_\lambda(\diag(t_1,\dots,t_N))=\prod_{i=1}^N\abs{t_i}^{\lambda_i}$, $t_i\in F^*$, $i\in[1,N]$. \begin{lemma}\label{lem local vanishing} Assume that $a\ne b$ and let $\chi$ be a character of $H(F)=\GL_a(F)\times \GL_b(F)$. 
There is a finite set $S\subseteq \C$ such that if $\lambda\in\C^N$ is such that $\Hom_{H(F)}(I(\lambda),\chi)\ne 0$ then $q^{\lambda_i}\in S$ for some $i\in [1,N]$.
\end{lemma}
\begin{proof}
Let $X$ be the $\GL_N(F)$-conjugacy class of $\epsilon=\diag(I_a,-I_b)$. We identify the quotient $\GL_N(F)/H(F)$ with $X$ via the map $gH\mapsto g\epsilon g^{-1}$. The double coset space $B\bs \GL_N(F)/H$ is then identified with the $B$-conjugacy classes in $X$. It follows from \cite[Propositions 6.6 and 6.8]{MR1215304} that inclusion defines a bijection from the $T$-conjugacy classes in $N_{\GL_N(F)}(T)\cap X$ to the $B$-conjugacy classes in $X$ and that these sets of conjugacy classes are finite. Note that for $w\in W$ we have that $wT\cap X$ is not empty if and only if $w$ is an involution with at least $\abs{b-a}$ fixed points. Furthermore, the centralizer of $x\in wT\cap X$ in $T$ is independent of $x$. It is the group 
\[
T_w=\{\diag(t_1,\dots,t_N): t_i=t_{w(i)},\ i\in[1,N]\}.
\] 
Fix a set of representatives $\Sigma_w$ for the set of $T$-conjugacy classes in $wT\cap X$. For $x\in \Sigma_w$ let $B_x$ be the centralizer of $x$ in $B$, $\delta_x$ be the character of $T_w$ defined by $\delta_x=\delta_B^{-\frac12}\delta_{B_x}$ and $\xi_x\in \GL_N(F)$ such that $\xi_x x\xi_x^{-1}=\epsilon$. Let $\chi^{x}$ be the character of $T_w$ defined by $\chi^x(t)=\chi(\xi_x t\xi_x^{-1})$. If it is an unramified character, then there exists $\mu(x)=(\mu_1(x),\dots,\mu_N(x))\in \C^N$ (with $w(\mu(x))=\mu(x)$) such that $\delta_x\chi^x=\chi_{\mu(x)}|_{T_w}$. Let $S$ be the set of all values $q^{\mu_i(x)}$ where $w\in W$, $x\in \Sigma_w$ is such that $\chi^x$ is unramified and $i\in[1,N]$ is such that $w(i)=i$. 

As a consequence of the geometric lemma of Bernstein and Zelevinsky (see \cite[Theorem 4.2]{MR3541705}), if $\Hom_{H(F)}(I(\lambda),\chi)\ne 0$ then there exists  an involution $w\in S_N$ and $x\in \Sigma_w$ such that $\Hom_{T_w}(\chi_\lambda,\delta_x\chi^x)\ne 0$. In particular, $\chi^x$ must be unramified and $q^{\lambda_i}=q^{\mu_i(x)}\in S$, for $i\in[1,N]$ such that $w(i)=i$ (there exists such $i$ since, by assumption, $\abs{b-a}\ge 1$). The lemma follows.
\end{proof}

For the rest of this section we go back to the global setting. 

\subsection{Global vanishing for the regularized linear periods}

Let $\chi$ be a character of $[H]$ and $\xi\in \aaa_{H,\C}^*$ be such that $e^{\sprod{\xi}{\HC_0(h)}}=\chi(h)$, $h\in A_H$.
Let $P=M\ltimes U\in \P_\std$ and let $\pi$ be an irreducible, cuspidal automorphic representation of $M(\A)$. Then $E_P^G(\varphi,\lambda)$ is in $\Aut_G^{\xi-\reg}$  for $\varphi\in I_P^G(\pi)$ and $\lambda\in \aaa_{M,\C}^*$ away from finitely many hyperplanes. The regularized period $\P_{H,\chi}(E_P^G(\varphi,\lambda))$ is therefore defined for generic $\lambda$ and defines a meromorphic function in $\lambda$ (as in \cite[Theorem 8.4.1 (4)]{MR2010737}).

\begin{lemma}\label{lem van per gl} 
Assume $a\ne b$ and $N\ge 2$. In the above notation we have
\[
\P_{H,\chi}(E(\varphi,\lambda))=0,\ \ \ \varphi\in I_P^G(\pi),\ \lambda\in\aaa_{M,\C}^*.
\]
\end{lemma}
\begin{proof}
If $M=G$ then $E(\varphi,\lambda)=\varphi_\lambda$ is a cusp form on $\GL_N$ with $N\ge 2$; the vanishing follows from \cite{MR1233493} when $a,b\ge 1$, and when $\min(a,b)=0$ (so that $H=G$) it follows from the orthogonality of cusp forms to the one-dimensional automorphic representations of $\GL_N(\A)$. For the rest of the proof we assume that $M\ne G$. Since $\P_{H,\chi}(E(\varphi,\lambda))=\P_{H,\chi}(E(\varphi,\lambda+\mu))$ for $\mu\in \aaa_{G,\C}^*$ it suffices to prove the vanishing for $\lambda\in(\aaa_{M,\C}^G)^*$.
It follows from \cite[Theorem 4.1 (4)]{MR4411860} that $\varphi\mapsto \P_{H,\chi}(E(\varphi,\lambda))$ is $(H(\A)^{1,G},\chi^{-1})$-equivariant on $I_P^G(\pi,\lambda)$. 
Since $H(\A)=A_GH(\A)^{1,G}$ we further conclude that for $\lambda\in(\aaa_{M,\C}^G)^*$ it is $(H(\A),\chi')$-equivariant for a character $\chi'$ of $[H]$ dependent only on $\chi$ and the central character of $\pi$ restricted to $A_G$.
Let $v$ be a finite place of $F$ such that $\pi_v$ and $\chi'_v$ are unramified. Since $\pi$ is cuspidal, $\pi_v$ is also generic. In particular, it is an unramified principal series. There exists therefore $\lambda_0\in \C^N$ such that $I_P^G(\pi,\lambda)_v=I(\lambda_0+\lambda)$ where on the right hand side we use the notation of \S\ref{ss localvan} with the local field $F_v$.
Assume by contradiction that $\P_{H,\chi}(E(\varphi,\lambda))$ is not identically zero. Then $\Hom_{H(F_v)}(I(\lambda_0+\lambda),\chi'_v)\ne 0$ for $\lambda$ in a non-empty open subset $\Omega$ of $(\aaa_{M,\C}^G)^*$. By Lemma \ref{lem local vanishing}, in its notation, for every $\lambda\in \Omega$ there exist $i\in [1,N]$ and $s\in S$ such that $q_v^{(\lambda_0+\lambda)_i}=s$. Consequently, the entire function $F(\lambda)=\prod_{i\in[1,N]}\prod_{s\in S}\big(q_v^{(\lambda_0+\lambda)_i}-s\big)$ on $(\aaa_{M,\C}^G)^*$ vanishes on the non-empty open set $\Omega$, hence identically, and therefore one of its factors vanishes identically. However, since $P$ is proper in $G$, the Levi subgroup $M$ has at least two blocks and for every $i\in [1,N]$ the function $\lambda\mapsto (\lambda_0+\lambda)_i$ is a non-constant affine function on $(\aaa_{M,\C}^G)^*$, so that $q_v^{(\lambda_0+\lambda)_i}-s$ is not identically zero. This is a contradiction. The lemma follows.
\end{proof}

\subsection{Reduced Bruhat representatives for $\GL_N$}\label{ss rwegl}
Let $P=M\ltimes U,\,Q=L\ltimes V\in\P_{\std}$. We explicate the set 
\[
{}_LW^\circ_M=\{w\in {}_LW_M:w(M)\subseteq L\}.
\]
Let $\alpha=(n_1,\dots,n_k),\ \beta=(m_1,\dots,m_t)\in\Comp_N$ be such that $M=M_\alpha$ and $L=M_\beta$. It is easy to see that 
${}_LW^\circ_M$ is parameterized by the set of partitions $(A_1,\dots,A_t)$ of $[1,k]$ into $t$ subsets such that 
\[
m_i=\sum_{j\in A_i} n_j,\ \ \ i\in [1,t].
\] 
The element $w\in {}_LW^\circ_M$ associated to such a partition is characterized by the property that it is order preserving from 
$\sqcup_{j\in A_i} [\nu_{j-1}+1,\nu_j]$ to $[\mu_{i-1}+1,\mu_i]$, $i\in [1,t]$ where we set 
\begin{equation}\label{eq interval gl}
\mu_0=\nu_0=0\ \ \ \text{and} \ \ \ \nu_j=\sum_{r=1}^j n_r,\ \ \ \mu_i=\sum_{r=1}^i m_r,\ \ \ j\in [1,k],\ i\in [1,t].
\end{equation}
Note that in this case $Q_w=P_\gamma$ where $\gamma=(\gamma_1,\dots,\gamma_t)\in \Comp_N$ is the refinement of $\beta$ where for $i\in [1,t]$ we have that 
\begin{equation}\label{eq refinement coord}
\gamma_i=(n_{j_1},\dots,n_{j_r})\ \ \  \text{if} \ \ \ A_i=\{j_1,\dots,j_r\} \ \ \  \text{with}\ \ \ j_1<\cdots<j_r.
\end{equation}

\subsection{The relevant semi-standard parabolic subgroups}\label{ss rel ss par}

Fix a composition $(a,b)$ of $N$ and let $H=M_{(a,b)}=\GL_a\times \GL_b$.
Note that 
\[
P_0^H=P_0\cap H=B_a^{\GL}\times B_b^{\GL}
\] 
is a Borel subgroup of $H$. 
We would like to have a convenient parameterization of the set $\F_{a,b}=\F_{a,b}^{\GL}=\F^G(P_0^H)$.
Clearly, 
\[
\F_{a,b}=\sqcup_{\alpha\in \Comp_N} \F_{a,b}^\alpha\ \ \ \text{where} \ \ \ \F_{a,b}^\alpha=\{\sigma(P_\alpha):\sigma\in W \text{ and }\sigma(P_\alpha)\in \F_{a,b}\}.
\]

Fix $\alpha=(n_1,\dots,n_k)\in \Comp_N$ and let $P=M\ltimes U=P_\alpha\in \P_\std$. 
The Weyl group $W_M$ may be identified with the subgroup $S_\alpha$ of $S_N$ of permutations that preserve each of the intervals 
\[
I_i=[\nu_{i-1}+1,\nu_i],\ \ \  i\in [1,k]
\] where we apply the notation of \eqref{eq interval gl} for the $\nu_i$'s.

Note that $w\mapsto (w(I_1),\dots,w(I_k))$ is a parameterization of $[W/W_M]=[S_N/S_\alpha]$ by ordered partitions of the set $[1,N]$ into $k$ subsets with corresponding sizes $(n_1,\dots,n_k)$.

Let $\Tau^\alpha_{a,b}$ be the set of tables $\tbl$ of the form
\begin{equation}\label{eq typ t}
\tbl=\begin{array}{ |c|c|c||c|} 
 \hline
 n_1 & \cdots & n_k  & N \\
 \hline\hline
a_1 & \cdots & a_k  & a \\
\hline
b_1 & \cdots & b_k  & b \\
\hline
\end{array}
\end{equation}
where $a_i,b_i \in \Z_{\ge 0}$, $i\in [1,k]$, satisfy
\begin{itemize}
\item $a_1+\cdots+a_k=a$,
\item $b_1+\cdots+b_k=b$ and
\item $a_i+b_i=n_i$, $i\in [1,k]$.
\end{itemize}
For a table $\tbl\in \Tau_{a,b}^\alpha$ as above we would like to associate an element $\sigma_{\tbl}\in [S_N/S_\alpha]$. We first introduce some further terminology. 

For two intervals $I$ and $J$ of integers we say that $J$ is consecutive to $I$ and write $I\rightarrow J$ if $I=[x,y]$ and $J=[y+1,z]$ for some integers $x-1\le y\le z$ (so that $I$, respectively $J$, is empty if $y=x-1$, respectively $z=y$). 
For $\beta=(m_1,\dots,m_t)\in \Comp_m$ and $I$ an interval of $m$ integers, the sequence $I_1\rightarrow I_2\rightarrow\cdots\rightarrow I_t$ of consecutive intervals that forms a partition of $I$ and such that $\abs{I_j}=m_j$, $j\in [1,t]$ will be referred to as the sequence of consecutive intervals associated to $I$ and $\beta$. 

Let 
\[
I_1^A\rightarrow I_1^B\rightarrow \cdots\rightarrow I_k^A\rightarrow I_k^B
\]
respectively 
\[
A_1 \rightarrow \cdots \rightarrow A_k\rightarrow B_1 \rightarrow \cdots \rightarrow B_k,
\]  
be the sequence of consecutive intervals associated to $[1,N]$ and $(a_1,b_1,\dots,a_k,b_k)$, respectively $(a_1,\dots,a_k,b_1,\dots,b_k)$. Finally,  let $\sigma_{\tbl}\in [S_N/S_\alpha]$ be the permutation that is order preserving from $I_j^A$ to $A_j$ and from $I_j^B$ to $B_j$ for all $j\in [1,k]$.

Direct computation shows that
\begin{equation}\label{eq wtp intersection gl}
\sigma_{\tbl}(P)\cap H=P_{(a_1,\dots,a_k)} \times P_{(b_1,\dots,b_k)}
\end{equation}
and in particular, that $\sigma_{\tbl}(P)\in \F_{a,b}^\alpha$.
Furthermore, 
\begin{equation}\label{eq wtm intersection gl}
\sigma_{\tbl}(M)\cap H=M_{(a_1,\dots,a_k,b_1,\dots,b_k)}, \ \ \ M\cap \sigma_{\tbl}^{-1} H\sigma_{\tbl}=M_{(a_1,b_1,\dots,a_k,b_k)}
\end{equation}
and
\begin{equation}\label{eq wtp0 intersection gl}
\sigma_{\tbl}(M)\cap P_0^H=M_{(a_1,\dots,a_k,b_1,\dots,b_k)}\cap P_0, \ \ \ M\cap \sigma_{\tbl}^{-1} P_0^H\sigma_{\tbl}=M_{(a_1,b_1,\dots,a_k,b_k)}\cap P_0.
\end{equation}

\begin{lemma}\label{lem GL parabolics}
The map $\tbl\mapsto \sigma_{\tbl}(P):\Tau_{a,b}^\alpha\rightarrow \F_{a,b}^\alpha$ is a bijection. 
\end{lemma}
\begin{proof}
Set $A=[1,a]$ and $B=[a+1,N]$. Note that the map $\tbl\mapsto \sigma_{\tbl}:\Tau_{a,b}^\alpha\rightarrow [S_N/S_\alpha]$ is one to one. Indeed, $\tbl$ as in \eqref{eq typ t} is recovered from $\sigma_{\tbl}$ by observing that
\[
a_i=\abs{\sigma_{\tbl}(I_i)\cap A} \ \ \ \text{and} \ \ \ b_i=\abs{\sigma_{\tbl}(I_i)\cap B},\ \ \ i\in[1,k].
\]
Since $w\mapsto w(P):[W/W_M]\rightarrow \P_{\sstd}$ is also one to one we conclude that the map is injective. 

Note that $w\in [S_N/S_\alpha]$ is of the form $w=\sigma_{\tbl}$ for some $\tbl\in \Tau_{a,b}^\alpha$ if and only if 
\begin{itemize}
\item $w(I_j)\cap A$ and $w(I_j)\cap B$ are intervals for all $j\in [1,k]$ and
\item $w(I_1)\cap A \rightarrow \cdots\rightarrow w(I_k)\cap A\rightarrow w(I_1)\cap B \rightarrow \cdots\rightarrow w(I_k)\cap B$.
\end{itemize}
That is, if  $w\in [S_N/S_\alpha]$ is not of the form $\sigma_{\tbl}$ for some $\tbl\in \Tau_{a,b}^\alpha$ then there exist $1\le s<t\le k$, $i\in I_s$ and $j\in I_t$ such that $w(i)>w(j)$ and either $w(i),w(j)\in A$ or $w(i),w(j)\in B$. Let $e_{j,i}$ be the $N\times N$ matrix with one in the $(j,i)$ entry and zeroes in all other entries and note that the matrix $u=I_{N}+e_{j,i}\in G\setminus P$. However, for a representative $\tilde w$ of $w$ in $G$ we have that $\tilde wu\tilde w^{-1}\in P_0^H$ (indeed, $(w(j),w(i))$ is its only non-zero non-diagonal entry) while $\tilde wu\tilde w^{-1}\not\in w(P)$. This shows that $w(P)\not\in\F_{a,b}^\alpha$ and the lemma follows.
\end{proof}

We say that $\beta\in\Comp_N$ refines $\alpha=(n_1,\dots,n_k)$ if $\beta=(\beta_1,\dots,\beta_k)$ where $\beta_i\in\Comp_{n_i}$, $i\in [1,k]$. When this is the case, we say that $\tbls\in \Tau_{a,b}^\beta$ refines $\tbl\in \Tau_{a,b}^\alpha$ if 
\begin{equation}\label{eq tablerefine}
\tbl=\begin{array}{ |c||c|} 
 \hline
\alpha  & N \\
 \hline\hline
\gamma  & a \\
\hline
\delta  & b \\
\hline
\end{array}=
\begin{array}{ |c|c|c||c|} 
 \hline
 n_1 & \cdots & n_k  & N \\
 \hline\hline
a_1 & \cdots & a_k  & a \\
\hline
b_1 & \cdots & b_k  & b \\
\hline
\end{array}
 \ \ \ \text{ and } \tbls=\begin{array}{ |c||c|} 
 \hline
\beta  & N \\
 \hline\hline
\gamma'  & a \\
\hline
\delta'  & b \\
\hline
\end{array}
=\begin{array}{ |c|c|c||c|} 
 \hline
 \beta_1 & \cdots & \beta_k  & N \\
 \hline\hline
\gamma_1 & \cdots & \gamma_k  & a \\
\hline
\delta_1 & \cdots & \delta_k  & b \\
\hline
\end{array}
\end{equation}
where $\gamma'=(\gamma_1,\dots,\gamma_k)$ is a refinement of $\gamma$ and $\delta'=(\delta_1,\dots,\delta_k)$ is a refinement of $\delta$. Here $\gamma_i\in\Comp_{a_i}$ and $\delta_i\in\Comp_{b_i}$, $i\in[1,k]$.
\begin{corollary}\label{cor restrict partol}
Let $Q=L\ltimes V\in \F_{a,b}$. Then the map $P\mapsto P\cap L$ is a bijection from $\{P\in \F_{a,b}: P\subseteq Q\}$ to $\F^L(P_0^H\cap L)$.
\end{corollary}
\begin{proof}
By Lemma \ref{lem GL parabolics} there exists $\alpha=(n_1,\dots,n_k)\in \Comp_N$ and $\tbl\in \Tau_{a,b}^\alpha$ such that $Q=\sigma_{\tbl}(P_\alpha)$ and we further have $L=\sigma_{\tbl}(M_\alpha)$. Thus, $\sigma_{\tbl}$-conjugation defines a bijection from $\F^{M_\alpha}(\sigma_{\tbl}^{-1}P_0^H\sigma_{\tbl}\cap M_\alpha)$ to $\F^L(P_0^H\cap L)$. 
It now follows from \eqref{eq wtp0 intersection gl} that for $\tbl$ as in \eqref{eq typ t}, applying Lemma \ref{lem GL parabolics} to each block $\GL_{n_i}$ of $M_\alpha$ the map
\[
(\tbls_1,\dots,\tbls_k)\mapsto \sigma_{\tbl}(\sigma_{\tbls_1}(P_{\beta_1})\times \cdots\times \sigma_{\tbls_k}(P_{\beta_k}))
\]
is a bijection from $\sqcup_\beta \Tau_{a_1,b_1}^{\beta_1}\times \cdots\times\Tau_{a_k,b_k}^{\beta_k}$ to $\F^L(P_0^H\cap L)$ where the disjoint union is over all refinements $\beta=(\beta_1,\dots,\beta_k)$ of $\alpha$.

Similarly, if $P\in \F_{a,b}$ then there exists $\beta\in\Comp_N$ and $\tbls\in \Tau_{a,b}^\beta$ such that $P=\sigma_\tbls(P_\beta)$. If in addition $P\subseteq Q$ then $P_\beta\subseteq \sigma_\tbls^{-1}\sigma_{\tbl}(P_\alpha)$ and therefore $\sigma_\tbls^{-1}\sigma_{\tbl}(P_\alpha)=P_\alpha$ (different standard parabolic subgroups are not conjugate to each other). It follows that $\beta$ is a refinement of $\alpha$ and that $\sigma_\tbls^{-1}\sigma_{\tbl}\in W_{M_\alpha}$. We conclude that the set $\{P\in \F_{a,b}: P\subseteq Q\}$ is the disjoint union over all refinements $\beta$ of $\alpha$ of $\{P\in \F_{a,b}^\beta: P\subseteq Q\}$.

Fix a refinement $\beta=(\beta_1,\dots,\beta_k)$ of $\alpha$. From the definition of the Weyl elements, direct computation shows that for $\tbls\in \Tau_{a,b}^\beta$ we have $\sigma_\tbls^{-1}\sigma_{\tbl}\in W_{M_\alpha}$ if and only if $\tbls$ is a refinement of $\tbl$.
Thus $\tbls\mapsto \sigma_\tbls(P_\beta)$ is a bijection from $\{\tbls\in \Tau_{a,b}^\beta: \tbls \text{ refines }\tbl\}$ to $\{P\in \F_{a,b}^\beta: P\subseteq Q\}$.
Applying the notation of \eqref{eq tablerefine} it is straightforward that $\tbls\mapsto (\tbls_1,\dots,\tbls_k)$ where
\[
\tbls_i=\begin{array}{ |c||c|} 
 \hline
\beta_i  & n_i \\
 \hline\hline
\gamma_i  & a_i \\
\hline
\delta_i  & b_i \\
\hline
\end{array},\ \ \ i\in [1,k]
\]
is a bijection from the set of refinements of $\tbl$ in $\Tau_{a,b}^{\beta}$ to $\Tau_{a_1,b_1}^{\beta_1}\times \cdots\times\Tau_{a_k,b_k}^{\beta_k}$.
To conclude it is now left to observe that with this notation we have
\[
\sigma_\tbls(P_\beta)\cap L=\sigma_{\tbl}(\sigma_{\tbls_1}(P_{\beta_1})\times \cdots\times \sigma_{\tbls_k}(P_{\beta_k})).
\]
\end{proof}

\subsection{Proof of Proposition \ref{prop GL vanishing}}
Note first that the assumption of the proposition forces $a\ne b$ and $N\ge 2$. In particular Lemma \ref{lem van per gl}, and the vanishing results of Ash, Ginzburg and Rallis \cite{MR1233493}, both of which require $a\ne b$, are available.
If $k=1$, that is, if $M=G$ then $E(\varphi,\lambda)=\varphi_\lambda$ is a cusp form so that also $\Lambda^T E(\varphi,\lambda)=\varphi_\lambda$. In this case the proposition follows from the vanishing results of Ash, Ginzburg and Rallis \cite{MR1233493} when $a,b\ge 1$; when $\min(a,b)=0$ we have $H=G$ and it follows from the orthogonality of cusp forms  to the one-dimensional automorphic representations on $\GL_N$, $N\ge 2$. 

We proceed by induction on $k$. Write $\pi=\pi_1\otimes \cdots\otimes \pi_k$ where $\pi_i$ is an irreducible cuspidal automorphic representation of $\GL_{n_i}(\A)$.
Twisting by an element in $\aaa_{M}^*$ we assume without loss of generality that $\pi$ has a unitary central character. 


For $Q=L\ltimes V\in \P_\std$, \eqref{eq consterm cusp} is an expansion of $C_{G,Q}E_P^G(\varphi,\lambda)$ as a sum of $A_L$-eigenfunctions. For $w\in{}_LW^\circ_M$, $E_{Q_w}^Q(M(w,\lambda)\varphi,w\lambda)$ is an eigenfunction for the twisted action of $A_L$ with eigenvalue $(w\lambda)_Q$. Based on Lemma \ref{lem GL parabolics} we have $\Sigma_Q=\{\sigma_\tbls: \tbls\in \Tau_{a,b}^\beta\}$ where $\beta\in \Comp_N$ is such that $Q=P_\beta$ and recall that $\rho_{Q,\sigma}=\rho_Q-2\rho_{Q\cap \sigma^{-1}H\sigma}$, $\sigma\in \Sigma_Q$.
By \eqref{eq Zydor} we have
\begin{multline}\label{eq zydor formula gl}
\P_{H,\chi}(E(\varphi,\lambda))=\sum_{Q\in \P_\std} (-1)^{\dim \aaa_Q^G} \ v_Q\sum_{w\in {}_LW^\circ_M}\sum_{\sigma\in \Sigma_Q}\frac{e^{\sprod{\sigma^{-1}T}{(w\lambda+\rho_{Q,\sigma}+\sigma^{-1}\xi)_Q^G}}}{\prod_{\varpi\in \hat\Delta_Q^\vee}\sprod{w\lambda+\rho_{Q,\sigma}+\sigma^{-1}\xi}{\varpi} }\times \\
\int_{K_H} \int_{[L\cap \sigma^{-1} H\sigma]^{1,L}}\chi(\sigma\ell\sigma^{-1} k) \delta_{Q\cap \sigma^{-1} H\sigma}^{-1}(\ell) \Lambda^{\sigma^{-1}T,Q}E_{Q_w}^Q(\ell\sigma^{-1} k,M(w,\lambda)\varphi,w\lambda)\ d\ell \ dk.
\end{multline}

By Lemma \ref{lem van per gl} the left hand side of \eqref{eq zydor formula gl} is zero. On the right hand side, the term for $Q=G$ is $\P_{H,\chi}^{\GL_N,T}(\varphi,\lambda)$.

To complete the proof of the proposition we claim that whenever $Q$ is proper in $G$, for any $w$ and $\sigma$ in the summation above the inner integral vanishes, that is
\begin{equation}\label{eq innerint}
\int_{[L\cap \sigma^{-1} H\sigma]^{1,L}} \chi(\sigma\ell\sigma^{-1} )\delta_{Q\cap \sigma^{-1} H\sigma}^{-1}(\ell) \Lambda^{\sigma^{-1}T,Q}E_{Q_w}^Q(\ell\sigma^{-1} k,M(w,\lambda)\varphi,w\lambda)\ d\ell=0.
\end{equation}
For the rest of the proof we explain how this follows from the induction hypothesis.

Indeed, write $\beta=(m_1,\dots,m_t)$ and $\sigma=\sigma_\tbls$ where 
\[
\tbls=\begin{array}{ |c|c|c||c|} 
 \hline
m_1 & \cdots & m_t  & N \\
 \hline\hline
a_1 & \cdots & a_t  & a \\
\hline
b_1 & \cdots & b_t  & b \\
\hline
\end{array}
\]
then 
\[
L\cap \sigma^{-1} H\sigma=M_{(a_1,b_1)}\times \cdots\times M_{(a_t,b_t)}.
\]
Let $j\in [1,k]$ be such that $n_j\ge 2$ and $2\min(a,b)<n_j$ (given by the assumption of the proposition) and let $i\in [1,t]$ be such that $j\in A_i$ where $w$ is associated, as in \S\ref{ss rwegl}, with the partition $(A_1,\dots,A_t)$ of $[1,k]$. Thus $2\min(a_i,b_i)\le 2\min(a,b)<n_j$, and $n_j\ge 2$ is an entry of the composition $\gamma_i$ of $m_i$ defined in \eqref{eq refinement coord}, so that the assumption of the proposition is satisfied for the group $\GL_{m_i}$, the subgroup $\GL_{a_i}\times \GL_{b_i}$ and the parabolic subgroup $P_{\gamma_i}$. Note that since $Q$ is proper in $G$ we have $t\ge 2$, so that the sets $A_i$ occurring below all satisfy $\abs{A_i}<k$. Note further that the same computation as above, applied to the pair $(a_i,b_i)$ and the composition $\gamma_i$ of $m_i$ defined below, shows that $a_i\ne b_i$ whenever the induction hypothesis is invoked.
From the induction hypothesis we have 
\[
\P_{\GL_{a_i}\times\GL_{b_i},\chi''}^{\GL_{m_i},T'}(\xi,\mu)=0
\] 
(see \eqref{eq truncint} for the notation) for $\gamma_i$ as defined in \eqref{eq refinement coord} and in its notation for any $\xi\in I_{P_{\gamma_i}}^{\GL_{m_i}}(\pi_{j_1}\otimes \cdots\otimes \pi_{j_r})$, $\mu\in \aaa_{M_{\gamma_i},\C}^*$, any character $\chi''$ of $[\GL_{a_i}\times\GL_{b_i}]$ and any $T'$ positive enough. Note further that if $T$ is positive enough in $G$ then $\sigma^{-1}T$ is positive enough in $L$.  
Applying Corollary \ref{cor restrict partol} the left hand side of \eqref{eq innerint} factors through $\P_{\GL_{a_i}\times\GL_{b_i},\chi''}^{\GL_{m_i},T'}(\xi,\mu)$ for appropriate data $\xi$, $\chi''$ and $T'$ as above. The vanishing \eqref{eq innerint} follows and this completes the proof of the proposition. \qed

\section{Vanishing of periods for the symplectic group}\label{vanishing-22}

\subsection{The period subgroups}
For a composition $N=a+b$ we realize $\Sp_a\times \Sp_b$ as a subgroup of $\Sp_N$ as follows.
Let $\jj=\jj_{a,b}: \Sp_a\times \Sp_b\rightarrow \Sp_N$ be the imbedding
\[
\jj(h_1,h_2)=\begin{pmatrix} A & & B \\ & h_2 & \\ C & & D\end{pmatrix}\ \ \ \text{where} \ \ \ h_1=\begin{pmatrix} A & B \\ C & D \end{pmatrix}\in \Sp_a,\,h_2\in \Sp_b
\] 
(here  $A,\,B,\,C,\,D$ are of size $a\times a$)
and set
\[
H_{a,b}=\jj(\Sp_a\times \Sp_b).
\]
Note that $H_{a,b}$ is the centralizer in $G$ of $\inj(I_a,-I_{b})=\diag(I_a,-I_{2b},I_a)$.
\subsection{A local vanishing result for symplectic groups}\label{ss localvansp}
Assume here that $F$ is a $p$-adic field and $q$ the size of its residue field. Let $B=T\ltimes U_0$ be the standard Borel subgroup of $\Sp_N(F)$ with unipotent radical $U_0$ consisting of the upper triangular unipotent matrices and $T$ the maximal split torus consisting of diagonal matrices in $\Sp_N(F)$. 

For $\lambda=(\lambda_1,\dots,\lambda_N)\in  \C^N$ let $I(\lambda)$ be the unramified principal series representation of $\Sp_N(F)$ parabolically induced from $B$ and the character 
\[
\chi_\lambda(\inj(t_1,\dots,t_N))=\prod_{i=1}^N\abs{t_i}^{\lambda_i}, \ \ \ t_i\in F^*, \ \ \ i\in[1,N]. 
\]

We recall some facts about involutions in $\Sp_N(F)$.
It follows from \cite[Propositions 6.6 and 6.8]{MR1215304} that inclusion defines a bijection from the $T$-conjugacy classes of involutions in $N_{\Sp_N(F)}(T)$ to the $B$-conjugacy classes of involutions in $\Sp_N(F)$ and that these sets of conjugacy classes are finite. Denote by $\W^+_2$ the set of involutions $w=\tau\set\in \W_N$ such that $\tau$ has no fixed points in $\set$.
For $w=\tau\set\in\W_N$ let $Y_w$ be the set of involutions in $wT$. Then $Y_w$ is empty unless $w\in \W_2^+$ and in this case the $T$-conjugacy classes in $Y_w$ can be parameterized by subsets of the set $\{i\in [1,N]:\tau(i)=i\}$. The exact parameterization does not play a role here.

Denote by $Q_y$ the centralizer of $y$ in $Q$ for any involution $y$ and subgroup $Q$ in $\Sp_N(F)$. 
For $w=\tau\set\in \W_2^+$ let 
\[
T_w=\{\inj(t_1,\dots,t_N): t_i=\begin{cases} t_{\tau(i)}& i\not\in\set\\ t_{\tau(i)}^{-1} & i\in \set \end{cases},\ i\in[1,N]\}.
\] 
If $y\in Y_w$ then $T_y=T_w$ depends only on $w$. Let $\delta_y=\delta_B^{-\frac12}\delta_{B_y}|_{T_w}$. As an unramified real character of $T_w$, there exists $\mu(y)=(\mu_1(y),\dots,\mu_N(y))\in \R^N$ such that $w\mu(y)=\mu(y)$ and $\delta_y=\chi_{\mu(y)}|_{T_w}$.

Fix a composition $N=a+b$ with $a,\,b\in [0, N]$ and let $H=H_{a,b}(F)\simeq \Sp_a(F)\times \Sp_b(F)$. Let $X$ be the $\Sp_N(F)$-conjugacy class of $\epsilon=\inj(I_a,-I_{b})$. For every $w\in \W_2^+$ we fix a set of representatives $\Sigma_w$ for the $T$-conjugacy classes in $Y_w\cap X$.
For $y\in \Sigma_w$ let $c_y=(c_1,\dots,c_N)$ where 
\[
c_i=\begin{cases}
q^{\mu_i(y)+\mu_{\tau(i)}(y)} & i\not\in\set \\
q^{\mu_i(y)-\mu_{\tau(i)}(y)} & i\in\set,  
\end{cases}\ \ \ i\in [1,N]
\]
and let $C_w=\{c_y:y\in \Sigma_w\}$. The set $C_w$ is finite.

\begin{lemma}\label{lem local vanishing sp}
If $\lambda\in\C^N$ is such that $\Hom_H(I(\lambda),\triv_H)\ne 0$ then there exists $w=\tau\set\in \W_2^+$ and $(c_1,\dots,c_N)\in C_w$ such that 
\begin{equation}\label{eq closed cond sp}
c_i=\begin{cases} q^{\lambda_i+\lambda_{\tau(i)}} & i\not\in\set \\q^{\lambda_i-\lambda_{\tau(i)}} & i\in\set\end{cases},\ \ \ i\in[1,N].
\end{equation}
\end{lemma}
\begin{proof}
 As a consequence of the geometric lemma of Bernstein and Zelevinsky (see \cite[Theorem 4.2]{MR3541705}), if $\Hom_H(I(\lambda),\triv_H)\ne 0$ then there exists  $w=\tau\set\in \W_2^+$ and $y\in \Sigma_w$ such that $\Hom_{T_w}(\chi_\lambda,\delta_y)\ne 0$. That is, $c_y=(c_1,\dots,c_N)\in C_w$ satisfies \eqref{eq closed cond sp}.
The lemma follows.
\end{proof}
For the rest of this section we go back to the global setting. 

\subsection{Global vanishing for the regularized period}

Let $N=a+b$ with $a,\,b\in [0,N]$ and set  $H=H_{a,b}$.
Let $P=M\ltimes U\in \P_\std$ be contained in the Siegel parabolic $P_{(N;0)}$ and let $\pi$ be an irreducible, cuspidal automorphic representation of $M(\A)$. Then $E_P^G(\varphi,\lambda)$ is in $\Aut_G^{H-\reg}$ for $\varphi\in I_P^G(\pi)$ and $\lambda\in \aaa_{M,\C}^*$ away from finitely many hyperplanes. The regularized period $\P_{H}(E_P^G(\varphi,\lambda))$ is therefore defined for generic $\lambda$ and defines a meromorphic function in $\lambda$ (as in \cite[Theorem 8.4.1 (4)]{MR2010737}).

\begin{lemma}\label{lem van per sp} 
Assume that $a\ne b$. In the above notation we have
\[
\P_H(E_P^G(\varphi,\lambda))=0,\ \ \ \varphi\in I_P^G(\pi),\ \lambda\in\aaa_{M,\C}^*.
\]
\end{lemma}
\begin{proof}
It follows from \cite[Theorem 4.1 (4)]{MR4411860} that $\varphi\mapsto \P_H(E_P^G(\varphi,\lambda))$ is $H(\A)$-invariant on $I_P^G(\pi,\lambda)$. 
Let $v$ be a finite place of $F$ such that $\pi_v$ is unramified. Since $\pi$ is cuspidal on a Levi subgroup of $\GL_N$, $\pi_v$ is also generic. In particular, it is an unramified principal series. There exists therefore $\lambda_0\in \C^N$ such that $I_P^G(\pi,\lambda)_v=I(\lambda_0+\lambda)$ where on the right hand side we use the notation of \S\ref{ss localvansp} with the local field $F_v$.

Assume by contradiction that $\P_H(E_P^G(\varphi,\lambda))$ is not identically zero. Then $\Hom_{H(F_v)}(I(\lambda_0+\lambda),\triv_H)\ne 0$ for $\lambda$ in a non-empty open subset of $\aaa_{M,\C}^*$. It follows from Lemma \ref{lem local vanishing sp}, in its notation, that there exists an involution $w=\tau\set\in \W_2^+$ and $y\in \Sigma_w$ such that for a non-empty open set of $\lambda$'s in $\aaa_{M,\C}^*$, we have that $\lambda_0+\lambda$ satisfies \eqref{eq closed cond sp} with $c_y=(c_1,\dots,c_N)$ (this follows arguing as in the proof of Lemma \ref{lem van per gl}).  We claim that $\set\ne [1,N]$. Indeed, assume that $\set=[1,N]$. Then $\tau$ has no fixed points and any element of $Y_w$ has an eigenspace of dimension $N$ for each of its eigenvalues $\pm 1$, whereas the eigenspace of $\epsilon=\inj(I_a,-I_b)$ for the eigenvalue $1$ has dimension $2a\ne N$, since $a\ne b$. Consequently, $Y_w\cap X=\emptyset$, which contradicts $y\in \Sigma_w$. 
Let now $i\in [1,N]\setminus\set$ and note that also $\tau(i)\not\in \set$. The corresponding condition in \eqref{eq closed cond sp} is $q^{(\lambda_0+\lambda)_i+(\lambda_0+\lambda)_{\tau(i)}}=c_i$. However, the left hand side is a non-constant holomorphic function on $\aaa_{M,\C}^*$ (its exponent is an affine function of $\lambda$ with linear part $\lambda_j+\lambda_{j'}$, where $j$ and $j'$ are the indices of the blocks of $M$ containing $i$ and $\tau(i)$) and cannot be constant on a non-empty open set. This is a contradiction and the lemma follows.
\end{proof}

\subsection{Reduced Bruhat representatives for $\Sp_N$}\label{ss reducedc}
We obtain a convenient parameterization of certain left and right reduced Weyl elements.
The setup is as follows.

Let 
\[
\alpha=(m_1,\dots,m_t;0),\ \beta=(n_1,\dots,n_k;r)\in \Comp_N^\Sp,
\]  
\[
P=M\ltimes U=P_\alpha\subseteq P_{(N;0)}\ \ \ \text{and}\ \ \ Q=L\ltimes V=P_\beta. 
\]
Based on \eqref{eq basic red weyl} and the characterization of reduced Weyl elements in \S\ref{ss reduced} we parameterize elements of ${}_LW^\circ_M$ by partitions 
\[
\Data=(D_1,E_1,\dots,D_k,E_k, D_0)
\]
of $[1,t]$ into $2k+1$ subsets such that
\[
n_i=\sum_{j\in D_i\sqcup E_i} m_j \ \ \ \text{and}\ \ \ r=\sum_{j\in D_0} m_j.
\]
The element $w=w_\Data\in {}_LW^\circ_M$ associated to $\Data$ is characterized by explicating its action $w:\aaa_{M}^*\rightarrow \aaa_{w(M)}^*$.
In order to explicate this action write
\[
D_i=\{x_{i,1}<\dots <x_{i,d_i}\},\ \ \ i\in [0,k]\ \ \ \text{and} \ \ \ E_i=\{y_{i,e_i}<\dots<y_{i,1}\},\ \ \ i\in [1,k]
\]
and let
\[
\gamma_i=(m_{x_{i,1}},\dots,m_{x_{i,d_i}},m_{y_{i,1}},\dots,m_{y_{i,e_i}})\in \Comp_{n_i},\ i\in [1,k]\ \ \ \text{and}\ \ \ \gamma_0=(m_{x_{0,1}},\dots,m_{x_{0,d_0}})\in \Comp_r. 
\]
Then 
\[
w(M)=M_{(\gamma_1,\dots,\gamma_k,\gamma_0;0)} 
\]
and
\[
w\lambda=(\mu_1,\dots,\mu_k,\mu_0) \ \ \ \text{for}\ \ \ \lambda=(\xi_1^{(m_1)},\dots,\xi_t^{(m_t)})\in \aaa_M^*
\]
where
\[
\mu_i=(\xi_{x_{i,1}}^{(m_{x_{i,1}})},\dots,\xi_{x_{i,d_i}}^{(m_{x_{i,d_i}})},(-\xi_{y_{i,1}})^{(m_{y_{i,1}})},\dots,(-\xi_{y_{i,e_i}})^{(m_{y_{i,e_i}})})\in \aaa_{M_{\gamma_i}}^*,\ i\in [1,k]
\]
and
\[
\mu_0=(\xi_{x_{0,1}}^{(m_{x_{0,1}})},\dots,\xi_{x_{0,d_0}}^{(m_{x_{0,d_0}})})\in \aaa_{M_{\gamma_0}}^*.
\]

\subsection{The relevant semi-standard parabolic subgroups}
Let $N\in \N$ and set $G=\Sp_N$. Let $P_0=B_N=P_{(1^{(N)};0)}$ be the Borel subgroup consisting of the upper-triangular matrices in $G$. 
For a composition $N=a+b$ with $a,\,b\in[0,N]$ set $H=H_{a,b}$ and let 
\[
P_0^H=P_0\cap H=\jj(B_a\times B_b).
\] 
Set $\F_{a,b}=\F_{a,b}^{\Sp}=\F^G(P_0^H)$. Our goal here is to show analogous results for $\F_{a,b}^{\Sp}$ to those in \S\ref{ss rel ss par} for $\F_{a,b}^{\GL}$. In fact, the analysis in \S\ref{ss rel ss par} will be applied.  
It is again straightforward that $\P_\std\subseteq \F_{a,b}$. We write $\F_{a,b}=\sqcup_{P\in \P_\std} \F_{a,b}^P$ where $\F_{a,b}^P$ is the set of $Q\in \F_{a,b}\cap \P_{\sstd}$ that are conjugate to $P$. 
If $P=P_{(n_1,\dots,n_k;r)}$ write $\alpha_P=(n_1,\dots,n_k,r)\in\Comp_N$ and $\Tau_{a,b}^P=\Tau_{a,b}^{\alpha_P}$ (see \S\ref{ss rel ss par} for the definition of the right hand side). Note that $\alpha_P=\alpha_{P\cap P_{(N;0)}}$ so that $\Tau_{a,b}^P=\Tau_{a,b}^{P\cap P_{(N;0)}}$. For $\tbl\in \Tau_{a,b}^P$ we defined in  \S\ref{ss rel ss par}  the Weyl element $\sigma_{\tbl}\in S_N$. Recall that now $S_N$ is a subgroup of $W=\W_N$. Direct computation shows that for 
\begin{equation}\label{eq t typ sp}
\tbl=\begin{array}{ |c|c|c|c||c|} 
 \hline
 n_1 & \cdots & n_k & r & N \\
 \hline\hline
a_1 & \cdots & a_k & r_1 & a \\
\hline
b_1 & \cdots & b_k & r_2 & b \\
\hline
\end{array}\in \Tau_{a,b}^P
\end{equation}
we have
\begin{equation}\label{eq par of h}
\sigma_{\tbl}(P)\cap H=\jj(P_{(a_1,\dots,a_k;r_1)}\times P_{(b_1,\dots,b_k;r_2)})
\end{equation}
and in particular, $\sigma_{\tbl}(P)\in \F_{a,b}^P$.
Furthermore,
\begin{equation}\label{eq mhintersection sp}
\sigma_{\tbl}(M)\cap H=\jj(M_{(a_1,\dots,a_k;r_1)},M_{(b_1,\dots,b_k;r_2)}),\ \ \  M\cap \sigma_{\tbl}^{-1}H\sigma_{\tbl}=\inj(M_{(a_1,b_1,\dots,a_k,b_k)};H_{r_1,r_2}),
\end{equation}
\begin{equation}\label{eq p0hintersection sp}
\sigma_{\tbl}(M)\cap P_0^H=\jj\left(\inj(B_{a_1}^\GL\times \cdots\times B_{a_k}^\GL;B_{r_1})\times \inj(B_{b_1}^\GL\times \cdots\times B_{b_k}^\GL;B_{r_2})\right)
\end{equation}
and
\begin{equation}\label{eq w conj p0hintersection sp}
M\cap \sigma_{\tbl}^{-1}P_0^H\sigma_{\tbl}=\inj(B_{a_1}^\GL\times B_{b_1}^\GL\times \cdots\times B_{a_k}^\GL\times B_{b_k}^\GL;\jj(B_{r_1}\times B_{r_2})),
\end{equation}
where, for $r'\in \Z_{\ge 0}$, $B_{r'}$ denotes the Borel subgroup of upper triangular matrices in $\Sp_{r'}$, so that $\jj(B_{r_1}\times B_{r_2})=P_0^{H_{r_1,r_2}}$ is the Borel subgroup of $H_{r_1,r_2}$ analogous to $P_0^H$.

\begin{lemma}\label{lem sp parabolics}
With the above notation the map $\tbl\mapsto \sigma_{\tbl}(P):\Tau_{a,b}^P\rightarrow\F_{a,b}^P$ is a bijection.
\end{lemma}
\begin{proof}
Let $Q\in \F_{a,b}^P$ and let $w=\tau \set\in [W/W_M]$ be such that $Q=w(P)$. We claim that, in fact, $w\in S_N$, that is that $\set$ is empty. Since $w$ is right $M$-reduced we have $\set\subseteq [1,N-r]$. Assume by contradiction that $i\in \set$ and let $\tilde w\in G$ be a representative of $w$. Let $q\in Q$ and let $p\in P$ be such that $\tilde w p\tilde w^{-1}=q$. Note that $p_{2N+1-i,i}=0$ and by matrix multiplication, consequently also $q_{\tau(i), 2N+1-\tau(i)}=0$. That is, there exists $j=\tau(i)\in[1,N]$ such that $q_{j,2N+1-j}=0$ for every $q\in Q$.  
However, writing $e_{i,j}$ for the $2N\times 2N$ matrix with one in the $(i,j)$th entry and zeroes elsewhere we observe that $b=I_{2N}+e_{j,2N+1-j}\in P_0^H$ is such that $b_{j,2N+1-j}=1$ and in particular $b\not\in Q$. This contradicts the assumption that $Q\in\F_{a,b}$. We conclude that $w\in S_N$ and in fact $w\in [S_N/S_{\alpha_P}]$. This allows some reduction to computations in $\GL_N$. 
Note that 
\[
\inj(B_a^{\GL},B_b^{\GL})=P_0^H\cap M_{(N;0)}\subseteq Q\cap M_{(N;0)}=\inj(w(P_{\alpha_P}))
\]
 (here, $P_{\alpha_P}$ is the standard parabolic subgroup of $\GL_N$ such that $\inj(P_{\alpha_P})=P\cap M_{(N;0)}$). It now follows from Lemma \ref{lem GL parabolics} that there exists $\tbl\in \Tau_{a,b}^{\alpha_P}$ such that $w=\sigma_{\tbl}$. Thus, the map is surjective. It also follows from Lemma
  \ref{lem GL parabolics} that the map $\tbl\mapsto \sigma_{\tbl}(P)\cap M_{(N;0)}$ is injective and in particular, $\tbl\mapsto \sigma_{\tbl}(P)$ is also injective. The lemma follows.
\end{proof}

Let $\alpha=(n_1,\dots,n_k;r)\in \Comp_N^\Sp$. We say that $\beta=(\beta_1,\dots,\beta_k,\gamma)\in \Comp_N^\Sp$ is a refinement of $\alpha$ if $\beta_i\in \Comp_{n_i}$, $i\in [1,k]$ and $\gamma\in\Comp_r^\Sp$.
\begin{corollary}\label{cor restrict partol sp}
Let $Q=L\ltimes V\in \F_{a,b}$. Then the map $P\mapsto P\cap L$ is a bijection from $\{P\in \F_{a,b}: P\subseteq Q\}$ to $\F^L(P_0^H\cap L)$.
\end{corollary}
\begin{proof}
By Lemma \ref{lem sp parabolics} there is $\alpha=(n_1,\dots,n_k;r)\in\Comp_N^\Sp$ and $\tbl\in \Tau_{a,b}^{P_\alpha}$ such that $Q=\sigma_{\tbl}(P_\alpha)$. We further have $L=\sigma_{\tbl}(M_\alpha)$. 
Thus, $\sigma_{\tbl}$-conjugation defines a bijection from $\F^{M_\alpha}(\sigma_{\tbl}^{-1}P_0^H\sigma_{\tbl}\cap M_\alpha)$ to $\F^L(P_0^H\cap L)$. 
It now follows from \eqref{eq w conj p0hintersection sp} and \eqref{eq p0hintersection sp} that writing $\tbl$ as in \eqref{eq t typ sp}, applying Lemma \ref{lem GL parabolics} to each block $\GL_{n_i}$ of $M_\alpha$ and Lemma \ref{lem sp parabolics} to the block $\Sp_r$ the map
\[
(\tbls_1,\dots,\tbls_k,R)\mapsto \sigma_{\tbl}(\inj(\sigma_{\tbls_1}(P_{\beta_1}), \cdots, \sigma_{\tbls_k}(P_{\beta_k}); \sigma_R(P_\gamma)))
\]
is a bijection from $\sqcup_\beta \Tau_{a_1,b_1}^{\beta_1}\times \cdots\times\Tau_{a_k,b_k}^{\beta_k}\times \Tau_{r_1,r_2}^{P_\gamma}$ to $\F^L(P_0^H\cap L)$ where the disjoint union is over all refinements $\beta=(\beta_1,\dots,\beta_k,\gamma)\in \Comp_N^\Sp$ of $\alpha$.

Similarly, if $P\in \F_{a,b}$ then there exists $\beta\in \Comp_N^\Sp$ and $\tbls\in \Tau_{a,b}^{P_\beta}$ such that $P=\sigma_{\tbls}(P_\beta)$. If in addition $P\subseteq Q$ then $P_\beta\subseteq \sigma_{\tbls}^{-1}\sigma_{\tbl}(P_\alpha)$ and therefore $\sigma_{\tbls}^{-1}\sigma_{\tbl}(P_\alpha)=P_\alpha$ (different standard parabolic subgroups are not conjugate to each other). It follows that $\beta$ is a refinement of $\alpha$ and that $\sigma_{\tbls}^{-1}\sigma_{\tbl}\in W_{M_\alpha}$. Consequently, the set $\{P\in \F_{a,b}: P\subseteq Q\}$ is the disjoint union over all refinements $\beta$ of $\alpha$ of $\{P\in \F_{a,b}^{P_\beta}: P\subseteq Q\}$.

Fix a refinement $\beta=(\beta_1,\dots,\beta_k,\gamma)$ of $\alpha$. As in the proof of Corollary \ref{cor restrict partol}, for $\tbls\in \Tau_{a,b}^{P_\beta}$ we have $\sigma_{\tbls}^{-1}\sigma_{\tbl}\in W_{M_\alpha}$ if and only if $\tbls$ is a refinement of $\tbl$.
Thus $\tbls\mapsto \sigma_{\tbls}(P_\beta)$ is a bijection from $\{\tbls\in \Tau_{a,b}^{P_\beta}: \tbls \text{ refines }\tbl\}$ to $\{P\in \F_{a,b}^{P_\beta}: P\subseteq Q\}$.
As in the proof of Corollary \ref{cor restrict partol}, there is a bijection $\tbls\mapsto (\tbls_1,\dots,\tbls_k,R)$ from the set of refinements of $\tbl$ in $\Tau_{a,b}^{P_\beta}$ to $\Tau_{a_1,b_1}^{\beta_1}\times \cdots\times\Tau_{a_k,b_k}^{\beta_k}\times \Tau_{r_1,r_2}^{P_\gamma}$.
To conclude it is now left to observe that with this notation we have
\[
\sigma_{\tbls}(P_\beta)\cap L=\sigma_{\tbl}(\inj(\sigma_{\tbls_1}(P_{\beta_1}), \cdots, \sigma_{\tbls_k}(P_{\beta_k}); \sigma_R(P_\gamma))).
\]

\end{proof}

\subsection{The main vanishing result for periods on $\Sp_N$}

Fix $\alpha=(m_1,\dots,m_t;0)\in \Comp_N^\Sp$ such that $P=M\ltimes U=P_\alpha\in \P_\std$ is contained in the Siegel parabolic subgroup $P_{(N;0)}$. Fix $a,b\in [0,N]$ such that $N=a+b$ and let $H=H_{a,b}$. Our main result in this section is the following.
\begin{proposition}\label{prop Sp vanishing}
If $\max(1,2\min(a,b))<\max(m_1,\dots,m_t)$ then $\P_H^{G,T}$ is identically zero on $I_P^G(\pi)\times \aaa_{M,\C}^*$ for any irreducible, cuspidal automorphic representation $\pi$ of $M(\A)$.
\end{proposition}

\begin{proof}
Note first that the assumption of the proposition forces $a\ne b$, so that Lemma \ref{lem van per sp} is available. We proceed by induction on $t$.
Write $\pi=\pi_1\otimes \cdots\otimes \pi_t$ where $\pi_i$ is an irreducible cuspidal automorphic representation of $\GL_{m_i}(\A)$, $i\in [1,t]$.
After possibly twisting $\pi$ by an element of $\aaa_{M}^*$ we assume, without loss of generality, that $\pi$ has a unitary central character. 

For $Q=L\ltimes V\in \P_\std$, \eqref{eq consterm cusp} is an expansion of $C_{G,Q}E_P^G(\varphi,\lambda)$ as a sum of $A_L$-eigenfunctions. For $w\in{}_LW^\circ_M$, $E_{Q_w}^Q(M(w,\lambda)\varphi,w\lambda)$ is an eigenfunction for the twisted action of $A_L$ with eigenvalue $(w\lambda)_Q$. 
Based on Lemma \ref{lem sp parabolics} we have $\Sigma_Q=\{\sigma_{\tbls}: \tbls\in \Tau_{a,b}^Q\}$ and recall that $\rho_{Q,\sigma}=\rho_Q-2\rho_{Q\cap \sigma^{-1}H\sigma}$, $\sigma\in \Sigma_Q$.
By \eqref{eq Zydor} we have
\begin{multline}\label{eq zydor formula sp}
\P_H(E_P^G(\varphi,\lambda))=\sum_{Q\in \P_\std} (-1)^{\dim\aaa_Q^G}\ v_Q\sum_{w\in {}_LW^\circ_M}\sum_{\sigma\in \Sigma_Q}\frac{e^{\sprod{\sigma^{-1}T}{(w\lambda+\rho_{Q,\sigma})_Q}}}{\prod_{\varpi\in \hat\Delta_Q^\vee}\sprod{w\lambda+\rho_{Q,\sigma}}{\varpi} }\times \\
\int_{K_H} \int_{[L\cap \sigma^{-1} H\sigma]^{1,L}} \delta_{Q\cap \sigma^{-1} H\sigma}^{-1}(\ell) \Lambda^{\sigma^{-1}T,Q}E_{Q_w}^Q(\ell\sigma^{-1} k,M(w,\lambda)\varphi,w\lambda)\ d\ell \ dk.
\end{multline}
By Lemma \ref{lem van per sp} the left hand side of \eqref{eq zydor formula sp} is zero. On the right hand side, the term for $Q=G$, is $\P_H^{G,T}(\varphi,\lambda)$.

To complete the proof of the proposition we claim that whenever $G\ne Q\in \P_\std$, for any $w$ and $\sigma$ in the summation above the inner integral vanishes, that is
\begin{equation}\label{eq innerint sp}
\int_{[L\cap \sigma^{-1} H\sigma]^{1,L}} \delta_{Q\cap \sigma^{-1} H\sigma}^{-1}(\ell) \Lambda^{\sigma^{-1}T,Q}E_{Q_w}^Q(\ell\sigma^{-1} k,M(w,\lambda)\varphi,w\lambda)\ d\ell=0.
\end{equation}
For the rest of the proof we explain how this follows from the induction hypothesis and our previous results. 

Let $\beta=(n_1,\dots,n_k;r)\in \Comp_N^\Sp$ be such that $Q=P_\beta$ and fix $w\in {}_LW^\circ_M$ and $\sigma\in \Sigma_Q$.
Applying Lemma \ref{lem sp parabolics} write  $\sigma=\sigma_{\tbls}$ where 
\[
\tbls=\begin{array}{ |c|c|c|c||c|} 
 \hline
n_1 & \cdots & n_k  & r & N \\
 \hline\hline
a_1 & \cdots & a_k  &r_1 &  a \\
\hline
b_1 & \cdots & b_k  & r_2 &b \\
\hline
\end{array} \in \Tau_{a,b}^Q.
\]
Then (see \eqref{eq mhintersection sp})
\[
L\cap \sigma^{-1} H\sigma=\inj(M_{(a_1,b_1)}, \cdots, M_{(a_k,b_k)};H_{r_1,r_2}).
\]
Let $j\in [1,t]$ be such that $m_j\ge 2$ and $2\min(a,b)<m_j$ and let $\Data=(D_1,E_1,\dots,D_k,E_k, D_0)$ be the partition of $[1,t]$ such that $w=w_\Data$ as in 
\S\ref{ss reducedc}. If $j\in D_0$ set $i=0$ and otherwise let $i\in [1,k]$ be such that $j\in D_i\sqcup E_i$. 

If $i\in [1,k]$ set
\[
\mathcal{Q}_i(\xi,\mu)=\P_{\GL_{a_i}\times\GL_{b_i},\chi}^{\GL_{n_i},T'}(\xi,\mu)
\]
(see \eqref{eq truncint} for the notation) for a character $\chi$ of $[\GL_{a_i}\times \GL_{b_i}]$, $\xi\in I_{P_{\gamma_i}}^{\GL_{n_i}}(\pi_{x_{i,1}}\otimes \cdots \otimes \pi_{x_{i,d_i}}\otimes \pi_{y_{i,1}}^\vee\otimes \cdots \otimes \pi_{y_{i,e_i}}^\vee)$ (the blocks indexed by $E_i$ are acted upon by $w$ through the involution $g\mapsto g^*$, which replaces the cuspidal representation $\pi_{y_{i,j}}$ by its contragredient), $\mu\in \aaa_{M_{\gamma_i},\C}^*$ and $T'$ positive enough and note that $2\min(a_i,b_i)\le 2\min(a,b)<m_j\le n_i$ and that $m_j\ge 2$ is an entry of $\gamma_i$, so that the assumption of Proposition \ref{prop GL vanishing} is satisfied.
If $i=0$ set 
\[
\mathcal{Q}_0(\xi,\mu)=\P_{H_{r_1,r_2}}^{\Sp_r,T'}(\xi,\mu)
\]
for all $\xi\in I_{P_{(\gamma_0;0)}}^{\Sp_r}(\pi_{x_{0,1}}\otimes \cdots \otimes \pi_{x_{0,d_0}})$, $\mu\in\aaa_{M_{(\gamma_0;0)},\C}^*$ and $T'$ positive enough and note that $2\min(r_1,r_2)\le 2\min(a,b)<m_j\le r$ and that $m_j\ge 2$ is an entry of $\gamma_0$, so that the assumption of the proposition is satisfied for $\Sp_r$, $H_{r_1,r_2}$ and $P_{(\gamma_0;0)}$.
In either case the notation for $\gamma_i$ and its entries is associated to $w$ as in \S\ref{ss reducedc}.

It follows from Proposition \ref{prop GL vanishing} if $i\in [1,k]$ and from the induction hypothesis if $i=0$ that $\mathcal{Q}_i(\xi,\mu)=0$ for all data as above. Note that $r<N$ and therefore the number $\abs{D_0}$ of general linear factors of $M_{(\gamma_0;0)}$ is strictly less than $t$. In particular, for the base of induction, $t=1$, it is always the case that $i\ne 0$.

Note further that if $T$ is positive enough in $G$ then $\sigma^{-1}T$ is positive enough in $L$.  
Applying Corollary \ref{cor restrict partol sp}, the left hand side of \eqref{eq innerint sp} factors through $\mathcal{Q}_i(\xi,\mu)$ for appropriate $\xi,\mu$, $T'$ (and $\chi$ if $i\in[1,k]$). The vanishing \eqref{eq innerint sp} follows and this completes the proof of the proposition.

\end{proof}

\subsection{Another generic local vanishing result}

In this section $F$ is a non-archimedean local field of characteristic different from two. Our goal is to show a certain local generic vanishing result on $\Sp_N$. Our first step is extending a well known result on disjointness of models for $\GL_N$.

\subsubsection{Inherited disjointness of models for general linear groups}
Fix a non-trivial additive character $\psi_0$ of $F$. Denote by $U_n$ the group of upper triangular unipotent matrices in $\GL_n$ and let $\psi_n(u)=\psi_0(u_{1,2}+\cdots+u_{n-1,n})$ be a generic character of $U_n(F)$. We say that a subgroup $H$ of $\GL_n$ is disjoint from $(U_n,\psi_n)$ if 
\[
\psi_n|_{U_n(F)\cap gH(F)g^{-1}}\not\equiv 1,\ \ \ \text{for all}\ \ \ g\in \GL_n(F).
\] 
Based on the ideas of \cite{MR0425030} and \cite{MR0333080} we have the following implication (see for example the explanation in \cite{MR4017938}):
if $H$ is disjoint from $(U_n,\psi_n)$ then for every irreducible generic representation $\pi$ of $\GL_n(F)$ we have $\Hom_H(\pi,\triv_H)=0$.
\begin{lemma}\label{lem disjointness}
Let $m,k\in\N$ and let $H$ be a subgroup of $\GL_k$ that is disjoint from $(U_k,\psi_k)$. Then the subgroup 
\[
H_m=\{\begin{pmatrix} I_m & X \\ & h\end{pmatrix}:h\in H,\ X\in M_{m\times k}\}
\]
of $\GL_{m+k}$ is disjoint from $(U_{m+k},\psi_{m+k})$.
\end{lemma}
\begin{proof}
Let $n=m+k$, $G=\GL_n$, $g\in G(F)$, $L=M_{(m,k)}$, $V=U_{(m,k)}$ and $T$ the diagonal torus in $\GL_n$. By Bruhat decomposition we can write $g=uw\diag(g_1,g_2)\sm{I_m}{X}{}{I_k}$ where $u\in U_{n}(F)$, $w\in [W/W_L]$ viewed as a permutation matrix, $g_1\in \GL_m(F)$, $g_2\in \GL_k(F)$ and $X\in M_{m\times k}(F)$. In order to prove that 
\begin{equation}\label{eq nontriv char}
\psi_n|_{U_n(F)\cap gH_m(F)g^{-1}}\not\equiv 1
\end{equation}
we may and do assume, without loss of generality, that $u=I_n$. 
Note then that
\[
gH_m g^{-1}=w (g_2Hg_2^{-1})_m w^{-1}.
\]
If $w\ne \sm{}{I_k}{I_m}{}$ then as a permutation in $S_n$, there exist $i\le m$ and $j>m$ such that $w(j)=w(i)+1$. It follows that the root space 
\[
U_{w(i),w(i)+1}=\{I_{n}+ae_{w(i),w(i)+1}: a\in F\}\subseteq U_n(F)\cap wV(F)w^{-1}\subseteq U_n(F)\cap gH_m(F)g^{-1}
\] 
and therefore \eqref{eq nontriv char} follows. Assume now that $w= \sm{}{I_k}{I_m}{}$. Then 
\[
U_n\cap gH_mg^{-1}=\diag(U_k\cap g_2Hg_2^{-1},I_m)
\] 
and \eqref{eq nontriv char} follows from the assumption that $H$ is disjoint from $(U_k,\psi_k)$. The lemma follows.
\end{proof}

\begin{corollary}\label{cor disj}
Let $m,\,n\in \N$, $H=\Sp_{n}$ and $N=m+2n$. In the notation of Lemma \ref{lem disjointness}, for every irreducible, generic representation $\pi$ of $\GL_N(F)$ we have $\Hom_{H_m}(\pi,\triv_{H_m})=0$.
\end{corollary}
\begin{proof}
It follows from \cite[Proposition 2]{MR1159511} that $H$ is disjoint from $(U_{2n},\psi_{2n})$. The corollary now follows from Lemma \ref{lem disjointness} and the paragraph that precedes it. 
\end{proof}

\subsubsection{A generic vanishing result on $\Sp_N$}
We apply the previous corollary to prove the following result.
\begin{proposition}\label{prop genvanish sp}
Let $m,\,n\in \N$, $\pi_1$ a finite length representation of $\GL_m(F)$ and $\pi_2$ an irreducible generic representation of $\GL_{2n}(F)$. Then
\begin{equation}\label{eq vanish period}
\Hom_{H_{m+n,n}(F)}(I_{P_{(m,2n;0)}}^{\Sp_{m+2n}(F)}(\pi_1\abs{\det}^s \otimes \pi_2),\triv_{H_{m+n,n}(F)})=0
\end{equation}
except for finitely many values of $\abs{\cdot}^s$.
\end{proposition}
\begin{proof}
Let $P=M\ltimes U=P_{(m,2n;0)}$ and let $X$ be the $\Sp_N(F)$-conjugacy class of $\inj(I_{m+n},-I_n)$. It follows from \cite[Proposition 6.8]{MR1215304} that there are finitely many $P$-conjugacy classes in $X$. Every $P$-conjugacy class in $X$ is contained in a Bruhat cell $PwP$ for a unique involution $w\in{}_MW_M$. For such $w$ let $M[w]=M\cap wMw^{-1}$. In fact, it follows from \cite[Lemma 3.2]{MR3541705} that $O\mapsto O\cap M[w]w$ is a bijection from the $P$-conjugacy classes in $X\cap PwP$ to the $M[w]$-conjugacy classes in $X\cap M[w]w$.
For every involution $w\in {}_MW_M$ fix a (finite) set of representatives $\Sigma_w$ for the $M[w]$-conjugacy classes in $X\cap M[w]w$.

As a consequence of the geometric lemma of Bernstein and Zelevinsky (see \cite[Theorem 4.2]{MR3541705}), if \eqref{eq vanish period} does not hold then 
there exists an involution $w\in {}_MW_M$ and $x\in \Sigma_w$ such that $\Hom_{L_x(F)}(r_{L,M}(\pi_1\abs{\det}^s \otimes \pi_2),\delta_x)\ne 0$ where $L=M[w]$, $Q$ is the standard parabolic subgroup of $\Sp_N$ with Levi subgroup $L$ and $\delta_x=\delta_Q^{-\frac12}\delta_{Q_x}|_{L_x(F)}$. Here and elsewhere in the proof $r_{L,M}$ denotes the normalized Jacquet functor from representations of $M$ to representations of $L$.

With this notation it therefore suffices to show that, for every involution $w\in {}_MW_M$ and $x\in \Sigma_w$ we have 
\[
\Hom_{L_x(F)}(r_{L,M}(\pi_1\abs{\det}^s \otimes \pi_2),\delta_x)= 0
\]
except for finitely many values of $\abs{\cdot}^s$. For this purpose, we need a more explicit description of the possible Levi subgroups $L$ and stabilizer subgroups $L_x$ obtained in the above construction. We introduce some further notation.

Let $L=M_{(n_1,\dots,n_k;0)}$ be a Levi subgroup of the Siegel Levi of $\Sp_N$. For any $w\in W(L)$ 
there is a unique $\tau\set\in\W_k$ and a representative $\tilde w$ of $w$ in $\Sp_N(F)$,  unique mod the center of $L$, such that 
\begin{equation}\label{eq w}
\tilde w \inj(g_1,\dots,g_k)\tilde w^{-1}=\inj(g_1',\dots,g_k')\ \ \ \text{where}\ \ \ g_j'=\begin{cases} g_{\tau(j)} & j\not\in \set \\  g_{\tau(j)}^* & j\in \set.\end{cases}
\end{equation} 

Fix $w\in {}_MW_M$ and $x\in \Sigma_w$, set $L=M[w]$ and let $\tau\set$ be as in \eqref{eq w}. It is a simple combinatorial observation based on \S\ref{ss reduced}, that the combination of the two properties 
\[
w\in {}_MW_M\ \ \ \text{and}\ \ \ L=M\cap wMw^{-1}
\] 
implies that 
there are decompositions 
\[
m=m_1+m_2+m_3+m_4\ \ \ \text{and}\ \ \ 2n=m_5+m_6+m_7+m_8
\] 
such that 
\begin{itemize}
\item $m_i\in\Z_{\ge0}$, $i\in [1,8]$ 
\item $m_2=m_5$ and $m_3=m_8$ and 
\item $L=M_{(m_1,m_2,m_3,m_4,m_5,m_6,m_7,m_8;0)}$ 
\end{itemize}
and furthermore, $\tau\set\in\W_8$ is such that $\tau=(2,5)(3,8)$ and $\set=\{3,4,7,8\}$. An explication of orbits and stabilizers along the same lines as in \cite[\S3.3]{MR3776281} (which is the case $m=0$) shows that if $\Sigma_w$ is not empty then $m_4$ and $m_7$ (the $m_i$'s with $i\in \set$ fixed by $\tau$) are even and 
that for any $x\in \Sigma_w$, there are compositions $m_i=p_i+q_i$, $i=1,6$ such that $p_1+p_6-(q_1+q_6)=m$ and $L_x$ is $L$-conjugate to 
\[
\{\inj(h_1,h_2,h_3,h_4, h_2, h_6,h_7,h_3^*):h_i\in\GL_{m_i},\ i=2,3, h_i\in \GL_{p_i}\times \GL_{q_i},\ i=1,6, h_i\in \Sp_{m_i/2},\ i=4,7\}
\]
\[
\simeq  (\GL_{p_1}\times \GL_{q_1})\times \GL_{m_2}\times \GL_{m_3} \times \Sp_{m_4/2}\times (\GL_{p_6}\times \GL_{q_6})\times \Sp_{m_7/2}.
\]

Note that
\[
r_{L,M}(\pi_1\abs{\det}^s \otimes \pi_2)=r_{M_{(m_1,m_2,m_3,m_4)},\GL_m}(\pi_1)\abs{\det}^s \otimes r_{M_{(m_5,m_6,m_7,m_8)},\GL_{2n}}(\pi_2).
\]
Since $\pi_1$ is of finite length so is its Jacquet module. It therefore suffices to show that for every irreducible representation $\sigma$ of $M_{(m_1,m_2,m_3,m_4)}(F)$ we have
\[
\Hom_{L_x(F)}(\sigma\abs{\det}^s \otimes r_{M_{(m_5,m_6,m_7,m_8)},\GL_{2n}}(\pi_2),\delta_x)=0
\]
except for finitely many values of $\abs{\cdot}^s$. Since $\sigma$ has a central character, by replacing $L_x$ with its intersection with the center of $L$ this follows immediately if at least one of $m_1$, $m_2$ or $m_3$ is not zero. If $m_1=m_2=m_3=0$ and $\Hom_{L_x(F)}(\sigma\abs{\det}^s \otimes r_{M_{(m_5,m_6,m_7,m_8)},\GL_{2n}}(\pi_2),\delta_x)\ne 0$ then $m_6+m_7=2n$, $p_6-q_6=m=m_4$ (so that $m$ is even and in particular $m\ge 2$) and 
\begin{equation}\label{eq step}
\Hom_{(\GL_{p_6}(F)\times \GL_{q_6}(F))\times \Sp_{m_7/2}(F)}(r_{M_{(m_6,m_7)},\GL_{2n}}(\pi_2),\delta)\ne 0
\end{equation}
for some character $\delta$ of $(\GL_{p_6}(F)\times \GL_{q_6}(F))\times \Sp_{m_7/2}(F)$. If $m_7=0$, that is, if $m_6=2n$, this contradicts the results of 
\cite{MR3928287} (see also \cite{MR4017938}) that imply that $\Hom_H(\pi,\triv_{H})=0$ whenever $\pi$ is a generic irreducible representation of $\GL_{2n}(F)$ and $H=\SL_p(F)\times \SL_q(F)$ where $p+q=2n$ and $q\ne p$. Otherwise, applying the functorial property of Jacquet modules, this contradicts Corollary \ref{cor disj}. Indeed, in the notation of the corollary applied with $(m,n)$ replaced by $(m_6,m_7/2)$, so that $m_6+m_7=2n$, \eqref{eq step} implies that $\Hom_{H_{m_6}}(\pi_2,\triv_{H_{m_6}})\ne 0$ where $H=\Sp_{m_7/2}$. 
The proposition follows.
\end{proof}

\section{Periods of residual representations}\label{periods of res reps-11}
Throughout this section we apply the setup and notation of \S\ref{sec sq int} for $G=\Sp_{m+2n}$, and its parabolic subgroups $Q_i=P_{(1^{(i)},2n,1^{(m-i)};0)}$, $i\in [0,m]$. Set $P=M\ltimes U=P_{(1^{(m)};2n)}$ and $H=H_{m+n,n}=\jj( \Sp_{m+n}\times \Sp_n)$ and let $\pi$ be an irreducible cuspidal automorphic representation of $\GL_{2n}(\A)$ such that $\Res_{s=1}L(\pi,\wedge^2,s)\ne 0$ and $L(\pi,\frac12)\ne 0$.
We apply the notation of \S\ref{ss other res} for $I^m(\pi)$ and $E^m(\varphi,\lambda)$.

\subsection{Exponents of $E^m(\varphi,\lambda)$}\label{ss exp eis}
For $Q=L\ltimes V\in\P_\std$ let
\[
{}_LW_M'=\{w\in {}_LW_M: w(L_m)\subseteq L\}
\]
and recall that for $w\in {}_LW_M$ and $\varphi\in I^m(\pi)$ we have $C_{P,P_w}\varphi=0$ unless $w\in {}_LW_M'$.
It follows from \eqref{eq consterm} that
\[
C_{G,Q}E^m(\varphi,\lambda)=\sum_{w\in {}_LW_M'} E_{Q_w}^Q(M(w,\lambda)C_{P,P_w}\varphi,w\lambda).
\]

Note further that $W_M(L_m)=\{e,w_{L_m}^M\}$ (see \S\ref{ss other res}), ${}_LW_M'\subseteq {}_LW_{L_m}^\circ$ and in terms of the parameterization \S\ref{ss reducedc} of $ {}_LW_{L_m}^\circ$ with its notation for $L$ we have
\[
{}_LW_M'=\{w_\Data: \Data=(D_1,E_1,\dots,D_k,E_k,D_0) \text{ where }m+1\in D_i\text{ for some }i\in [0,k]\}.
\]
For $w=w_\Data\in{}_LW_M'$ we have 
\[
P_w=\begin{cases} P & m+1\in D_0\\ Q_m& m+1\in D_i, \ i\in [1,k].\end{cases}
\]
For future reference let $i(\Data)\in [0,k]$ be the integer such that $m+1\in D_i$.
Furthermore, set $\theta_0=(0^{(m)},(-\frac12)^{(2n)})$ and write accordingly 
\[
\theta_w=\begin{cases} 0 & i(\Data)=0\\ \theta_0& i(\Data)\in [1,k].\end{cases}
\]
As a consequence of \S\ref{ss gsr} we have $C_{P,Q_m}: I^m(\pi,\lambda)\rightarrow I_m(\pi,\lambda+\theta_0)$ is a map of representations for any $\lambda\in \aaa_{M,\C}^*$. 
The twisted action of $A_Q$ satisfies
\[
a\cdot E_{Q_w}^Q(M(w,\lambda)C_{P,P_w}\varphi,w\lambda)=e^{\sprod{w(\lambda+\theta_w)}{\HC_Q(a)}} E_{Q_w}^Q(M(w,\lambda)C_{P,P_w}\varphi,w\lambda),\ \ \ a\in A_Q.
\]

\subsection{Vanishing of the regularized period}\label{ss vanishing period} 

We first show that the regularized period of the Eisenstein series at hand makes sense as a meromorphic function in order to then prove that it identically equals zero. 
\begin{lemma}\label{lem reg}
We have $E^m(\varphi,\lambda)\in \Aut_G^{H-\reg}$ for $\varphi\in I^m(\pi)$ and $\lambda\in \aaa_{M,\C}^*$ away from finitely many hyperplanes.
\end{lemma}
\begin{proof}
By \eqref{eq reg cond} and the analysis of exponents in \S\ref{ss exp eis} it suffices to show that the equation 
\[
(w(\lambda+\theta_w)+\rho_{Q,\sigma})_Q=0
\]
defines a non-trivial closed condition on $\aaa_{M,\C}^*$ for every maximal $Q\in \P_\std$, $w\in {}_LW_M'$ and $\sigma\in \Sigma_Q$. 
Let $Q=P_{(n_1;r)}$ and $w=w_\Data\in  {}_LW_M'$ with $\Data=(D_1,E_1,D_0)$ a partition of $[1,m+1]$ with $m+1\in D_0\sqcup D_1$. 
For $\lambda=(\lambda_1,\dots,\lambda_m,0^{(2n)})\in \aaa_{M,\C}^*$ we have
\[
w\lambda=\begin{cases}
(\lambda_{x_1},\dots,\lambda_{x_{d_1}},-\lambda_{y_1},\dots,-\lambda_{y_{e_1}},\lambda_{z_1},\dots,\lambda_{z_{d_0}},0^{(2n)}) & m+1\in D_0\\
(\lambda_{x_1},\dots,\lambda_{x_{d_1}},0^{(2n)},-\lambda_{y_1},\dots,-\lambda_{y_{e_1}},\lambda_{z_1},\dots,\lambda_{z_{d_0}}) & m+1\in D_1
\end{cases}
\]
where
\[
D_0\setminus \{m+1\}=\{z_1<\cdots<z_{d_0}\},\ D_1\setminus \{m+1\}=\{x_1<\cdots<x_{d_1}\},\ \text{and}\ E_1=\{y_1>\cdots>y_{e_1}\}.
\]
It follows that the equation $(w\lambda+\mu)_Q=0$ defines a non-trivial closed condition on $\lambda$ for any $\mu\in \aaa_M^*$ unless $n_1=2n$ and $m+1\in D_1$ (or equivalently $D_1=\{m+1\}$ and $D_0=[1,m]$ and $E_1$ is empty). When $Q=P_{(2n;m)}$ and $m+1\in D_1$ we have $(w\lambda)_Q=0$
and it suffices to show that in this case $(w\theta_w+\rho_{Q,\sigma})_Q\ne 0$
for all $\sigma\in \Sigma_Q$.
We have
\[
w\theta_w=((-\frac12)^{(2n)},0^{(m)}),\ \text{and}\  \rho_Q=((m+n+\frac12)^{(2n)},0^{(m)}).
\]
Applying Lemma \ref{lem sp parabolics}, we write $\sigma=\sigma_{\tbl}$ where
\[
\tbl=\begin{array}{ |c|c||c|} 
 \hline
 2n &  m & m+2n \\
 \hline\hline
b_1 & r_2 & m+n \\
\hline
a_1 & r_1 & n \\
\hline
\end{array}\in \Tau_{m+n,n}^Q.
\]
Note that $\tbl$ is determined by $a_1$ and that $a_1$ varies in $[\max(0,n-m),n]$. 
Using \eqref{eq par of h} and \eqref{eq mhintersection sp} we have
\[
2\rho_{Q\cap \sigma^{-1}H\sigma}=((2m+a_1+1)^{(2n-a_1)},(2n+1-a_1)^{(a_1)},0^{(m)}).
\]
It follows that $(w\theta_w+\rho_{Q,\sigma})_Q=(\rho^{(2n)},0^{(m)})$ where
\[
2n\rho=n[2(m+n)+1]-a_1(2n+1-a_1)-(2n-a_1)(2m+a_1+1)-n=2f(a_1)
\]
where
\[
f(x)=x^2+(m-2n)x+n[n-m-1].
\]
Note that 
\[
f(n)=-n=f(n-m)<0\text{ and } f(0)=n[n-(m+1)].
\]
If $n\le m$ then also $f(0)<0$. Since $f$ is convex it follows that $f(x)<0$ for all $\max(0,n-m)\le x\le n$ and consequently $\rho<0$. The lemma follows.
\end{proof}

\begin{lemma}\label{lem global vanishmn}
We have
\[
\P_H(E^m(\varphi,\lambda))=0,\ \ \ \varphi\in I^m(\pi),\ \lambda\in\aaa_{M,\C}^*.
\]
\end{lemma}
\begin{proof}
Recall from \S\ref{ss eis ct} that $E^m(\varphi,\lambda)$ is defined by an absolutely convergent series, and in particular is holomorphic in $\lambda$, for $\lambda\in \aaa_{M,\C}^*$ such that $\Re\sprod{\lambda}{\alpha^\vee}\gg 1$ for all $\alpha\in \Delta_P$, $\varphi\in I^m(\pi)$. Let $\Omega_0$ be the set of such $\lambda$ that, in addition, lie off the finitely many hyperplanes of Lemma \ref{lem reg}. Thus, $\Omega_0$ is a non-empty open subset of $\aaa_{M,\C}^*$ and by Lemma \ref{lem reg} the regularized period $\P_H(E^m(\varphi,\lambda))$ is defined for every $\varphi\in I^m(\pi)$ and $\lambda\in \Omega_0$; as in \cite[Theorem 8.4.1 (4)]{MR2010737} it extends to a meromorphic function of $\lambda$ on $\aaa_{M,\C}^*$. It follows from \cite[Theorem 4.1 (4)]{MR4411860} that $\varphi\mapsto \P_H(E^m(\varphi,\lambda))$ is $H(\A)$-invariant on $I^m(\pi,\lambda)$ for every $\lambda\in \Omega_0$.

Assume by contradiction that $\P_H(E^m(\varphi,\lambda))$ is not identically zero and fix $\varphi_0\in I^m(\pi)$ such that $\lambda\mapsto \P_H(E^m(\varphi_0,\lambda))$ is not identically zero. Set
\[
\Omega=\{\lambda\in \Omega_0: \P_H(E^m(\varphi_0,\lambda))\ne 0\}.
\]
Then $\Omega$ is a non-empty open subset of $\aaa_{M,\C}^*$. For every $\lambda\in \Omega$ the form $\varphi\mapsto \P_H(E^m(\varphi,\lambda))$ is therefore a non-zero element of $\Hom_{H(\A)}(I^m(\pi,\lambda),\triv_{H(\A)})$ and consequently
\[
\Hom_{H(F_v)}(I^m(\pi,\lambda)_v,\triv_{H(F_v)})\ne 0,\ \ \ \lambda\in \Omega
\]
for every finite place $v$ of $F$. Recall that $I^m(\pi,\lambda)$ is a quotient of $I_m(\pi,\lambda+\theta)$ where $\theta=-\theta_0=(0^{(m)},(\frac12)^{(2n)})$ (see \S\ref{ss exp eis}). Consequently,
\begin{equation}\label{eq nonvan om}
\Hom_{H(F_v)}(I_m(\pi,\lambda+\theta)_v,\triv_{H(F_v)})\ne 0,\ \ \ \lambda\in \Omega.
\end{equation}

Fix a finite place $v$ of $F$ and let $q_v$ be the cardinality of the residue field of $F_v$. Fix $\lambda'=(\lambda_1',\dots,\lambda_m',0^{(2n)})\in \Omega$ and let $I_v(\lambda')$ be the unramified principal series representation of $\GL_m(F_v)$ induced from $\diag(t_1,\dots,t_m)\mapsto \prod_{i=1}^m \abs{t_i}^{\lambda_i'}$; it is a representation of finite length. For $s\in \C$ set $\lambda(s)=\lambda'+(s^{(m)},0^{(2n)})\in \aaa_{M,\C}^*$. Since $\lambda(s)+\theta=\lambda'+(s^{(m)},(\frac12)^{(2n)})$, transitivity of induction gives
\begin{equation}\label{eq transitivity}
I_m(\pi,\lambda(s)+\theta)_v\simeq I_{P_{(m,2n;0)}(F_v)}^{G(F_v)}\big(I_v(\lambda')\abs{\det}^{s}\otimes \pi_v\abs{\det}^{\frac12}\big),\ \ \ s\in \C.
\end{equation}
Note that $s$ enters the right hand side of \eqref{eq transitivity} only through the twist $\abs{\det}^{s}$ of its first factor.

Recall that $\pi_v$, and therefore also $\pi_v\abs{\det}^{\frac12}$, is an irreducible generic representation of $\GL_{2n}(F_v)$. Applying Proposition \ref{prop genvanish sp} to $\pi_1=I_v(\lambda')$ and $\pi_2=\pi_v\abs{\det}^{\frac12}$, and taking \eqref{eq transitivity} into account, we conclude that there is a finite subset $X$ of $\C^*$ such that
\[
\Hom_{H(F_v)}(I_m(\pi,\lambda(s)+\theta)_v,\triv_{H(F_v)})=0\ \ \ \text{whenever}\ \ \ q_v^{-s}\not\in X.
\]
On the other hand, $V=\{s\in \C:\lambda(s)\in \Omega\}$ is an open subset of $\C$ containing $0$ and by \eqref{eq nonvan om} the above Hom space is non-zero for every $s\in V$. Therefore $V\subseteq \{s\in \C: q_v^{-s}\in X\}$, which is a finite union of cosets of $\frac{2\pi \iii}{\log q_v}\Z$ and in particular a discrete subset of $\C$. This contradicts the fact that $V$ is a non-empty open subset of $\C$, and the lemma follows. 
\end{proof}

\subsection{Failure of $H$-integrability for $\E^m(\pi)$}
Ginzburg, Rallis and Soudry prove in \cite{MR1740991} that the period integral over $[H_{n,n}]$ converges on $\E^0(\pi)$. We show here that if $m\ge 1$ the situation for $\E^m(\pi)$ is different. In fact, even the regularized period integral $\P_H$ is not defined on $\E^m(\pi)$ by Zydor's regularization \cite[Theorem 4.1 (3)]{MR4411860}.
\begin{lemma}\label{lem not reg}
For $m\ge 1$ the representation space $\E^m(\pi)$ is not contained in $\Aut_G^{H-\reg}$. In particular, the period integral over $[H]$ does not converge on $\E^m(\pi)$.
\end{lemma}
\begin{proof}
The second part of the lemma follows from the first and \cite[Theorem 4.5]{MR4411860}. For the first part, it follows from Proposition \ref{prop sqrint}, in its notation, and transitivity of the constant term that $(-\lambda_0)_Q$ is an exponent of $\E^m(\pi)$ along any $Q$ such that  $Q_m\subseteq Q\in \P_\std$. Let $Q=P_{(1;m+2n-1)}$ and let $\sigma=\sigma_{\tbl}$ with
\[
\tbl=\begin{array}{ |c|c||c|} 
 \hline
1 &  m+2n-1 & m+2n \\
 \hline\hline
0 & m+n & m+n \\
\hline
1 & n-1 & n \\
\hline
\end{array}\in \Tau_{m+n,n}^Q.
\]
We have $(-\lambda_0)_Q=(-m,0^{(m+2n-1)})$, $\rho_Q=(m+2n,0^{(m+2n-1)})$ and $\sigma(Q)\cap H=\jj(\Sp_{m+n}\times P_{(1;n-1)})$ so that 
$2\rho_{Q\cap \sigma^{-1}H\sigma}=(2n,0^{(m+2n-1)})$ and it follows that $(-\lambda_0+\rho_{Q,\sigma})_Q=0$. The lemma follows from \eqref{eq reg cond}.
\end{proof}

\subsection{A preliminary formula for the period of a truncated Eisenstein series}\label{ss prel}
Let $Q=L\ltimes V=P_{(n_1,\dots,n_k;r)}\in \P_\std$. If $r\ge n$ let $\tbl_{Q,0}\in \Tau_{m+n,n}^Q$ be defined by
\[
\tbl_{Q,0}=
\begin{array}{ |c|c|c|c||c|} 
 \hline
 n_1 & \cdots & n_k & r & N \\
 \hline\hline
n_1 & \cdots & n_k & r-n & m+n \\
\hline
0 & \cdots & 0 & n & n \\
\hline
\end{array}
\]
and if $i\in [1,k]$ is such that $n_i\ge n$ let 
\[
\tbl_{Q,i}=
\begin{array}{|c|c|c|c|c|c|c|c||c|} 
 \hline
 n_1 & \cdots & n_{i-1} & n_i & n_{i+1} & \cdots  & n_k & r & N \\
 \hline\hline
n_1 & \cdots & n_{i-1} & n_i-n & n_{i+1} & \cdots &  n_k & r & m+n \\
\hline
0 & \cdots & 0 & n & 0 & \cdots & 0 & 0 & n \\
\hline
\end{array}\in \Tau_{m+n,n}^Q.
\]
Apply the notation of \S\ref{ss reducedc} and \S\ref{ss exp eis} for the parameterization of ${}_LW_M'\subseteq {}_LW_{L_m}^\circ$. For $w=w_\Data\in {}_LW_M'$ with $\Data=(D_1,E_1,\dots,D_k,E_k,D_0)$ (a partition of $[1,m+1]$) let $i=i(\Data)\in [0,k]$ (so that $m+1\in D_i$) and set 
\[
\sigma_{Q,w}=\sigma_{\tbl_{Q,i}}\ \ \  \text{and}\ \ \  \eta[Q,w;\lambda]=(w(\lambda+\theta_w)+\rho_{Q,\sigma_{Q,w}})_Q,\ \ \ \lambda\in \aaa_{M,\C}^*,
\] 
where $\theta_w$ is defined in \S\ref{ss exp eis}. Recall that, by \S\ref{ss exp eis}, the eigenvalue for the twisted action of $A_Q$ of the summand $E_{Q_w}^Q(M(w,\lambda)C_{P,P_w}\varphi,w\lambda)$ of $C_{G,Q}E^m(\varphi,\lambda)$ is $(w(\lambda+\theta_w))_Q$; the exponent $\eta[Q,w;\lambda]$ is therefore the exponent $\lambda[Q,i]+(\rho_{Q,\sigma})_Q$ of \eqref{eq Zydor} attached to this summand and to $\sigma=\sigma_{Q,w}$.
\begin{lemma}\label{lem prel per}
We have (see \eqref{eq truncint} for the notation)
\begin{multline*}
\P_H^{G,T}(\varphi,\lambda)=\sum_{G\ne Q\in\P_\std}(-1)^{\dim\aaa_Q^G-1}\ v_Q\sum_{w\in {}_LW'_M}\frac{e^{\sprod{\sigma_{Q,w}^{-1}T}{ \eta[Q,w;\lambda]}}}{\prod_{\varpi\in \hat\Delta_Q^\vee}\sprod{ \eta[Q,w;\lambda]}{\varpi} }\times \\
\int_{K_H} \int_{[L\cap \sigma_{Q,w}^{-1} H\sigma_{Q,w}]^{1,L}} \delta_{Q\cap \sigma_{Q,w}^{-1} H\sigma_{Q,w}}^{-1}(\ell) \Lambda^{\sigma_{Q,w}^{-1}T,Q}E_{Q_w}^Q(\ell\sigma_{Q,w}^{-1} k,M(w,\lambda)C_{P,P_w}\varphi,w\lambda)\ d\ell \ dk,\ \ \ \varphi\in I^m(\pi).
\end{multline*}
\end{lemma}
\begin{proof}
Based on Lemmas \ref{lem global vanishmn} and  \ref{lem sp parabolics} it follows from \eqref{eq Zydor} that
\begin{multline*}
\P_H^{G,T}(\varphi,\lambda)=\sum_{G\ne Q\in\P_\std}(-1)^{\dim\aaa_Q^G-1}\ v_Q\sum_{w\in {}_LW'_M}\sum_{\tbl\in \Tau_{m+n,n}^Q}\frac{e^{\sprod{\sigma_{\tbl}^{-1}T}{(w(\lambda+\theta_w)+\rho_{Q,\sigma_{\tbl}})_Q}}}{\prod_{\varpi\in \hat\Delta_Q^\vee}\sprod{w(\lambda+\theta_w)+\rho_{Q,\sigma_{\tbl}}}{\varpi} }\times \\
\int_{K_H} \int_{[L\cap \sigma_{\tbl}^{-1} H\sigma_{\tbl}]^{1,L}} \delta_{Q\cap \sigma_{\tbl}^{-1} H\sigma_{\tbl}}^{-1}(\ell) \Lambda^{\sigma_{\tbl}^{-1}T,Q}E_{Q_w}^Q(\ell\sigma_{\tbl}^{-1} k,M(w,\lambda)C_{P,P_w}\varphi,w\lambda)\ d\ell \ dk.
\end{multline*}
Fix $G\ne Q=L\ltimes V\in\P_\std$, let $w=w_\Data\in {}_LW'_M$ and let $i=i(\Data)\in [0,k]$ (see \S\ref{ss reducedc} and \S\ref{ss exp eis}) so that $\sigma_{Q,w}=\sigma_{\tbl_{Q,i}}$. 
To complete the proof of the lemma it therefore suffices to show that the inner integral 
\[
 \int_{[L\cap \sigma_{\tbl}^{-1} H\sigma_{\tbl}]^{1,L}} \delta_{Q\cap \sigma_{\tbl}^{-1} H\sigma_{\tbl}}^{-1}(\ell) \Lambda^{\sigma_{\tbl}^{-1}T,Q}E_{Q_w}^Q(\ell\sigma_{\tbl}^{-1} k,M(w,\lambda)C_{P,P_w}\varphi,w\lambda)\ d\ell
\]
vanishes for $\tbl\ne \tbl_{Q,i}$ (that is if $\sigma_{\tbl}\ne \sigma_{Q,w}$).
Note that for 
\begin{equation}\label{eq type tm}
\tbl=\begin{array}{ |c|c|c|c||c|} 
 \hline
 n_1 & \cdots & n_k & r & N \\
 \hline\hline
b_1 & \cdots & b_k & r_2 & m+n \\
\hline
a_1 & \cdots & a_k & r_1 & n \\
\hline
\end{array}\in \Tau_{m+n,n}^Q
\end{equation}
and $\sigma=\sigma_{\tbl}$, by \eqref{eq mhintersection sp}\footnote{Note that, since $H=H_{m+n,n}$, the first data row of a table in $\Tau_{m+n,n}^Q$ is the one summing to $a=m+n$; throughout this section we label its entries $b_1,\dots,b_k,r_2$ and those of the second row $a_1,\dots,a_k,r_1$, which is the opposite of the labelling used in \eqref{eq typ t} and \eqref{eq t typ sp}. The formulas \eqref{eq par of h} and \eqref{eq mhintersection sp} are to be read accordingly.} we have
\[
L\cap \sigma^{-1} H\sigma=\inj(M_{(b_1,a_1)},\dots,M_{(b_k,a_k)};H_{r_2,r_1}). 
\]
Assume that $\tbl\ne \tbl_{Q,i}$.
First apply Corollary \ref{cor restrict partol sp} to the definition of $\Lambda^{\sigma^{-1}T,Q}$.
We argue separately for the vanishing in the two cases $i=0$ or $i\in [1,k]$.

 If $i=0$ then $r_1<n$ (since $\tbl\ne \tbl_{Q,0}$) and
we see that the inner integral factors through $\P_{H_{r_2,r_1}}^{\Sp_r, T'}(\xi,\mu)$ for some $T'$ positive enough, $\xi\in I^{r-2n}(\pi)$ and $\mu\in\aaa_{M',\C}^*$, where $I^{r-2n}(\pi)$ (respectively $I_{r-2n}(\pi)$ below) is defined with respect to the group $\Sp_r$, with $r-2n$ in place of $m$, and $M'=M_{(1^{(r-2n)};2n)}$ is the corresponding Levi subgroup of $\Sp_r$.
It follows from Proposition \ref{prop Sp vanishing} that 
\[
\P_{H_{r_2,r_1}}^{\Sp_r, T'}(\eta,\mu+(0^{(r-2n)},t^{(2n)}))=0,\ \ \ \eta\in I_{r-2n}(\pi), \ \mu\in\aaa_{M',\C}^*\ \text{and}\ t\in \C
\]
(the block of size $2n$ satisfies $2\min(r_1,r_2)\le 2r_1<2n$, so that the assumption of Proposition \ref{prop Sp vanishing} holds).
Realizing $\E^0(\pi)$ as a residue, the residue map $\lim_{t\to \frac12}(t-\frac12)\P_{H_{r_2,r_1}}^{\Sp_r, T'}(\eta,\mu+(0^{(r-2n)},t^{(2n)}))$ together with \S\ref{ss other res} gives that 
\[
\P_{H_{r_2,r_1}}^{\Sp_r, T'}(\xi,\mu)=0,\ \ \ \xi\in I^{r-2n}(\pi),\ \ \ \mu\in\aaa_{M',\C}^*.
\]
Otherwise, $i\in [1,k]$ and $a_i<n$ (since $\tbl\ne \tbl_{Q,i}$).  
The inner integral now factors through $\P_{\GL_{b_i}\times \GL_{a_i},\chi}^{\GL_{n_i},T'}(\xi,\mu)$ for some $T'$ positive enough,  $\xi\in I_{P_{(1^{(d)},2n,1^{(e)})}}^{\GL_{n_i}}(\triv_{\GL_1^d}\otimes \pi\abs{\det}^{-\frac12}\otimes \triv_{\GL_1^e})$ and some character $\chi$ of $\GL_{b_i}(\A)\times \GL_{a_i}(\A)$ where $d=\abs{D_i\setminus\{m+1\}}$ and $e=\abs{E_i}$. 
It follows from Proposition \ref{prop GL vanishing} (applied with the block of size $2n\ge 2$, for which $2\min(a_i,b_i)\le 2a_i<2n$) that $\P_{\GL_{b_i}\times \GL_{a_i},\chi}^{\GL_{n_i},T'}(\xi,\mu)=0$. 
The lemma follows.
\end{proof}

\subsection{Explication of exponents}
For $Q=L\ltimes V\in \P_\std$ and  $w\in {}_LW_M'$ let (see \S\ref{ss prel})
\[
\rho[Q,w]=\rho_{Q,\sigma_{Q,w}}.
\]

Write $w=w_\Data$ and let $i=i(\Data)$ according to \S\ref{ss reducedc} and \S\ref{ss exp eis}.
If $i=0$ let $\wperm_i$ be the identity in $W$. Otherwise, write $Q=P_{(n_1,\dots,n_k;r)}$, set $\nu_{i-1}=n_1+\cdots+n_{i-1}$ (note that $\nu_{i-1}\le m$ in this case) and let $\wperm_i\in S_N$ be represented in $\GL_N$ by
\[
\tilde\wperm_i=\begin{pmatrix} I_{\nu_{i-1}} & & \\ & & I_{2n} \\ & I_{m-\nu_{i-1}} & \end{pmatrix},
\]
so that $\wperm_i$ moves the size $2n$ block to the beginning of the $i$th block of $Q$ (in particular, $\wperm_i$ depends on $Q$ and not only on $i$) and
\[
\wperm_i\lambda_0=(m,\dots,m+1-\nu_{i-1},(\tfrac12)^{(2n)},m-\nu_{i-1},\dots,1)\ \ \ \text{for}\ \ \ \lambda_0=(m,\dots,1,(\tfrac12)^{(2n)}).
\]
Note that $2n\le n_i$, so that the block of $2n$ coordinates of $\wperm_i\lambda_0$ is contained in the $i$th block associated to blocks of $Q$.

\begin{lemma}\label{lem expexp}
Let $Q=L\ltimes V=P_{(n_1,\dots,n_k;r)}\in \P_\std$, $w=w_\Data\in {}_LW_M'$ and $i=i(\Data)$. With the above notation we have
\[
\rho[Q,w]_Q=-(\wperm_i\lambda_0)_Q \ \ \ \text{where}\ \ \ \lambda_0=(m,\dots,1,(\frac12)^{(2n)}).
\]
\end{lemma}
\begin{proof}
We begin by the explication of $\rho[Q,w]_Q$.
We have
\[
\rho_Q=(x_1^{(n_1)},\dots,x_k^{(n_k)},0^{(r)})\ \ \ \text{where} \ \ \ x_j=m+2n+\frac{1-n_j}2-\sum_{\ell=1}^{j-1} n_\ell,\ \ \ j\in [1,k].
\]
Set $\sigma=\sigma_{Q,w}$.
By definition $\sigma=\sigma_{\tbl_{Q,i}}$ and from  \eqref{eq par of h} and \eqref{eq mhintersection sp} we have
\[
2\rho_{Q\cap \sigma^{-1} H\sigma}=\begin{cases} (y_1^{(n_1)},\dots,y_k^{(n_k)},0^{(r)}) & i=0 \\ (y_1^{(n_1)},\dots,y_{i-1}^{(n_{i-1})}, y_i^{(n_i-n)},(n+1)^{(n)},y_{i+1}^{(n_{i+1})},\dots,y_k^{(n_k)},0^{(r)}) & i\in [1,k] \end{cases}
\]
where
\[
y_j=2(m+n)+1-n_j-2\sum_{\ell=1}^{j-1} n_\ell+\delta(j)\ \ \ \text{and} \ \ \ \delta(j)=\begin{cases} 0 & i=0 \text{ or } j<i\\ n& j=i\\ 2n & j>i\end{cases} \text{ for }j\in [1,k].
\]
It follows that
\[
(2\rho_{Q\cap \sigma^{-1} H\sigma})_Q=(z_1^{(n_1)},\dots,z_k^{(n_k)},0^{(r)})
\]
where
\[ 
n_jz_j=n_j[2(m+n)+1-n_j-2\sum_{\ell=1}^{j-1} n_\ell]+\begin{cases} 0 & i=0 \text{ or }j<i\\ 2n[\left(\sum_{\ell=1}^i n_\ell\right)-(m+n)]& j=i \\ 2nn_j& j>i.\end{cases}
\]
We conclude that
\[
\rho[Q,w]_Q=(\varrho_1^{(n_1)},\dots,\varrho_k^{(n_k)},0^{(r)})
\]
where
\[
2n_j\varrho_j=n_j[n_j+2\sum_{\ell=1}^{j-1}n_\ell-(2m+1)]+\begin{cases}
0& i=0 \text{ or }j<i \\
4n[(m+n)-\sum_{\ell=1}^i n_\ell] & j=i\in [1,k]\\
 -4nn_j & j>i\in [1,k].
\end{cases}
\]
Set $\nu_0=0$ and $\nu_\ell=\sum_{j=1}^\ell n_j$, $\ell\in [1,k]$.
Note that
\[
2\sum_{j=1}^\ell n_j\varrho_j=\begin{cases} \nu_\ell[\nu_\ell-(2m+1)] & i=0 \text{ or }\ell<i \\
 \nu_\ell[\nu_\ell-(2m+4n+1)]+4n(m+n) & \ell\ge i\in [1,k].
\end{cases}
\]

Note further that for $\mu=(\mu_1,\dots,\mu_{m+2n})\in \aaa_{0,\C}^*$ we have
\[
\mu_Q=(\xi_1^{(n_1)},\dots,\xi_k^{(n_k)},0^{(r)})\ \ \ \text{where}\ \ \ \sum_{j=1}^\ell n_j\xi_j=\sum_{j=1}^{\nu_\ell} \mu_j.
\]
We apply this for $\mu=\wperm_i\lambda_0$. We have
\[
2\sum_{j=1}^\ell n_j\xi_j=\begin{cases}2\sum_{j=1}^{\nu_\ell} (m+1-j) & i=0\text{ or }\ell<i \\
2n+2\sum_{j=1}^{\nu_\ell-2n} (m+1-j) & \ell\ge i\in [1,k]
\end{cases}=-2\sum_{j=1}^\ell n_j\varrho_j.
\]
Recall the explication of $ \hat\Delta_Q^\vee$ in \S\ref{ss sqrint}: for $\xi=(\xi_1^{(n_1)},\dots,\xi_k^{(n_k)},0^{(r)})\in \aaa_Q^*$ we have $\sprod{\xi}{\varpi_{\nu_\ell}}=c_\ell\sum_{j=1}^\ell n_j\xi_j$ with $c_\ell=c_\ell(Q)>0$ independent of $\xi$. We showed above that $\sprod{\rho[Q,w]}{\varpi}=-\sprod{\wperm_i\lambda_0}{\varpi}$ for all $\varpi\in  \hat\Delta_Q^\vee$. Since $ \hat\Delta_Q^\vee$ is a basis of $\aaa_Q$, the lemma follows.
\end{proof}

Recall that $\theta_0$ and $\theta_w$ are defined in \S\ref{ss exp eis}, $\eta[Q,w,\lambda]$ in \S\ref{ss prel} and that 
\[
\mu_0=(m,\dots,1,0^{(2n)})\in \aaa_M^*\ \ \ \text{and} \ \ \ \lambda_0=\mu_0-\theta_0.
\]

\begin{lemma}\label{lem exp zeros}
For $Q=L\ltimes V=P_{(n_1,\dots,n_k;r)}\in \P_\std$, $w=w_\Data\in{}_LW_M'$ with $\Data=(D_1,E_1,\dots,D_k,E_k,D_0)$ (see  \S\ref{ss reducedc} and \S\ref{ss exp eis}) and $\ell\in [1,k]$ we have
\[
\sprod{\eta[Q,w;\mu_0]}{\varpi_{\nu_\ell}}\le 0,\ \ \ \ell\in[1,k]
\]
and equality holds if and only if 
\begin{enumerate}
\item\label{part i} either $i(\Data)=0$ or $\ell<i(\Data)$ and 
\item\label{part d} $[1,\nu_\ell]=\sqcup_{j=1}^\ell D_j$.
\end{enumerate}
In particular, $\eta[Q,w;\mu_0]=0$ if and only if $w$ is the identity and $r\ge 2n$. (Here $\nu_\ell$ and $\varpi_{\nu_\ell}$ are as in \S\ref{ss sqrint}.)
\end{lemma}
\begin{proof}
Write $i=i(\Data)$.
For $\lambda=(\lambda_1,\dots,\lambda_m,t^{(2n)})\in\aaa_{L_m,\C}^*$ note that
\begin{multline}\label{eq fw}
\sprod{w\lambda}{\varpi_{\nu_\ell}}=c_\ell\left[\left(\sum\limits_{x\in \sqcup_{j=1}^\ell D_j\setminus\{m+1\}} \lambda_x\right)-\left(\sum\limits_{y\in \sqcup_{j=1}^\ell E_j} \lambda_y\right)+2nt\delta(w,\ell)\right] \ \ \ \text{where}\\ \delta(w,\ell)=\begin{cases} 0 & i=0\text{ or }\ell<i\in [1,k]\\ 1 & \ell\ge i\in [1,k]\end{cases}
\end{multline}
and $c_\ell=c_\ell(Q)\in\{1,\frac12\}$ is the positive constant of \S\ref{ss sqrint}, so that $\sprod{\mu}{\varpi_{\nu_\ell}}=c_\ell\sum_{j=1}^{\nu_\ell}\mu_j$ for $\mu\in\aaa_{0,\C}^*$.
Applying this for $\lambda=\mu_0+\theta_w$ (see \S\ref{ss exp eis}) and observing that in this case $\lambda_1>\cdots>\lambda_m>t$ we conclude that
\[
\sprod{w(\mu_0+\theta_w)}{\varpi_{\nu_\ell}}\le \sprod{\wperm_i(\mu_0+\theta_w)}{\varpi_{\nu_\ell}}
\]
and equality holds if and only if the set consisting of the first $\nu_\ell$ coordinates is the same for $w(\mu_0+\theta_w)$ and $\wperm_i(\mu_0+\theta_w)$.

If either $i=0$ or $\ell< i$ then 
\[
\sprod{w\theta_w}{\varpi_{\nu_\ell}}=0=\sprod{\wperm_i\theta_w}{\varpi_{\nu_\ell}},\ \ \  \sprod{\wperm_i\mu_0}{\varpi_{\nu_\ell}}=\sprod{\wperm_i \lambda_0}{\varpi_{\nu_\ell}}
\] 
and the set of first $\nu_\ell$ coordinates of $w(\mu_0+\theta_w)$ and $\wperm_i(\mu_0+\theta_w)$ is the same if and only if $[1,\nu_\ell]=\sqcup_{j=1}^\ell D_j$. 
Combined with Lemma \ref{lem expexp} it follows that
\[
\sprod{\eta[Q,w,\mu_0]}{\varpi_{\nu_\ell}}=\sprod{w\mu_0-\wperm_i\lambda_0}{\varpi_{\nu_\ell}}\le \sprod{\wperm_i(\mu_0-\lambda_0)}{\varpi_{\nu_\ell}}=0
\]
and equality holds if and only if $[1,\nu_\ell]=\sqcup_{j=1}^\ell D_j$.
 
If $\ell\ge i\in [1,k]$ then $\sprod{\wperm_i\theta_0}{\varpi_{\nu_\ell}}=-c_\ell n$ and therefore
\[
\sprod{\wperm_i(\mu_0+\theta_w)}{\varpi_{\nu_\ell}}=\sprod{\wperm_i\lambda_0}{\varpi_{\nu_\ell}}-2c_\ell n.
\]
Combined with Lemma \ref{lem expexp} it follows that 
\[
\sprod{\eta[Q,w,\mu_0]}{\varpi_{\nu_\ell}}=\sprod{w(\mu_0+\theta_w)-\wperm_i\lambda_0}{\varpi_{\nu_\ell}}\le \sprod{\wperm_i(\mu_0+\theta_w-\lambda_0)}{\varpi_{\nu_\ell}}=-2c_\ell n<0.
\]

Since $\{\varpi_{\nu_\ell}: \ell\in [1,k]\}$ is a basis of $\aaa_Q$ and the conditions \eqref{part i} and \eqref{part d} are satisfied for every $\ell\in [1,k]$ if and only if $w$ is the identity and $i(\Data)=0$ the rest of the lemma follows.
\end{proof}

\subsection{The regularized period of the residue}
For $\phi\in \Aut_G$ write $\P_0(\phi)$ for the coefficient of the zero exponent of the exponential polynomial function $\int_{[H]}\Lambda^T\phi(h)\ dh$.
It is a priori a polynomial in $T$ and it is independent of $T$ if $\phi\in \Aut_G^{H-\reg}$. Our main theorem is a formula for $\P_0(\phi)$ for $\phi\in \E^m(\pi)$ that, in particular, shows that it is independent of $T$ and defines an $H(\A)$-invariant linear form on $\E^m(\pi)$. 

Let
\[
\LL(\varphi)=\int_{K_H}\int_{[H_{n,n}]} \varphi(\diag(I_m,h,I_m)k)\ dh \ d k,\ \ \ \varphi\in I^m(\pi).
\]
It follows from \cite{MR1740991} that the inner integral is not identically zero and the argument in \cite[Proposition 2]{MR1142486} then shows that $\LL$ is not identically zero. It is also easy to see that it defines an $H(\A)$-invariant linear form on $I^m(\pi,\mu_0)$ for $\mu_0=(m,\dots,1,0^{(2n)})$. This is an immediate consequence of the identity $\mu_0+\rho_P-2\rho_{P\cap H}=0$ and the Iwasawa decomposition $H(\A)=(P\cap H)(\A) K_H$. 
\begin{theorem}\label{thm main}
There is a positive constant $\upsilon$ that depends only on normalization of measures such that 
\[
\P_0(\E^m(\varphi))=\upsilon \LL(\varphi),\ \ \ \varphi\in I^m(\pi).
\]
In particular,
\begin{itemize}
\item $\P_0|_{\E^m(\pi)}$ is independent of $T$ and defines a non-zero $H(\A)$-invariant linear form;
\item The linear form $\LL$ factors through the projection $\E^m: I^m(\pi,\mu_0)\rightarrow \E^m(\pi)$.
\end{itemize}
\end{theorem}
\begin{proof}
Recall the realization  \eqref{eq new real} of $\E^m(\pi)=\{\E^m(\varphi):\varphi\in I^m(\pi)\}$.
It follows from \eqref{eq com res} and Lemma \ref{lem prel per} that
\begin{equation}\label{eq thm step}
\int_{[H]}\Lambda^T\E^m(h,\varphi)\ dh=\lim_{\lambda\to \mu_0}p(\lambda)\sum_{G\ne Q\in\P_\std}(-1)^{\dim\aaa_Q-1}\ v_Q\sum_{w\in {}_LW'_M}\frac{e^{\sprod{\sigma_{Q,w}^{-1}T}{ \eta[Q,w;\lambda]}}}{p_{Q,w}(\lambda)}  f^T_{Q,w}(\lambda)
\end{equation}
where
\[
p(\lambda)=\prod_{\alpha\in \Delta_M} \sprod{\lambda-\mu_0}{\alpha^\vee},\ \ \ p_{Q,w}(\lambda) =\prod_{\varpi\in \hat\Delta_Q^\vee}\sprod{ \eta[Q,w;\lambda]}{\varpi}, \ \ \ \lambda\in\aaa_{M,\C}^* 
\]
and 
\[
f^T_{Q,w}(\lambda)=\int_{K_H} \int_{[L\cap \sigma_{Q,w}^{-1} H\sigma_{Q,w}]^{1,L}} \delta_{Q\cap \sigma_{Q,w}^{-1} H\sigma_{Q,w}}^{-1}(\ell) \Lambda^{\sigma_{Q,w}^{-1}T,Q}E_{Q_w}^Q(\ell\sigma_{Q,w}^{-1} k,M(w,\lambda)C_{P,P_w}\varphi,w\lambda)\ d\ell \ dk.
\]
We point out that in terms of the dependence on $T$, $\sprod{\sigma_{Q,w}^{-1}T}{ \eta[Q,w;\lambda]}$ depends only on $(\sigma_{Q,w}^{-1}T)_Q$ while $f^T_{Q,w}(\lambda)$ depends only on $(\sigma_{Q,w}^{-1}T)^Q$. 

The product of $p(\lambda)$ with the sum on the right hand side of \eqref{eq thm step} is holomorphic at $\lambda=\mu_0$ so that the limit exists, however, the individual summands need not be holomorphic at $\lambda=\mu_0$, the denominator $p_{Q,w}(\lambda)$ introduces new hyperplane singularities. 
We claim however that viewed as an iterated limit, it can be computed term-wise.

We define recursively
\[
\itlim{x_1\to y_1,\dots,x_m\to y_m}f(x_1,\dots,x_m)=\itlim{x_1\to y_1,\dots,x_{m-1}\to y_{m-1}}[\lim_{x_m\to y_m} f(x_1,\dots,x_m)]
\]
and claim that for any $G\ne Q\in\P_\std$ and $w\in {}_LW'_M$ the iterated limit
\begin{equation}\label{eq main itlim}
\itlim{\lambda_1\to m,\dots,\lambda_m\to 1} \frac{p(\lambda)}{p_{Q,w}(\lambda)}e^{\sprod{\sigma_{Q,w}^{-1}T}{ \eta[Q,w;\lambda]}}  f^T_{Q,w}(\lambda)
\end{equation}
converges. In fact we prove by induction on $m$ the following properties of \eqref{eq main itlim}:
\begin{enumerate}
\item\label{part converge} it converges; 
\item\label{part nonzero denom} it is an exponential polynomial in $T$ with no zero exponent unless $Q=P_{(1;N-1)}$ and $w$ is the identity;
\item\label{part w=e} if $Q=P_{(1;N-1)}$ and $w$ is the identity it is an exponential polynomial in $T$ whose coefficient of the zero exponent equals $\upsilon'\,\LL(\varphi)$ for a positive $\upsilon'$ dependent only on normalization of measures.
\end{enumerate}
Note that once these three properties are established, taking the iterated limit $\itlim{\lambda_1\to m,\dots,\lambda_m\to 1}$ through the sum over $Q$ and $w$ on the right hand side of \eqref{eq thm step} expresses the left hand side of \eqref{eq thm step} as an exponential polynomial in $T$ whose coefficient of the zero exponent, namely $\P_0(\E^m(\varphi))$, is the coefficient of the zero exponent of the term $Q=P_{(1;N-1)}$, $w=e$ (for which $(-1)^{\dim\aaa_Q^G-1}=1$), and the theorem follows with $\upsilon=v_{P_{(1;N-1)}}\upsilon'$. In particular, we may assume by induction that the theorem holds for $m-1$. The base of the induction is $m=0$, where $\Delta_M$ is empty, $p\equiv 1$, the period integral $\P_H$ converges on $\E^0(\pi)$ and the statement is the theorem of Ginzburg, Rallis and Soudry \cite[Theorem D]{MR1740991}; in this case the three properties above hold for every term of \eqref{eq thm step}, since there are no limits to take and $\eta[Q,w;\mu_0]\ne 0$ for every proper $Q$ by Lemma \ref{lem exp zeros} (whose statement and proof apply verbatim for $m=0$), so that each term is an exponential polynomial with no zero exponent by the argument in the first case below.

Based on the principle of \eqref{eq com res} we carry the iterated limit throughout the proof through truncation and convergent integrals. That is, if 
\[
\itlim{\lambda_j\to m+1-j,\dots,\lambda_m\to 1} \frac{p(\lambda)}{p_{Q,w}(\lambda)}E_{Q_w}^Q(M(w,\lambda)C_{P,P_w}\varphi,w\lambda)
\]
converges then denoting this iterated limit by $\E(\lambda_1,\dots,\lambda_{j-1})$, a meromorphic family in $\Aut_Q$, we have
\begin{multline*}
\itlim{\lambda_j\to m+1-j,\dots,\lambda_m\to 1} \frac{p(\lambda)}{p_{Q,w}(\lambda)}e^{\sprod{\sigma_{Q,w}^{-1}T}{ \eta[Q,w;\lambda]}}  f^T_{Q,w}(\lambda)=e^{\sprod{\sigma_{Q,w}^{-1}T}{ \eta[Q,w;(\lambda_1,\dots,\lambda_{j-1},m+1-j,\dots,1,0^{(2n)})]}}\\ \times  \int_{K_H} \int_{[L\cap \sigma_{Q,w}^{-1} H\sigma_{Q,w}]^{1,L}} \delta_{Q\cap \sigma_{Q,w}^{-1} H\sigma_{Q,w}}^{-1}(\ell) \Lambda^{\sigma_{Q,w}^{-1}T,Q}\E(\ell\sigma_{Q,w}^{-1} k,\lambda_1,\dots,\lambda_{j-1})\ d\ell\ dk.
\end{multline*}

Note that $p(\lambda)E_{Q_w}^Q(M(w,\lambda)C_{P,P_w}\varphi,w\lambda)$ is holomorphic at $\lambda=\mu_0$ by Corollary \ref{cor consterm hol}. If $p_{Q,w}(\mu_0)\ne 0$ it follows that  
\[
\frac{p(\lambda)}{p_{Q,w}(\lambda)}e^{\sprod{\sigma_{Q,w}^{-1}T}{ \eta[Q,w;\lambda]}}  f^T_{Q,w}(\lambda)
\]
is holomorphic at $\lambda=\mu_0$ and in particular \eqref{eq main itlim} equals
\[
\frac{e^{\sprod{\sigma_{Q,w}^{-1}T}{ \eta[Q,w;\mu_0]}} }{p_{Q,w}(\mu_0)} \lim_{\lambda\to \mu_0}p(\lambda) f^T_{Q,w}(\lambda).
\]
Since by assumption $\eta[Q,w;\mu_0]$ is a non-zero exponent for $(\sigma_{Q,w}^{-1}T)_Q$ while $f^T_{Q,w}(\lambda)$ is an exponential polynomial in $(\sigma_{Q,w}^{-1}T)^Q$ parts \eqref{part converge} and \eqref{part nonzero denom} follow in this case.

If $p_{Q,w}(\mu_0)= 0$ let $\ell\in [1,k]$ be maximal such that $\sprod{ \eta[Q,w;\mu_0]}{\varpi_{\nu_\ell}} =0$. Writing $w=w_\Data$ with $\Data=(D_1,E_1,\dots,D_k,E_k,D_0)$ and $i=i(\Data)$ we conclude from Lemma \ref{lem exp zeros} that either $i=0$ or $\ell<i$ and furthermore $[1,\nu_\ell]=\sqcup_{a=1}^\ell D_a$. 
In the notation of \S\ref{ss other res} we then have $w\alpha_{\nu_\ell}>0$ and $(w\alpha_{\nu_\ell})_Q\ne 0$. By Corollary \ref{cor consterm hol} we have
 \begin{equation}\label{eq more hol}
 \frac{p(\lambda)}{\sprod{\lambda-\mu_0}{\alpha_{\nu_\ell}^\vee}}E_{Q_w}^Q(M(w,\lambda)C_{P,P_w}\varphi,w\lambda)\text{ is holomorphic at }\lambda=\mu_0.
 \end{equation}
We further separate into two cases.

If $\nu_\ell>1$ then furthermore taking \eqref{eq fw} into consideration note that
\[
p_{Q,w}(\lambda_1,\dots,\lambda_{\nu_\ell-1},m+1-\nu_\ell,\dots,1,0^{(2n)})
\]
is a non-zero polynomial in $\lambda_1,\dots,\lambda_{\nu_\ell-1}$. 
Consequently, as a meromorphic function of $\lambda_1,\dots,\lambda_{\nu_\ell-1}$ we have
\begin{multline*}
\itlim{\lambda_{\nu_\ell}\to m+1-\nu_\ell,\dots,\lambda_m\to 1} \frac{p(\lambda)}{p_{Q,w}(\lambda)}E_{Q_w}^Q(M(w,\lambda)C_{P,P_w}\varphi,w\lambda)=\\ \frac{\itlim{\lambda_{\nu_\ell}\to m+1-\nu_\ell,\dots,\lambda_m\to 1} p(\lambda)E_{Q_w}^Q(M(w,\lambda)C_{P,P_w}\varphi,w\lambda)}{p_{Q,w}(\lambda_1,\dots,\lambda_{\nu_\ell-1},m+1-\nu_\ell,\dots,1,0^{(2n)})}=0
\end{multline*}
where the vanishing follows from \eqref{eq more hol}.
Consequently,  \eqref{eq main itlim} equals zero in this case.

If $\nu_\ell=1$ write $Q=P_{(1,n_2,\dots,n_k;r)}$ and let $Q'=L'\ltimes V'=P_{(n_2,\dots,n_k;r)}$ and $P'=M'\ltimes U'=P_{(1^{(m-1)};2n)}$ be parabolic subgroups of $\Sp_{N-1}$. Then $w$ has a representative of the form $\diag(1,\tilde w',1)$ where $\tilde w'$ represents in $\Sp_{N-1}$ a Weyl element $w'\in {}_{L'}(W_{\Sp_{N-1}})'_{M'}$. 

If $Q=P_{(1;N-1)}$ then $w'$ is the identity in $\Sp_{N-1}$ and therefore $w$ is the identity. In particular, $Q_w=P$. We then have 
\[
p_{Q,w}(\lambda)=\lambda_1-m\ \ \ \text{and}\ \ \ \frac{p(\lambda)}{p_{Q,w}(\lambda)}=\begin{cases}\frac{\lambda_1-\lambda_2-1}{\lambda_1-m}\,p'(\lambda_2,\dots,\lambda_m) & m\ge 2\\ 1 & m=1\end{cases}
\] 
where $p'(\lambda_2,\dots,\lambda_m)=(\lambda_m-1)\prod_{j=2}^{m-1}(\lambda_j-\lambda_{j+1}-1)$ is the polynomial $p$ attached to $m-1$ (with respect to the group $\Sp_{N-1}$). Furthermore, $E_P^Q(\varphi,\lambda)|_{Q(\A)^1 K}$ is independent of $\lambda_1$. Note further that $\sigma_{Q,w}$ is the identity, so that $f^T_{Q,w}(\lambda)$ is independent of $\lambda_1$ and $e^{\sprod{\sigma_{Q,w}^{-1}T}{ \eta[Q,w;\lambda]}}=e^{(\lambda_1-m)T_1}$, where $T_1$ is the first coordinate of $T$. Since $\frac{\lambda_1-\lambda_2-1}{\lambda_1-m}\to 1$ as $\lambda_2\to m-1$ (for $\lambda_1\ne m$), we conclude that \eqref{eq main itlim} equals
\[
\lim_{\lambda_1\to m}e^{(\lambda_1-m)T_1}\ \itlim{\lambda_2\to m-1,\dots,\lambda_m\to 1} p'(\lambda_2,\dots,\lambda_m)  f^T_{Q,w}(\lambda)=\itlim{\lambda_2\to m-1,\dots,\lambda_m\to 1} p'(\lambda_2,\dots,\lambda_m)  f^T_{Q,w}(\lambda)
\]
(for $m=1$ the iterated limit on the right hand side is over an empty set of variables and is to be read as $f^T_{Q,w}(\lambda)$). Here $L\cap H=\GL_1\times H'$ with $H'=H_{m-1+n,n}\subseteq \Sp_{N-1}$, the mixed truncation $\Lambda^{T,Q}$ on $L$ is the mixed truncation of $\Sp_{N-1}$ relative to $H'$ with respect to $T^Q$, and for $\ell k\in L(\A)^1K$ the function $E_P^Q(\ell k,\varphi,\lambda)$ is the Eisenstein series on $\Sp_{N-1}$ attached to $P'$, an element of $I^{m-1}(\pi)$ (defined with respect to $\Sp_{N-1}$) depending on $k$, and the parameter $(\lambda_2,\dots,\lambda_m,0^{(2n)})$. By \eqref{eq com res} and \eqref{eq new real} for $\Sp_{N-1}$, the right hand side is therefore a positive constant multiple (depending only on normalization of measures) of the integral over $K_H$ of the truncated period $\int_{[H']}\Lambda^{T^Q}\E^{m-1}(\varphi_k)$, where $\varphi_k\in I^{m-1}(\pi)$ is the element described above. By the $m-1$ case of the theorem, this is an exponential polynomial in $T$ whose coefficient of the zero exponent is $\upsilon'\,\LL(\varphi)$ for a positive constant $\upsilon'$ depending only on normalization of measures (note that the other exponents of the lower rank truncated period are non-zero exponents in $T^Q$ and remain non-zero exponents in $T$). This proves part \eqref{part w=e} in this case (parts \eqref{part converge} and \eqref{part nonzero denom} being void).

Otherwise, $Q'$ is a proper parabolic subgroup of $\Sp_{N-1}$ and we may use the induction hypothesis. Set $p_1(\lambda)=\frac{p(\lambda)}{\lambda_1-\lambda_2-1}$ (for $m=1$ set $p_1\equiv 1$) and $p_2(\lambda)=\frac{p_{Q,w}(\lambda)}{\lambda_1-m}$ two polynomials in $\lambda$ (note that $p_1$ is independent of $\lambda_1$ and is the polynomial $p$ attached to $m-1$, while $p_2$ is the polynomial $p_{Q',w'}$ attached to $m-1$, $Q'$ and $w'$). By maximality of $\ell$ 
the polynomial $p_2(\lambda)$ does not vanish at $\lambda=\mu_0$ and therefore
 $p_{Q',w'}(m-1,\dots,1;0^{(2n)})\ne 0$ so that for $\Sp_{N-1}$ we are in the first case considered.
By the induction hypothesis,
\[
\itlim{\lambda_2\to m-1,\dots,\lambda_m\to 1} \frac{p_1(\lambda)}{p_2(\lambda)}e^{\sprod{\sigma_{Q,w}^{-1}T}{ \eta[Q,w;\lambda]}}  f^T_{Q,w}(\lambda)
\]
converges and defines an exponential polynomial in $T$ with no zero exponent and it follows from \eqref{eq more hol} that it is holomorphic at $\lambda_1=m$ (as a meromorphic function of $\lambda_1$ with coefficients that are exponential polynomials in $T$; note that $e^{\sprod{\sigma_{Q,w}^{-1}T}{ \eta[Q,w;\lambda]}}$ is the product of $e^{(\lambda_1-m)\sprod{\sigma_{Q,w}^{-1}T}{\varpi_1}}$ with the corresponding exponential attached to $m-1$, $Q'$ and $w'$, and the first factor is holomorphic and equal to one at $\lambda_1=m$). Its value at $\lambda_1=m$ equals \eqref{eq main itlim} and this concludes parts  \eqref{part converge} and \eqref{part nonzero denom} in this case.
The theorem follows.

\end{proof}

\def\cprime{$'$} \def\Dbar{\leavevmode\lower.6ex\hbox to 0pt{\hskip-.23ex
  \accent"16\hss}D} \def\cftil#1{\ifmmode\setbox7\hbox{$\accent"5E#1$}\else
  \setbox7\hbox{\accent"5E#1}\penalty 10000\relax\fi\raise 1\ht7
  \hbox{\lower1.15ex\hbox to 1\wd7{\hss\accent"7E\hss}}\penalty 10000
  \hskip-1\wd7\penalty 10000\box7}
  \def\polhk#1{\setbox0=\hbox{#1}{\ooalign{\hidewidth
  \lower1.5ex\hbox{`}\hidewidth\crcr\unhbox0}}} \def\dbar{\leavevmode\hbox to
  0pt{\hskip.2ex \accent"16\hss}d}
  \def\cfac#1{\ifmmode\setbox7\hbox{$\accent"5E#1$}\else
  \setbox7\hbox{\accent"5E#1}\penalty 10000\relax\fi\raise 1\ht7
  \hbox{\lower1.15ex\hbox to 1\wd7{\hss\accent"13\hss}}\penalty 10000
  \hskip-1\wd7\penalty 10000\box7}
  \def\ocirc#1{\ifmmode\setbox0=\hbox{$#1$}\dimen0=\ht0 \advance\dimen0
  by1pt\rlap{\hbox to\wd0{\hss\raise\dimen0
  \hbox{\hskip.2em$\scriptscriptstyle\circ$}\hss}}#1\else {\accent"17 #1}\fi}
  \def\bud{$''$} \def\cfudot#1{\ifmmode\setbox7\hbox{$\accent"5E#1$}\else
  \setbox7\hbox{\accent"5E#1}\penalty 10000\relax\fi\raise 1\ht7
  \hbox{\raise.1ex\hbox to 1\wd7{\hss.\hss}}\penalty 10000 \hskip-1\wd7\penalty
  10000\box7} \def\lfhook#1{\setbox0=\hbox{#1}{\ooalign{\hidewidth
  \lower1.5ex\hbox{'}\hidewidth\crcr\unhbox0}}}
\providecommand{\bysame}{\leavevmode\hbox to3em{\hrulefill}\thinspace}
\providecommand{\MR}{\relax\ifhmode\unskip\space\fi MR }
\providecommand{\MRhref}[2]{%
  \href{http://www.ams.org/mathscinet-getitem?mr=#1}{#2}
}
\providecommand{\href}[2]{#2}


\end{document}